\documentclass[11pt]{article}
\usepackage{amssymb}
\usepackage{mathrsfs}
\usepackage{amsfonts}
\usepackage{amsthm}
\usepackage{amsmath,amscd}
\usepackage{amsxtra}     
\usepackage{bm}
\usepackage[all,cmtip]{xy}
\usepackage{epsfig}
\usepackage{tikz-cd}
\usepackage{verbatim}
\usepackage{color}
\usepackage{commath}
\usepackage{mathtools}
\usepackage{cite}
\usepackage{extpfeil}
\usepackage{hyperref}
\usepackage{enumerate}
\usepackage{extarrows}
\usepackage{bbold}
\usepackage{xfrac}
\usepackage[nice]{nicefrac}
\usepackage{tikz}
\usetikzlibrary{arrows,matrix,shapes,trees}
\usepackage[capitalise]{cleveref}
\usepackage{blindtext}
\usepackage{geometry}
\theoremstyle{plain}
 \newtheorem{thm}{Theorem}[subsection]
 \newtheorem{prop}[thm]{Proposition}
 \newtheorem{lem}[thm]{Lemma}
 \newtheorem{cor}[thm]{Corollary}

 \newtheorem{lem'}{``Lemma''}
 
 \newtheorem*{claim}{Claim}
\theoremstyle{definition}
 \newtheorem{ex}[thm]{Example}
 \newtheorem{defn}[thm]{Definition}
\theoremstyle{remark}
 \newtheorem{rmk}[thm]{Remark}
 \newtheorem{constr}[thm]{Construction}
 \numberwithin{equation}{section}

\newcommand{\N}{{\mathbb N}}
\newcommand{\Q}{{\mathbb Q}}
\newcommand{\Z}{{\mathbb Z}}

\newcommand{\F}{{\mathbb F}}
\newcommand{\prism}{{\mathbb{\Delta}}}

\newcommand{\MF}{\mathcal{M}\mathcal{F}}
\newcommand{\MFa}{\mathcal{M}\mathcal{F}_{\mr{[0, a]}}}
\newcommand{\mr}{\mathrm}

\newcommand{\DPa}{DivCR^{\wedge, \phi}_{[0, a]}}

\newcommand{\overbar}[1]{\mkern 3.0mu\overline{\mkern-3.0mu#1\mkern-2.0mu}\mkern 2.0mu}

\author{\bfseries Matti W\"urthen}
\title{Divided prismatic Frobenius crystals of small height and the category $\mathcal{M}\mathcal{F}$}

\begin{document}
\maketitle
\begin{abstract}
Let $\mathcal{X}$ be a smooth $p$-adic formal scheme over a mixed characteristic complete discrete valuation ring $\mathcal{O}_{K}$ with perfect residue field. We introduce a general category $\MF_{\mr{[0, p-2]}}^{\mr{tor-free}}(\mathcal{X})$ of $p$-torsion free crystalline coefficient objects and show that this category is equivalent to the category of completed prismatic Frobenius crystals of height $p-2$, recently introduced by Du-Liu-Moon-Shimizu. In particular this shows that the category $\MF^{\mr{tor-free}}_{\mr{[0, p-2]}}(\mathcal{X})$ is equivalent to the category of crystalline $\Z_p$-local systems on $\mathcal{X}$ with Hodge-Tate weights in $\{0,\ldots , p-2\}$, which generalizes the crystalline part of a theorem of Breuil-Liu to higher dimensions. 
\end{abstract}

\tableofcontents
\section{Introduction}
Let $K$ be a discretely valued $p$-adic field with perfect residue field. Recently, in \cite{BS-pris-cris}, Bhatt and Scholze proved that $\Z_p$-lattices in crystalline $\Q_p$-representation of the absolute Galois group $G_{K}$ correspond to locally free Frobenius crystals on the absolute prismatic site of $Spf(\mathcal{O}_{K})$. This improves Kisin's classification from \cite{Kis-cris} via so-called Breuil-Kisin modules.\\
This correspondence has now been extended to higher dimensions by the work of Du-Liu-Moon-Shimizu (\cite{DLMS}) and independently by Guo-Reinecke (\cite{GR}). More precisely, let $\mathcal{X}$ be a smooth $p$-adic formal scheme over $\mathcal{O}_{K}$. In \cite{DLMS} it is then shown that lattices in crystalline local systems with positive Hodge-Tate weights of the generic fiber $\mathcal{X}_{K}$ correspond to so called completed prismatic Frobenius crystals of finite height \footnote{We remark already here that we use a different convention than \cite{DLMS} for the \'etale realization functor. In particular we use the covariant version of the \'etale realization functor.}. Recall from \cite{DLMS} that a completed prismatic Frobenius crystal of height $a\geq 0$ is given by a pair $(\mathcal{M}, \Phi_{\mathcal{M}})$, consisting of a crystal in complete modules over $\mathcal{O}_{\prism}$, the structure sheaf of the absolute prismatic site of $\mathcal{X}$, together with a Frobenius map, satisfying the following condition: For any small affine open $Spf(R)\subset \mathcal{X}$ the evaluation at $\mathfrak{S}(R)$, a chosen versal deformation of $R$ along $\mathfrak{S}\to \mathcal{O}_{K}$, is a saturated torsion-free Breuil-Kisin module of height $a$. In particular, this means that the Frobenius is injective with $E^{a}$-torsion cokernel (see \cref{completed-crystals-pris-site} for the precise definition).\\
For representations with Hodge-Tate weights in $\{0, \ldots, p-2\}$, there exists another line of work, which uses coefficient objects that are closer to crystals on the crystalline site. Historically, coefficient objects for lattices in crystalline representations were first studied by Fontaine-Laffaille (\cite{Font-Laff}) - this work was then extended to higher dimensions by Faltings in \cite{Fal-cris}. In both cases one was still restricted to the case where $K$ is totally unramified. The ramified case was then treated by Faltings in \cite{Fal-very} and in particular for the zero-dimensional case by Breuil in a series of works (see \cite{Breuil-constr}, \cite{Breuil-mod-fort} and \cite{Breuil-grp})\footnote{In fact, Breuil's work more generally treats the case of lattices in semistable representations}. Breuil conjectured that lattices in crystalline (or more generally semistable) representations with small Hodge-Tate weights should correspond to so called (crystalline) Breuil-modules. This conjecture was proven by Liu in \cite{liu-conj}, using the work of Kisin. One of the main goals of this paper is to prove a higher dimensional version of (the crystalline part of) this theorem. For any smooth $p$-adic formal scheme $\mathcal{X}$, we introduce a category of higher dimensional analogues of crystalline Breuil-modules $\MFa^{\mr{tor-free}}(\mathcal{X})$, where now we assume $a\leq p-2$, whose objects in the local case $\mathcal{X}=Spf(R)$ are given by quadruples $(M, Fil^{a}M, \nabla, \Phi^{a}_{M})$ over $D$, which is the PD-envelope of $ker(\mathfrak{S}(R)\to R)$ in the versal deformation $\mathfrak{S}(R)$. Here $M$ is a finitely generated torsion-free module, which is equipped with a flat connection, a filtration $Fil^{a}M\subset M$ and a divided Frobenius map $\Phi^{a}_{M}:Fil^{a}M\to M$.\\
A completed prismatic Frobenius crystal of height $a$ in the sense of Du-Liu-Moon-Shimizu is filtered by the so called Nygaard filtration, defined by $\mathcal{N}_{\mathcal{M}}^{i}:=\Phi^{-1}_{\mathcal{M}}(\mathcal{I}_{\prism}^{i}\mathcal{M})$.
We use the $a$-th step of this Nygaard filtration to construct an equivalence between completed prismatic Frobenius crystals and a category of filtered crystals equipped with a divided Frobenius (see \cref{defn-divcrys}), which we denote by $DivCR^{\wedge, \phi, \mr{tor-free}}_{\mr{[0, a]}}(\mathcal{X})$. It is this latter category, which from our point of view is most naturally associated with the category $\MF$. Our main theorem is then the following. 
\begin{thm}[\cref{main-thm}]
There is a canonical equivalence of categories
\[DivCR^{\wedge, \phi, \mr{tor-free}}_{\mr{[0, a]}}(\mathcal{X})\to \MFa^{\mr{tor-free}}(\mathcal{X}).\]
\end{thm}
The proof is given by passing to the local case and then comparing the descent data on both sides with respect to a perfectoid covering of $Spf(R)$. There the proof boils down to a general comparison theorem between divided Frobenius modules over $A_{inf}$ and $A_{cris}$ respectively, which generalizes various similar results in the literature (see for example \cite[Prop. 1.1.11]{Kisin-moduli}, \cite[Thm. 2.2.1]{CarLiu}, \cite[\S 2]{Fal-MF-st} and also \cite[2.3.2]{Cais-Lau} and \cite[Prop. 9.3]{lau-perf} for the case of Dieudonn\'e-modules).\\
As a corollary, using the work \cite{DLMS}, we then obtain that $\MFa^{\mr{tor-free}}(\mathcal{X})$ is equivalent to crystalline $\Z_p$-local systems of small Hodge-Tate weights.
\begin{cor}
Let $a\leq p-2$. There is an equivalence of categories
\[\MFa^{\mr{tor-free}}(\mathcal{X})\to \Z_p-Loc_{\mr{[0, a]}}^{\mr{cris}}(\mathcal{X}_{K}),\]
where $\Z_p-Loc_{\mr{[0, a]}}^{\mr{cris}}(\mathcal{X}_{K})$ denotes the category of crystalline $\Z_p$-local systems with Hodge-Tate weights in $\{0, \ldots , a\}$.
\end{cor}
We remark here that the works of Fontaine-Laffaille, Faltings and Breuil study a contravariant \'etale realization functor, whereas our functor is covariant.\\
The construction of the category $\MF$ is given in $\S 3$. In section $4$ we then first study the Nygaard filtration of completed prismatic Frobenius crystals. Building on this, we then introduce the category $DivCR^{\wedge, \phi, \mr{tor-free}}(\mathcal{X})$, which corresponds to crystalline local systems with positive Hodge-Tate weights. Section $5$ then contains the main results.
\section*{Acknowledgements}
I thank Johannes Ansch\"utz and H\'el\`ene Esnault for discussions.\\
This work was supported by Deutsche Forschungsgemeinschaft (DFG, German Research Foundation), TRR 326 \textit{Geometry and Arithmetic of Uniformized Structures} , project number 444845124, and by Deutsche Forschungsgemeinschaft (DFG, German Research Foundation), Sachbeihilfe \textit{Vector bundles and local systems on non-Archimedean analytic spaces}, project number 446443754.

\section{Preliminaries}
\subsection{Notation}
We fix a prime number $p > 2$ and denote by $\mathcal{O}_{K}$ a complete discrete valuation ring of mixed characteristic $(0, p)$ with perfect residue field $k$ and fraction field $K$. We also fix a uniformizer $\pi\in \mathcal{O}_{K}$ and denote by $(\mathfrak{S}=W(k)[[u]], E(u))$ the Breuil-Kisin prism associated to $\pi$. We denote by $\mathcal{S}$ the associated Breuil ring, i.e. the $p$-adically completed PD-envelope of the ideal $(E(u))\subset \mathfrak{S}(R)$ (which is compatible with the PD-structure on $p\mathfrak{S}$). We also define $c:=\frac{\phi(E)}{p}$, which is a unit in $\mathcal{S}$.\\
Moreover, we denote by $C:=\hat{\bar{K}}$ the $p$-adic completion of a fixed algebraic closure of $K$. We then have Fontaine's ring $A_{\mr{inf}}=W(\mathcal{O}_{C}^{\flat})=\prism_{\mathcal{O}_{C}}$ and define $\tilde{\xi}=\phi(\xi)$, where $\xi$ is a generator of the kernel of $\theta:A_{\mr{inf}}\to \mathcal{O}_{C}$. We will follow the convention to use $(A_{\mr{inf}}, \tilde{\xi})$ as the universal prism over $\mathcal{O}_{C}$. Moreover, we recall the ring $A_{cris}=\prism_{\mathcal{O}_{C}/p}$, which is identified as the $p$-adic completion of the PD-envelope of $(\xi)\subset A_{inf}$. There exists a faithfully flat map of prisms $\mathfrak{S}\to A_{inf}$, which extends to a map $\mathcal{S}\to A_{cris}$.\\
Lastly, by smooth $p$-adic formal scheme over $\mathcal{O}_{K}$ we mean a formal scheme $\mathcal{X}$, which is locally isomorphic to $Spf(R)$, where $R$ is the $p$-adic completion of a smooth $\mathcal{O}_{K}$-algebra. We write $\mathcal{X}_{K}$ for the adic space generic fiber of $\mathcal{X}$.

\subsection{Preliminaries on the prismatic site}\label{sec-prel-pris}
Let $\mathcal{X}$ be a smooth $p$-adic formal scheme over $\mathcal{O}_{K}$. In this section we wish to compare crystals on the absolute prismatic site with (prismatic) crystals on the quasisyntomic site of $\mathcal{X}$, generalizing \cite[Prop. 4.4]{ALB}. For the general terminology and main results about the prismatic site, we refer the reader to the foundational works \cite{BS-pris}, \cite{BS-pris-cris} and to \cite{ALB}. Since we wish to connect our work to the paper \cite{DLMS}, we will use the convention from loc. cit. and denote by $\mathcal{X}_{\prism}$ the site of bounded prisms (see \cite[Def. 3.2]{DLMS}). We then recall the following definition.
\begin{defn}\cite[Def. 3.11]{DLMS}
A completed crystal in $\mathcal{O}_{\prism}$-modules is a sheaf of $\mathcal{O}_{\prism}$-modules $\mathcal{M}$, such that, for any $(B, I)\in \mathcal{X}_{\prism}$, $\mathcal{M}(B)$ is a finitely generated, $(p, I)$-complete $B$-module and for any map $(B, I)\to (C, IC)$ in $\mathcal{X}_{\prism}$, the canonical map $\mathcal{M}(B)\hat{\otimes} C\to \mathcal{M}(C)$ is an isomorphism. We denote the category of completed prismatic crystals by $Crys(\mathcal{X}_{\prism})$.
\end{defn}
Following \cite{ALB}, we will mostly work with the quasi-syntomic site, which will turn out to be more convenient for us. Let $\mathcal{X}_{\mr{qsyn}}$ be the site consisting of $p$-complete rings $A$, equipped with a quasi-syntomic map $Spf(A)\to \mathcal{X}$ and where coverings are $p$-completely faithfully flat maps. Recall from \cite[\S 4]{ALB} that there is a morphism of ringed  topoi $\mu:\tilde{\mathcal{X}}_{\prism}\to \tilde{\mathcal{X}}_{\mr{qsyn}}$, where $\tilde{\mathcal{X}}_{\mr{qsyn}}$ is endowed with the sheaf of rings $\mathcal{O}^{\mr{pris}}:=\mu_{*}\mathcal{O}_{\prism}$. Moreover, we have the sheaf of ideals $\mathcal{I}=\mu_{*}\mathcal{I}_{\prism}$. The sheaf $\mathcal{O}^{pris}$ is endowed with the Nygaard filtration, which is an exhaustive and decreasing filtration, defined by 
\begin{center}
$\mathcal{N}^{i}_{\prism}:=\phi^{-1}(\mathcal{I}^{i})$, for any $i\geq 0$.
\end{center}
The quasi-syntomic site has a basis, given by the objects $Spf(B)\to \mathcal{X}$, for which $B$ is a $p$-torsion free quasi-regular semiperfectoid (qrsp) $\Z_p$-algebra. We refer the reader to \cite[\S 4.4]{BMS-hoch} and to \cite[\S 3.3 and \S 3.4]{ALB} for the definition and basic properties of such rings. The prismatic site of $B$ has an initial object, given by prismatic cohomology $\prism_{B}$ and we have $\mathcal{O}^{\mr{pris}}(B)=\prism_{B}$. Moreover, for any $n\geq 1$, the rings $B^{\otimes n}=\underbrace{B\hat{\otimes}_{R}\cdots \hat{\otimes}_{R}B}_{n-times}$ are qrsp as well (\cite[Lemma 4.27]{BMS-hoch}).
\begin{defn}
A crystal in $\mathcal{O}^{\mr{pris}}$-modules is a sheaf of $\mathcal{O}^{\mr{pris}}$-module 
$\mathcal{M}$ on $\mathcal{X}_{\mr{qsyn}}$, such that for any qrsp object $Spf(B)\in \mathcal{X}_{\mr{qsyn}}$ the $\prism_{B}$-module $\mathcal{M}(B)$ is finitely generated and classically $(I, p)$-complete and for any map $Spf(C)\to Spf(B)$ of qrsp objects in $\mathcal{X}_{\mr{qsyn}}$, the canonical map $\mathcal{M}(B)\hat{\otimes}_{\prism_{B}}\prism_{C}\to \mathcal{M}(C)$ is an isomorphism.\\
We denote the category of $\mathcal{O}^{\mr{pris}}$-crystals by $Crys(\mathcal{X}_{\mr{qsyn}})$.
\end{defn}
Using the map of ringed topoi $\mu:\tilde{\mathcal{X}}_{\prism}\to \tilde{\mathcal{X}}_{\mr{qsyn}}$, we get a functor $\mu_{*}:Crys(\mathcal{X}_{\prism})\to Crys(\mathcal{X}_{\mr{qsyn}})$.
\begin{prop}\label{equiv-qsyn}
The functor $\mu_{*}:Crys(\mathcal{X}_{\prism})\to Crys(\mathcal{X}_{\mr{qsyn}})$ is an equivalence of categories.
\end{prop}
The proof of the proposition will follow from the fact that Zariski-locally on $\mathcal{X}$ both types of crystals are equivalent to the same type of descent data. Assume that $\mathcal{X}=Spf(R)$ is affine. Let $R\to \tilde{R}$ be a $p$-completely faithfully flat covering by a qrsp $\Z_p$-algebra. Then by \cite[Lemma 7.11]{BS-pris}, $\prism_{\tilde{R}}$ is a weakly initial object of $\mathcal{X}_{\prism}$. For $n\geq 1$, we define $\prism_{\tilde{R}}^{(n)}:=\prism_{\tilde{R}^{\otimes n+1}}$ and we denote by 
$\xymatrix@C-=0.5cm{
\prism_{\tilde{R}}\ar@<.5ex>[r]^{p_{1}}\ar@<-.5ex>[r]_{p_{2}}  & \prism_{\tilde{R}}^{(1)}}$ the morphisms induced by sending $\tilde{R}$ to either the first or to the second factor. 
\begin{defn}
The category $Strat(\prism_{\tilde{R}}^{(\bullet)})$ consists of pairs $(M, \epsilon)$, where $M$ is a finitely generated classically $(p, I)$-complete $\prism_{\tilde{R}}$-module and $\epsilon:p_{1}^{*}M\xrightarrow{\cong}p_{2}^{*}M$ is an isomorphism, which satisfies the obvious cocycle condition over $\prism_{\tilde{R}}^{(2)}$.
\end{defn}
The crystal property ensures that we have evaluation functors $Crys(\mathcal{X}_{\prism})\to Strat(\prism_{\tilde{R}}^{(\bullet)})$ and $Crys(\mathcal{X}_{\mr{qsyn}})\to Strat(\prism_{\tilde{R}}^{(\bullet)})$ that fit into a commutative diagram
\begin{center}
$\xymatrix{
Crys(\mathcal{X}_{\prism})\ar[dr]\ar[rr]^{\mu_{*}} && Crys(\mathcal{X}_{\mr{qsyn}})\ar[dl]\\
& Strat(\prism_{\tilde{R}}^{(\bullet)}) &
}$
\end{center}

For the proof of the following proposition, note that one has effective descent for finitely generated modules along $(I, p)$-completely faithfully flat maps.
\begin{prop}
Both evaluation functors $Crys(\mathcal{X}_{\prism})\to Strat(\prism_{\tilde{R}}^{(\bullet)})$ and $Crys(\mathcal{X}_{\mr{qsyn}})\to Strat(\prism_{\tilde{R}}^{(\bullet)})$ are equivalences of categories.
\end{prop}
\begin{constr}\label{coverings}
We assume that $\mathcal{X}=Spf(R)$ is small affine. We will recall the construction of two different important coverings that will be used throughout the paper. The first construction is similar to \cite[Constr. 4.17]{BS-pris}. Let $P_{R}$ be the derived $(p, \xi)$-completed polynomial algebra over $A_{\mr{inf}}$ on the set $R_{\mathcal{O}_{C}}=R\hat{\otimes}_{\mathcal{O}_{K}}\mathcal{O}_{C}$. Let $B_{R}$ be the $(p, \xi)$-completion of the free $\delta$-$A_{\mr{inf}}$-algebra on the set $P_{R}$ and let $B_{R}^{\infty}=(colim_{\phi}B_{R})^{\wedge}$ be its derived $(p, \xi)$-completed perfection. Then $\tilde{R}:=B_{R}^{\infty}\hat{\otimes}_{P_{R}}R_{\mathcal{O}_{C}}$ is quasi-regular semiperfectoid and $R\to \tilde{R}$ is quasi-syntomic. The covering $\tilde{R}$ is functorial in $R$.\\
As a second covering, we recall the usual perfectoid covering, obtained by extracting $p$-th power roots of \'etale coordinates. Let $\mathcal{O}_{K}\{t_{1}^{\pm}, \ldots , t_{n}^{\pm}\}\to R$ be \'etale. We then define $R_{\infty}:=R\hat{\otimes}_{\mathcal{O}_{K}\{t_{1}^{\pm}, \ldots , t_{n}^{\pm}\}} \mathcal{O}_{C}\{t_{1}^{\pm p^{\frac{1}{\infty}}}, \ldots , t_{n}^{\pm p^{\frac{1}{\infty}}}\}$. Then $R_{\infty}$ is a perfectoid ring and $R\to R_{\infty}$ is quasi-syntomic.\\
Sometimes it will be necessary to use smaller weakly initial objects. Recall from \cite[\S 2]{DLMS} that there is a small $W(k)$-algebra $R_{0}$, such that $R=R_{0}\hat{\otimes}_{W(k)}\mathcal{O}_{K}$. Then we define the $\mathfrak{S}$-algebra $\mathfrak{S}(R):=R_{0}[[u]]$. It carries a Frobenius lift, compatible with the one on $\mathfrak{S}$, which is obtained by raising the \'etale coordinates to their $p$-th power. $(\mathfrak{S}(R), E(u))$ is then a prism, such that $\mathfrak{S}(R)/E(u)=R$. It is a weakly initial object of $\mathcal{X}_{\prism}$. There is a faithfully flat morphism of prisms $\mathfrak{S}(R)\to A_{\mr{inf}}(R_{\infty})$, compatible with the map $\mathfrak{S}\to A_{\mr{inf}}$ that takes $E(u)$ to $\tilde{\xi}$. It is defined by sending a coordinate $t_{i}$ to $[t_{i}^{\flat}]$, the Teichm\"uller lift of the element $t_{i}^{\flat}\in R_{\infty}^{\flat}$ that corresponds to a fixed sequence of $p$-th power roots of $t_{i}\in R_{\infty}$.
\end{constr}

\subsection{Preliminaries on the crystalline site}\label{sec-prel-cris}
For general results on the crystalline site, we refer to \cite{berthelot} - see also \cite[App. F]{BL}. 
\begin{defn}
Let $Y$ be a scheme on which $p$ is nilpotent. Then $(Y/\Z_p)_{\mr{cris}}$ denotes the site whose objects are given by triples $(A, J, f)$, where $A$ is a classically $p$-adically complete $\Z_p$-algebra, $J\subset A$ is an ideal, equipped with a divided power structure that is compatible with the PD-structure on $pA$ and $f:Spec(A/J)\to Y$ is a morphism.\\
The topology of $(Y/\Z_p)_{\mr{cris}}$ is generated by maps $(A, J, f)\to (B, JB, g)$, where $A\to B$ is $p$-completely faithfully flat.\\
Let $\mathcal{O}_{Y_{\mr{cris}}}$ be the sheaf associated to $(A, J, f)\mapsto A$. Then a complete $\mathcal{O}_{Y_{\mr{cris}}}$-module $\mathcal{M}$ is called a crystal in $\mathcal{O}_{Y_{\mr{cris}}}$-modules, if for any map $(A, J_{A}, f_{A})\to (B, J_{B}, f_{B})$ in $(Y/\Z_p)_{\mr{cris}}$ the canonical map $\mathcal{M}(A)\hat{\otimes}_{A}B\to \mathcal{M}(B)$ is an isomorphism.
\end{defn}
Let $\mathcal{X}=Spf(R)$ be a small smooth affine $p$-adic formal scheme over $\mathcal{O}_{K}$. As above, we fix a uniformizer $\pi\in \mathcal{O}_{K}$ and let again $(\mathfrak{S}, E)$ be the Breuil-Kisin prism associated to $\pi$ and consider the object $\mathfrak{S}(R)$ from \cref{coverings}. We denote by $D$ the $p$-adic completion of the PD-envelope (compatible with the PD-structure on $p\mathfrak{S}(R)$) of the kernel of the map $\mathfrak{S}(R)\to R$ and denote by $J_{R}$ the kernel of $D\to R$.\\
We write $\mathcal{X}_{n}:=Spec(R/p^{n})$. For any object $(A, J, f)\in (Spec(R/p^{n})/\Z_p)_{\mr{cris}}$, the map $R/p^{n}\to A/J$ lifts to a PD-map $D\to A$: namely, first note that $A$ is a $W(k)$-algebra, by deformation theory (since $A/J$ is such). Then $R/p^{n}\to A/J$ may be lifted to a map $\mathfrak{S}(R)\to A$, which then extends to the PD-hull $D$. This means that $(D\to R/p^{n}, (J_{R}, p^{n}))$ is weakly initial.\\
We then also have the following description of self-products in $(Spec(R/p^{n})/\Z_p)_{\mr{cris}}$. For $n\geq 1$, denote by $D(n)$ the $p$-adically completed PD-hull of the kernel of 
\[\underbrace{\mathfrak{S}(R)\otimes_{W(k)} \cdots \otimes_{W(k)} \mathfrak{S}(R)}_{n+1-times}\to \mathfrak{S}(R)\to R.\] 
This gives rise to a cosimplicial ring with PD-ideal $(D^{(\bullet)}, J_{R}(\bullet))$. We also denote by $p_{1}$ (resp. by $p_{2}$) the map $D\to D^{(1)}$ induced by $\mathfrak{S}(R)\to \mathfrak{S}(R)\otimes \mathfrak{S}(R)$ mapping to the first (resp. to the second) factor. \\
We will also again make use of weakly initial objects that are obtained by passing to coverings by qrsp rings. Let $R\to \tilde{R}$ be the functorial quasi-syntomic covering constructed in \cref{coverings}.

For $n\geq 0$, we define $\mathbb{A}_{\mr{cris}}(\tilde{R}/p)^{(n)}:=\prism_{\tilde{R}^{\otimes (n+1)}/p}$, where $\tilde{R}^{\otimes (n+1)}:=\underbrace{\tilde{R}\hat{\otimes}_{R}\cdots \hat{\otimes}_{R}\tilde{R}}_{(n+1)-times}$. This is the $p$-adically completed PD-envelope of $W(((\tilde{R}/p)^{\otimes (n+1})^{\flat})\twoheadrightarrow (\tilde{R}/p)^{\otimes (n+1)}$, where $((\tilde{R}/p)^{\otimes (n+1})^{\flat}$ denotes the inverse limit perfection of $(\tilde{R}/p)^{\otimes (n+1)}$ (see \cite[App. F]{BL}). Since $p\tilde{R}^{\otimes (n+1)}\subset \tilde{R}^{\otimes (n+1)}$ carries the canonical PD-structure (relative to $\Z_p$), there is a unique surjective morphism $\mathbb{A}_{\mr{cris}}(\tilde{R}/p)^{(n)}\to \tilde{R}^{\otimes (n+1)}$ compatible with the projection to $\tilde{R}^{\otimes (n+1)}/p$. In this way we get a cosimplicial algebra with PD-ideal $(\mathbb{A}_{\mr{cris}}(\tilde{R}/p)^{(\bullet)}, J_{\tilde{R}^{\bullet}})$.\\  Moreover, $(\mathbb{A}_{\mr{cris}}(\tilde{R}/p), (J_{\tilde{R}}, p^{n}))$ is a weakly inital object of $((\mathcal{X}/p^{n})/\Z_p)_{\mr{cris}}$, i.e. for any object $(A, I)\in ((\mathcal{X}/p^{n})/\Z_p)_{\mr{cris}}$ there exists a $p$-completely faithfully flat map $(A, I)\to (C, IC)$ such that there is a morphism $(\mathbb{A}_{\mr{cris}}(\tilde{R}/p), J_{\tilde{R}})\to (C, IC)$. To see this, first note that $A$ carries the structure of an $\mathcal{S}$-algebra, compatible with the $\mathcal{O}_{K}$-algebra structure of $R$. Now, the ring $B_{R}$ from \cref{coverings} is the $(p, \xi)$-completion of a polynomial ring over $A_{inf}$. Therefore, there is a map $B_{R}\to A\hat{\otimes}_{\mathcal{S}}A_{inf}$, lifting $R_{\mathcal{O}_{C}}\to A/I\otimes \mathcal{O}_{C}$. Then $(A_{A_{inf}}\hat{\otimes}_{B_{R}}B_{R}^{\infty}, I(A_{A_{inf}}\hat{\otimes}_{B_{R}}B_{R}^{\infty}))$ is a $\Z_p$-algebra with PD-ideal (completely faithfully flat over $A$), such that there is a map $\tilde{R}\to A_{A_{inf}}\hat{\otimes}_{B_{R}}B_{R}^{\infty}/I(A_{A_{inf}}\hat{\otimes}_{B_{R}}B_{R}^{\infty})$. It then follows from the universality of $\mathbb{A}_{\mr{cris}}(\tilde{R}/p)$, that there is a map $\mathbb{A}_{\mr{cris}}(\tilde{R}/p)\to A_{A_{inf}}\hat{\otimes}_{B_{R}}B_{R}^{\infty}$, lifting the map $R\to \tilde{R}\to A_{A_{inf}}/I\hat{\otimes}_{B_{R}}B_{R}^{\infty}$.\\
We also remark here already that 
\[\mathbb{A}_{\mr{cris}}(\tilde{R}/p)^{(n)}=\prism_{\tilde{R}^{\otimes n+1}/p}=\prism_{\tilde{R}^{\otimes n+1}}\hat{\otimes}_{A_{\mr{inf}}}A_{\mr{cris}},\]
which follows from the compatibility of prismatic cohomology with base-change. \\[0.5cm]

Let now $R_{\infty}$ be the perfectoid cover introduced in \cref{coverings}. We again obtain a cosimplicial object $\mathbb{A}_{\mr{cris}}(R_{\infty}/p)^{(\bullet)}$, where again $\mathbb{A}_{\mr{cris}}(R_{\infty}/p)$ is weakly initial. There is a $p$-completely faithfully flat map $D^{(\bullet)}\to \mathbb{A}_{\mr{cris}}(R_{\infty}/p)^{(\bullet)}$ (namely, it is obtained by base extending the map $\mathfrak{S}(R)^{(\bullet)}\to \prism_{R_{\infty}}^{(\bullet)}$ from \cref{coverings} along $A_{\mr{inf}}\to A_{\mr{cris}}$). This is compatible with the Frobenius lifts.

The following is a standard result, which says that crystals may be described by modules equipped with an HPD-stratification for either one of the coverings discussed above.
\begin{prop}
Let $A^{(\bullet)}$ be either $D^{(\bullet)}$, $\mathbb{A}_{\mr{cris}}(R_{\infty}/p)$ or $\mathbb{A}_{\mr{cris}}(\tilde{R}/p)^{(\bullet)}$. Then the category $Crys(\mathcal{X}_{n}/\Z_p)$ is equivalent to the category of $p$-adically complete $A$-modules $M$, equipped with an isomorphism $\epsilon:p_{1}^{*}M\xrightarrow{\cong} p_{2}^{*}M$ over $A^{(1)}$, satisfying the obvious cocycle condition over $A^{(2)}$.
\end{prop}

\section{The category $\mathcal{M}\mathcal{F}$}\label{sec-MF}
We will now introduce our category of crystalline coefficient objects. The following definition may be seen as a higher dimensional version of the categories introduced by Breuil in \cite{Breuil-constr} in the $0$-dimensional case, except that we only consider crystalline (instead of semistable) objects. It also generalizes the notion of the category $\MF$ introduced by Faltings in \cite{Fal-very}. In particular, we also consider objects which are not locally free.\\ 
Let $a\leq p-2$. Let $R$ be the $p$-adic completion of a small smooth $\mathcal{O}_{K}$-algebra, with \'etale coordinates $t_{1},\ldots , t_{n}$ and let $D$ be the ring introduced in the previous section - then $D$ is an $\mathcal{S}$-algebra, where $\mathcal{S}$ denotes the Breuil-ring. Let $J_{R}$ be the kernel of $D\to R$. The Frobenius $\phi_{D}$ satisfies $\phi_{D}(J_{R}^{[s]})\subset p^{i}D$ for $s \leq p-1$, so we have divided Frobenius maps $\phi_{D}^{s}:=\frac{\phi_{D}}{p^{s}}:J_{R}^{[s]}\to D$. Recall that we have the element $c:=\phi_{D}^{1}(E(u))\in D$.\\
We denote by $\hat{\Omega}^{1}_{D/\Z_p}=\Omega_{\mathfrak{S}(R)}^{1}\hat{\otimes}_{\mathfrak{S}(R)} D$ the $p$-adically completed module of PD-differential operators (see \cite[Tag 07HQ]{stacks} and in particular \cite[Tag 07HW]{stacks}). It is freely generated on $\{du, dt_{k}\}$.
\begin{defn}\label{def-MF}
The category $\MFa^{\mr{big}}(\mathcal{X})$ consists of the following data:
\begin{itemize}
\item A finitely generated $p$-torsion free complete $D$-module $M$, endowed with an integrable quasi-nilpotent connection $\nabla:M\to M\otimes_{D} \hat{\Omega}^{1}_{D/\Z_p}$.
\item A closed submodule $Fil^{a}M\subset M$, such that $J_{R}^{[a]}M\subset Fil^{a}M$, and which is Griffith transversal to $\nabla$, i.e. $J_{R}\nabla(Fil^{a}M)\subset Fil^{a}M$.
\item A $\phi_{D}$-linear morphism $\phi_{M}^{a}:Fil^{a}M\to M$, which is horizontal, i.e. the diagram
\begin{center}
$\xymatrix{
Fil^{a}M\ar[rr]^{E(u)\nabla}\ar[d]^{\phi_{M}^{a}} && Fil^{a}M\otimes \hat{\Omega}^{1}_{D/\Z_p}\ar[d]^{\phi_{M}^{a}\otimes \frac{d\phi_{D}}{p}}\\
M\ar[rr]^{c\nabla} && M\otimes \hat{\Omega}^{1}_{D/\Z_p}
}$
\end{center}
commutes and which induces an isomorphism 
\begin{align}\label{cond-frob-iso}
Fil^{a}M/J_{R}Fil^{a}M\otimes_{\phi_{D}}D/p\xrightarrow{\cong} M/pM.
\end{align}
This means in particular that the image of $\Phi^{a}_{M}$ generates $M$. 
\end{itemize}
We then define the  full subcategory $\MFa^{\mr{tor-free}}(\mathcal{X})\subset \MFa^{\mr{big}}(\mathcal{X})$, consisting of objects $(M, Fil^{a}M, \nabla, \Phi^{a}_{M})$, such that $M$ is flat over $\mathcal{S}$ and $M/Fil^{a}M$ is $p$-torsionfree.
\end{defn}
\begin{rmk}
We can define a Frobenius $\Phi^{0}_{M}:M\to M$, defined as $\Phi^{0}_{M}(m)=c^{-a}\Phi^{a}_{M}(E^{a}m)$. Since $M$ is $p$-torsion free, we have that $\Phi_{M}^{a}=\frac{\Phi^{0}_{M}}{p^{r}}$, so that $\Phi^{a}_{M}$ is in fact determined by $\Phi^{0}_{M}$.
\end{rmk}
The following constructions are similar to \cite[\S 3]{CaisLiu} in the case of Breuil-modules. For any object $(M, Fil^{a}M, \nabla, \phi_{M}^{a})\in \MFa(\mathcal{X})$, we may define a full filtration $Fil^{i}_{M, Fil^{a}M}$ by setting
\begin{center}
\begin{equation}\label{filtration}
Fil_{M, Fil^{a}}^{i}:= \begin{cases}
\{m\in M : J_{R}^{[a-i]}m\subset Fil^{a}M\} &\text{if $i\leq a$}\\
J_{R}^{[i-a]}Fil^{a}M&\text{if $i > a$}.
\end{cases}
\end{equation}
\end{center}
Then $Fil_{M, Fil^{a}M}^{\bullet}$ satisfies Griffiths transversality with respect to $\nabla$. We can also define Frobenius maps $\phi_{M}^{i}:Fil_{M, Fil^{a}M}^{i}\to M$, by setting $\phi_{M}^{i}(m)=c^{-(a-i)}\phi_{M}^{a}(E^{a-i}m)$ (this is independent of the distinguished element $E$). One can then check that the diagrams 
\begin{center}
$\xymatrix{
F_{M}^{i}\ar[rr]^{\nabla}\ar[d]^{\phi_{M}^{i}} && F_{M}^{i-1}\otimes \hat{\Omega}^{1}_{D/\Z_p}\ar[d]^{\phi_{M}^{i-1}\otimes \frac{d\phi_{D}}{p}}\\
M\ar[rr]^{\nabla} && M\otimes \hat{\Omega}^{1}_{D/\Z_p}
}$
\end{center}
commute.\\
For an extensive treatment of Griffiths transversality within the context of crystalline cohomology, we refer the reader to Ogus' book \cite{Ogus-griffiths}. We will now give a description of the category $\MF$ in terms of descent data. For this, we first make the following remark: Let $(D, J_{R})\to (B, J)$ be a $p$-completely flat map, where $B$ is a $p$-adically complete ring, equipped with a PD-ideal $J$ and equipped with a Frobenius lift $\phi_{B}$, compatible with the Frobenius on $D$. If we then equip the module $M_{B}=M\hat{\otimes}_{D}B$ with the filtration $Fil^{a}M_{B}$, defined as the closure of the image of $\sum\limits_{r+s=a}Fil^{r}M\otimes J^{[s]}$, then the morphism $\Phi^{a}_{M_{B}}:=\sum\limits_{r+s=a}\frac{\Phi^{r}_{M}\otimes \phi_{B}}{p^{s}}$ is well-defined on $Fil^{a}M_{B}$. This is because $M_{B}$ is $p$-torsion free and $\Phi^{a}_{M_{B}}$ is seen to agree with $\frac{\Phi^{0}_{M}\otimes \phi_{B}}{p^{a}}$. By abuse of notation we will also write $\Phi^{a}_{M_{B}}=\sum\limits_{r+s=a}\Phi^{r}_{M}\otimes \phi^{s}_{B}$.\\[0.5cm]
Let $(D^{(\bullet)}, J_{D^{(\bullet)}})$ be the cosimplicial object introduced in \cref{sec-prel-cris}, obtained by taking crystalline self-products of $D$. We also again denote by $p_{1}$ (resp. by $p_{2}$) the map $D\to D^{(1)}$ induced by $\mathfrak{S}(R)\to \mathfrak{S}(R)\otimes \mathfrak{S}(R)$ mapping to the first (resp. to the second) factor.
\begin{defn}\label{def-strat}
Let $(A, J, \phi_{A})$ be a $p$-torsion free $\Z_p$-algebra, equipped with a PD-ideal $J$ and a lift $\phi_{A}$ of the absolute Frobenius. We denote by $\MFa^{\mr{ big}}(A)$ the category of triples $(M, Fil^{a}M, \Phi^{a}_{M})$, such that
\begin{itemize}
\item $M$ is a torsion free $A$-module and $Fil^{a}M\subset M$ is a submodule, such that $J^{[a]}M\subset Fil^{a}M$.
\item $\Phi^{a}_{M}:Fil^{a}M\to M$ is a $\phi_{A}$-linear map, which satisfies the conditions analogous to \cref{def-MF}.
\end{itemize}
In analogy to \cref{def-MF}, we also have the full subcategory $\MFa^{\mr{tor-free}}(A)$, where we ask that the quotient $M/Fil^{a}M$ is $p$-torsion free.\\
We denote by $\MFa^{\mr{big}}(D^{(\bullet)})$ the category consisting of the following data
\begin{itemize}
\item A triple $(M, Fil^{a}M, \phi_{M})\in \MFa^{\mr{big}}(D)$. 
\item An isomorphism $\epsilon:p_{1}^{*}M\xrightarrow{\cong} p_{2}^{*}M$, satisfying the ususal cocycle condition over $D^{(2)}$. 
\item The isomorphism satisfies $\epsilon(F^{a}(p_{1}^{*}M))=F^{a}(p_{2}^{*}M)$ and $\epsilon\circ \Phi_{p_{1}^{*}M}^{a}=\Phi_{p_{2}^{*}M}^{a}\circ \epsilon$, where for $i=1,2$, we define $F^{a}(p_{i}^{*}M):=\sum\limits_{r+s=a}F^{r}_{M}\otimes J_{D^{(1)}}^{[s]}$, and the Frobenius $\Phi_{p_{i}^{*}M}^{a}:F^{a}(p_{i}^{*}M)\to p_{i}^{*}M$ by $\Phi_{p_{i}^{*}M}^{a}=\sum\limits_{r+s=a}\Phi_{M}^{r}\otimes \phi^{s}_{D^{(1)}}$. 
\end{itemize}
Again, we also have the full subcategory $\MFa^{\mr{tor-free}}(D^{(\bullet)})$.
\end{defn}
\begin{prop}\label{prop-cris-strat}
The category $\MFa^{\mr{big}}(\mathcal{X})$ is naturally equivalent to the category $\MFa^{\mr{big}}(D^{(\bullet)})$. This equivalence restricts to an equivalence between $\MFa^{\mr{tor-free}}(\mathcal{X})$ and $\MFa^{\mr{tor-free}}(D^{(\bullet)})$.
\end{prop}
\begin{proof}
First let $(M, Fil^{a}M, \nabla, \phi_{M}^{a})\in \MFa^{\mr{big}}(\mathcal{X})$. The integrable connection on $M$ corresponds to an HPD-stratification $\epsilon:p_{1}^{*}M\xrightarrow{\cong} p_{2}^{*}M$ which is given by the Taylor formula
\begin{center}
$m\otimes 1\mapsto \sum\limits_{E}(\prod\nabla_{\partial_{t_{k}}}^{e_{k}})(m)\otimes \prod\gamma_{e_{k}}(p_{1}(t_{k})-p_{2}(t_{k}))$.
\end{center}
Here $t_{1}, \ldots , t_{n}$ are \'etale coordinates and $t_{0}:=u$ and $E=(e_{k})_{k\in \{0, \ldots , n\}}$ runs over all multiindices. Using Griffiths transversality, one can then show that $\epsilon(F^{a}(p_{1}^{*}M))=F^{a}(p_{2}^{*}M)$ (see \cite[Lemma 3.1.3]{Ogus-griffiths}).\\
What is left to prove, is that $\epsilon\phi_{p_{1}^{*}M}^{a}=\phi_{p_{2}^{*}M}^{a}\epsilon$, if and only if the $\Phi^{i}_{M}$ are horizontal.\\
This is an exercise in calculus: The $\Phi_{M}^{i}$ are horizontal iff 
\begin{equation*}
\nabla_{\partial_{t_{k}}}(\Phi_{M}^{i}(m))= \begin{cases}
\Phi_{M}^{i-1}(\nabla_{\partial_{t_{k}}}(m))(\frac{\partial_{t_{k}}\phi_{D}(t_{k})}{p}) &\text{if $i > 0$}\\
\Phi_{M}^{i}(\nabla_{\partial_{t_{k}}}(m))(\partial_{t_{k}}\phi_{D}(t_{k})) &\text{if $i = 0$}
\end{cases}
\end{equation*}
More generally, using the higher chain rules, we may write 
\begin{equation}\label{eq1}
\nabla^{e_{k}}_{\partial_{t_{k}}}(\Phi_{M}^{i}(m))=\sum\limits_{l=1}^{e_{k}}\frac{1}{p^{min\{l, i\}}}\Phi_{M}^{i-l}(\nabla_{\partial_{t_{k}}}^{l}(m))B_{e_{k}, l}(\partial_{t_{k}}\phi_{D}(t_{k}),\ldots , \partial^{e_{k}-l+1}_{t_{k}}\phi_{D}(t_{k})),
\end{equation}
where we define $\Phi^{i}_{M}:=\Phi_{M}^{0}$, for $i < 0$. Also, by $B_{e_{k}, l}$ we denote the partial Bell polynomials. All terms are well-defined.\\
Now let $m\otimes x\in F^{r}_{M}\otimes I_{D^{(1)}}^{[s]}\subset  F^{a}(p_{1}^{*}M)$, i.e. $r+s=a$. Then $\Phi_{p_{1}^{*}M}^{a}(m\otimes x)=\Phi^{r}_{M}(m)\otimes \phi_{D^{(1)}}^{s}(x)$ and we have
\begin{center}
$\epsilon\Phi_{p_{1}^{*}M}^{a}(m\otimes x)=\sum\limits_{E=(e_{k})}(\prod\nabla_{\partial_{t_{k}}}^{e_{k}})(\Phi^{k}_{M}(m))\otimes \phi_{D^{(1)}}^{s}(x)\prod\gamma_{e_{k}}(p_{1}(t_{k})-p_{2}(t_{k}))$.
\end{center}
We now claim that 
\begin{align*}
& \sum\limits_{E=(e_{k})}(\prod\nabla_{\partial_{t_{k}}}^{e_{k}})(\Phi^{r}_{M}(m))\otimes \prod\gamma_{e_{k}}(p_{1}(t_{k})-p_{2}(t_{k}))\\
= & \sum\limits_{E}\Phi_{M}^{k-\abs{E}}(\prod\nabla^{e_{k}}_{\partial_{t_{k}}})(m)\otimes \frac{1}{p^{min\{\abs{E}, k\}}}\prod\gamma_{e_{k}}(p_{1}(\phi_{D}(t_{k}))-p_{2}(\phi_{D}(t_{k}))).
\end{align*}
We show this using induction. To shorten formulas, we write $\delta_{e_{k}}(t_{k}))=\gamma_{e_{k}}(p_{1}(t_{k})-p_{2}(t_{k}))$ and $\delta_{e_{k}}(\phi(t_{k})))=\gamma_{e_{k}}(p_{1}(\phi_{D}(t_{r}))-p_{2}(\phi_{D}(t_{r}))$. Fix an element $r_{0}\in R$. For any multiindex $E=(e_{k})$, we write $E'=(e_{k})_{k\in \{0, \ldots , n\}}$. By induction, we then have
\begin{align*}
& \sum\limits_{E=(e_{k})}(\prod\nabla_{\partial_{t_{k}}}^{e_{k}})(\Phi^{k}_{M}(m))\otimes \prod\delta_{e_{k}}(t_{k}))\\
= & \sum\limits_{e_{n}}\left[\nabla^{e_{n}}_{\partial_{t_{k}}}\sum\limits_{E'}(\prod\nabla_{\partial_{t_{k}}}^{e_{k}})(\Phi^{k}_{M}(m))\otimes \prod\delta_{e_{k}}(t_{k}))\right]\delta_{e_{n}}(t_{n}))\\
= & \sum\limits_{e_{n}}\left[\nabla^{e_{n}}_{\partial_{t_{n}}}\sum\limits_{E'}\Phi_{M}^{k-\abs{E'}}(\prod\nabla^{e_{k}}_{\partial_{t_{k}}})(m)\otimes \frac{1}{p^{min\{\abs{E'}, k\}}}\prod\delta_{e_{k}}(\phi(t_{k}))\right]\delta_{e_{n}}(t_{n}))
\end{align*}
Using \cref{eq1}, and setting $B_{n, l}=B_{n, l}(\partial_{t_{n}}\phi_{D}(t_{n}),\ldots , \partial^{n-l+1}_{t_{n}}\phi_{D}(t_{n}))$, we may rewrite this as 
\begin{align*}
= & \sum\limits_{e_{n}}\left[\sum\limits_{l=1}^{e_{n}}\sum\limits_{E'}\Phi_{M}^{k-\abs{E'}-l}(\prod\nabla^{e_{k}}_{\partial_{t_{k}}})\nabla^{l}_{\partial_{t_{n}}}(m)\otimes \frac{1}{p^{min\{\abs{E'}+l, k\}}}\prod\delta_{e_{k}}(\phi(t_{k}))p_{2}(B_{e_{n}, l})\right]\delta_{e_{n}}(t_{n}))\\
= & \sum\limits_{l=1}^{\infty}\left[\sum\limits_{E'}\Phi_{M}^{k-\abs{E'}-l}(\prod\nabla^{e_{k}}_{\partial_{t_{k}}})\nabla^{l}_{\partial_{t_{n}}}(m)\otimes \frac{1}{p^{min\{\abs{E'}+l, k\}}}\prod\delta_{e_{k}}(\phi(t_{k}))\right]\sum\limits_{e_{n}=l}^{\infty}p_{2}(B_{e_{n}, l})\delta_{e_{n}}(t_{n})\\
= & \sum\limits_{l=1}^{\infty}\left[\sum\limits_{E'}\Phi_{M}^{k-\abs{E'}-l}(\prod\nabla^{e_{k}}_{\partial_{t_{k}}})\nabla^{l}_{\partial_{t_{n}}}(m)\otimes \frac{1}{p^{min\{\abs{E'}+l, k\}}}\prod\delta_{e_{k}}(\phi(t_{k}))\right]\delta_{l}(\phi(t_{n}))\\
= & \sum\limits_{E}\Phi_{M}^{k-\abs{E}}(\prod\nabla^{e_{k}}_{\partial_{t_{k}}})(m)\otimes \frac{1}{p^{min\{\abs{E}, k\}}}\prod\delta_{e_{k}}(\phi(t_{k})))\\
= & \sum\limits_{E}\Phi_{M}^{k-\abs{E}}(\prod\nabla^{e_{k}}_{\partial_{t_{k}}})(m)\otimes \phi^{min\{\abs{E}, k \}}_{D^{(1)}}(\prod\delta_{e_{k}}(t_{r}))).
\end{align*}
Above we used that 
$\sum\limits_{e_{n}=l}^{\infty}p_{2}(B_{e_{n}, l})\delta_{e_{n}}(t_{n})=\delta_{l}(\phi(t_{n}))$, which follows from Taylor's theorem. We thus see that
\begin{align*}
\epsilon\Phi_{p_{1}^{*}M}^{a}(m\otimes x) & = & \sum\limits_{E=(e_{k})}(\prod\nabla_{\partial_{t_{k}}}^{e_{k}})(\Phi^{k}_{M}(m))\otimes \phi_{D^{(1)}}^{s}(x)\prod\delta_{e_{k}}(t_{r})\\
& = & \sum\limits_{E}\Phi_{M}^{k-\abs{E}}(\prod\nabla^{e_{k}}_{\partial_{t_{k}}})(m)\otimes \phi^{min\{\abs{E}, k \}}_{D^{(1)}}(\prod\delta_{e_{k}}(t_{k})))\phi_{D^{(1)}}^{s}(x)\\
& = & \Phi^{a}_{p_{2}^{*}M}\epsilon(m\otimes x).
\end{align*}
Conversely, one also easily sees that $\epsilon\Phi_{p_{1}^{*}M}^{a}=\Phi^{a}_{p_{2}^{*}M}\epsilon$ implies that the $\Phi_{M}^{i}$ are horizontal with respect to $\nabla$.
\end{proof}
\begin{rmk}
The stratification $\epsilon$ in the proposition above in fact fixes the full filtration of $p_{1}^{*}M$, which is defined as the product of the filtration $Fil^{\bullet}_{M, Fil^{a}M}$ with the PD-filtration $J_{R^{(1)}}^{[\bullet]}$ (this again follows from Griffiths transversality and \cite[Lemma 3.1.3]{Ogus-griffiths}).
\end{rmk}
We now want to globalize the construction of the category $\MFa$. We first record the following.
\begin{prop}\label{MF-to-crystals}
Let $n\geq 1$ and define $\mathcal{X}_{n}=\mathcal{X}/p^{n}$. Let $(M, Fil^{a}M, \nabla, \phi_{M}^{a})\in \MFa^{\mr{big}}(\mathcal{X})$. Then $(M/p^{n}, Fil^{a}M/p^{n}, \nabla/p^{n}, \Phi^{a}_{M}/p^{n})$ defines a pair $(\mathcal{M}_{n}, Fil^{a}\mathcal{M}_{n})$, where $\mathcal{M}_{n}$ is a crystal in $\mathcal{O}_{\mathcal{X}_{n, cris}}$-modules over $(\mathcal{X}_{n}/\mathbb{Z}_{p})_{\mr{cris}}$, and $\overbar{\mathcal{M}}\subset \mathcal{M}$ is a subsheaf, satisfying $\mathcal{I}_{\mathcal{X}_{n}}^{[a]}\mathcal{M}_{n}\subset Fil^{a}\mathcal{M}_{n}$. Here $Fil^{a}M/p^{n}$ denotes the image of $Fil^{a}M$ in $M/p^{n}$.
\end{prop}
\begin{proof}
This is standard. By \cref{sec-prel-cris} $D$ covers the final object of $(\mathcal{X}_{n}/\mathbb{Z}_{p})_{\mr{cris}}$. More precisely, for any $(B, I)\in (\mathcal{X}_{n}/\mathbb{Z}_{p})_{\mr{cris}}$ there exists a morphism $\iota:(D, J_{R})\to (B, I)$, and we define $\mathcal{M}_{n}(B):=M/p^{n}\otimes_{D, \iota} B$. Now consider the full filtration $Fil^{\bullet}_{M, Fil^{a}M}$ of $M$, defined by \cref{filtration}. We then define $Fil^{a}\mathcal{M}_{n}(B):=Im(\sum\limits_{r+s=a} Fil^{r}M/p^{n}\otimes I^{[s]}\to M/p^{n}\otimes B$. Since the HPD-stratification $\epsilon$ from \cref{prop-cris-strat} respects the pull-back filtrations, this is independent of $\iota$ and one checks that this defines a subsheaf of $\mathcal{M}_{n}$.
\end{proof}
Note that the construction of the subsheaves $Fil^{a}\mathcal{M}_{n}$ depends a priori on the full filtration $Fil^{\bullet}_{M, Fil^{a}M}$, which in turn depends on the chosen weakly initial object $D$. The following result however shows that this inconvenience is not so serious.
\begin{lem}\label{cris-maximality}
Let $P$ be a smooth $\mathfrak{S}$-algebra with a Frobenius lift $\phi$, equipped with a surjection $P\to R$, and let $P^{pd}\to R$ be the $p$-adically completed pd-envelope of the kernel $J$. Let further $(M, Fil^{a}M, \Phi^{a})\in \MFa(P^{pd})$. Then $Fil^{a}M$ is maximal in the following sense: There exists no triple $(M, \tilde{Fil}^{a}M, \Psi^{a})\in \MFa(P^{pd})$ with $Fil^{a}M\subsetneq \tilde{Fil}^{a}M$ and such that $\Psi^{a}$ extends $\Phi^{a}$.
\end{lem}
\begin{proof}
Assume that we have a triple $(M, \tilde{Fil}^{a}M, \Psi^{a})$, extending $(M, Fil^{a}M, \Phi^{a})$. Let $Q$ be the cokernel of $\tilde{Fil}^{a}M\hookrightarrow Fil^{a}M$. The Frobenius map $\Phi^{a}$ induces an isomorphism $Fil^{a}M/JFil^{a}M\otimes_{\phi}P^{pd}/p\cong M/p$ and likewise $\Psi^{a}$ induces an isomorphism $\tilde{Fil}^{a}M/J\tilde{Fil}^{a}M\otimes_{\phi}P^{pd}/p\cong M/p$. Thus, we have $Fil^{a}M/JFil^{a}M\otimes_{\phi}P^{pd}/p \xrightarrow{\cong} \tilde{Fil}^{a}M/J\tilde{Fil}^{a}M\otimes_{\phi}P^{pd}/p$. Modding out by $J_{R}^{[p]}$, it follows that $Fil^{a}M/JFil^{a}M\otimes_{\phi}P/p \xrightarrow{\cong} \tilde{Fil}^{a}M/J\tilde{Fil}^{a}M\otimes_{\phi}P/p$. Note here that as $P$ is $p$-torsion free, $P^{pd}$ is obtained by first adjoining the elements $\frac{x^{p^{n}}}{p^{n}}$ to $P$ and then taking the $p$-adic completion. This means in particular that $P^{pd}/(p, J_{R}^{[p]})=P/(p, \phi_{P}(J_{R}))$. Moreover, the Frobenius map is flat on $P$, so we then get that $Fil^{a}M/(J_{R}, p)=\tilde{Fil}^{a}M/(J_{R}, p)$. By Nakayama's lemma we then also get that $Fil^{a}M/(J_{R}^{[p]}, p)=\tilde{Fil}^{a}M/(J_{R}^{[p]}, p)$. But now $Q$ is $J_{R}^{[p]}$-torsion (as both $Fil^{a}M$ and $\tilde{Fil}^{a}M$ contain $J_{R}^{[a]}M$). Hence, we obtain $Q/p=0$. But then a further application of Nakayama's lemma yields $Q=0$.
\end{proof}
The maximality of $Fil^{a}M$ will later turn out to be crucial.\\

We will now discuss the independence of the Frobenius lift. We recall the following construction (see \cite[\S 1.4.]{Font-Laff} and \cite[\S II (b)]{Fal-cris}).
\begin{constr}
Let $(M, Fil^{a}M, \nabla, \Phi_{M}^{a})\in \MFa(\mathcal{X})$ and denote again by $Fil^{i}_{M, Fil^{a}M}\subset M$ the filtration defined by \cref{filtration}. We then denote by $\tilde{M}$ the $D$-module obtained by taking the quotient of the direct sum $\bigoplus\limits_{i\leq a}Fil^{i}_{M, Fil^{a}M}$ modulo the equivalence relation, which identifies $pFil^{i+1}_{M, Fil^{a}M}$ with $Fil^{i}_{M, Fil^{a}M}$. The module $\tilde{M}$ is endowed with an integrable $p$-connection and $\tilde{M}\otimes_{\phi} D$ then carries an integrable flat connection (see \cite[II (d)]{Fal-cris} - this is in fact an instance of a non-abelian Hodge correspondence in mixed characteristics, see \cite{shiho}). The Frobenius maps on $Fil^{i}M$ then induce a horizontal map 
\begin{center}
$\tilde{M}\otimes_{\phi} D\to M$
\end{center}
and $\Phi^{a}_{M}$ is simply the map $Fil^{a}M\to \tilde{M}\otimes_{\phi} D\to M$.\\

Let $(B, J, \phi_{B}, f)$ be a quadruple, where $B$ is $p$-complete and $p$-torsion free, $J$ is a PD-ideal (compatible with the PD-structure on $pB$), $\phi_{B}$ is a lift of the absolute Frobenius on $B/p$ and $f:Spf(B/J)\to \mathcal{X}$ is a morphism. Assume that there is a $p$-completely flat morphism $\iota:D\to B$, which lifts $R\to B/J$. Then consider $M\hat{\otimes}_{D}B$ with filtration the closure of $\sum\limits_{r+s=a}Fil^{r}M\hat{\otimes} J^{[s]}$ in $M_{B}$. We may define a Frobenius $\Phi^{a}_{M_{B}}$. Namely, by \cite[Theorem 2.3]{Fal-cris}, the connection induces a canonical isomorphism 
\begin{center}
$\alpha:\tilde{M}\otimes_{\phi_{D}}D\hat{\otimes}_{D}B\xrightarrow{\cong} \tilde{M}\hat{\otimes}_{D}B\otimes_{\phi_{B}}B$,
\end{center}
using that the composite maps $D\xrightarrow{\phi_{D}}D\to B$ and $D\to B\xrightarrow{\phi_{B}}B$ agree modulo $p$. The isomorphism is given by
\begin{equation}\label{p-taylor}
m\otimes 1\mapsto \sum\limits_{E=(e_{k})}\prod\nabla(\partial_{t_{k}})^{e_{k}}(m)\otimes \frac{(\iota(\phi_{D}(t_{k}))-\phi_{B}(\iota(t_{k}))^{e_{k}}}{p^{min\{i, \abs{E}\}}e_{k}!}
\end{equation}
if $m$ is a residue class of an element in $Fil^{i}$.\\

Let $r, s\geq 0$, such that $r+s=a$. On the image of $Fil^{r}M\otimes J_{B}^{[s]}$, we may then define maps $"\Phi^{r}_{M}\otimes \phi^{s}_{B}"$ by first defining $"\Phi^{r}_{M}\otimes \phi"$ as the morphism
\begin{equation}\label{Frob-mf}
Fil^{r}M\otimes B\to \tilde{M}\hat{\otimes}_{D}B\otimes_{\phi_{B}}B\xrightarrow[\cong]{\alpha} \tilde{M}\otimes_{\phi_{D}}D\hat{\otimes}_{D}B\xrightarrow{\tilde{\Phi}\otimes id}M\hat{\otimes}_{D}B,
\end{equation}
and then setting $"\Phi^{r}_{M}\otimes \phi^{s}_{B}"=\frac{"\Phi^{r}_{M}\otimes \phi"}{p^{s}}$, which is well defined by $p$-torsion freeness of $M_{B}$. The Frobenius $\Phi^{a}_{M_{B}}:Fil^{a}M_{B}\to M_{B}$ is then defined by $\Phi^{a}_{M_{B}}:=\sum\limits_{r+s=a}"\Phi^{r}_{M}\otimes \phi^{s}_{B}"$, which is easily checked to be well-defined, as it agrees with $\frac{"\Phi^{0}_{M}\otimes \phi"}{p^{a}}$.

\end{constr}
For a general smooth $p$-adic formal scheme, the category $\MF$ is defined by glueing. That this is indeed possible, follows from the following proposition.
\begin{prop}\label{indep-mf}
Let $\mathcal{O}_{K}\{t_{1}, \ldots, t_{n}\}\to R$ be another framing of $R$, i.e. an \'etale map and let $\mathfrak{S}\{t_{1}, \ldots, t_{n}\}\to \tilde{P}$ be its \'etale lift along $\mathfrak{S}\to \mathcal{O}_{K}$ endowed with the usual Frobenius lift $t_{i}\mapsto t_{i}^{p}$. Let $\tilde{D}$ be the $p$-adically completed pd-envelope of the kernel and denote by $\MFa(Spf(R))^{\tilde{D}}$ the category of quadruples $(M, Fil^{a}M, \nabla, \Phi^{a}_{M})$ over $\tilde{D}$ satisfying the conditions from \cref{def-MF}. Then there is a canonical equivalence of categories $\MFa(Spf(R))\xrightarrow{\cong}\MFa(Spf(R))^{\tilde{D}}$.
\end{prop}
\begin{proof}
The functor has essentially already been described above. Choose a map $D\to \tilde{D}$ compatible with the maps to $R$. This is flat by the fiberwise flatness criterion. Then the filtration $Fil^{a}M_{\tilde{D}}$ on $M\hat{\otimes}_{D}\tilde{D}$ is defined as the closure of $\sum_{r+s=a}Fil^{r}_{M}\otimes J_{\tilde{D}}^{[s]}$ in $M_{\tilde{D}}$ and the Frobenius map $\Phi^{a}_{M_{\tilde{D}}}$ over $\tilde{D}$ is defined by \cref{Frob-mf}, which is well-defined, since $M\hat{\otimes}_{D}\tilde{D}$ is $p$-torsion free, by the flatness of $D\to \tilde{D}$. One now sees that this also induces an isomorphism $Fil^{a}M_{\tilde{D}}/J_{\tilde{D}}Fil^{a}M_{\tilde{D}}\otimes_{\phi}\tilde{D}/p\tilde{D}\xrightarrow{\cong} M_{\tilde{D}}/pM_{\tilde{D}}$, by noting that this map is just the pushout of $Fil^{a}M/J_{D}Fil^{a}M\otimes_{\phi_{D}}\tilde{D}/p\tilde{D}\xrightarrow{\cong} M/pM$ up to an automorphism induced by \cref{p-taylor}. We thus see that we indeed obtain an object of $\MFa(Spf(R))^{\tilde{D}}$.\\
A quasi-inverse is obtained by choosing a map $\tilde{D}\to D$ and then applying the same construction as above. To see that these constructions are quasi-inverse to each other, denote the chosen maps by $f:D\to \tilde{D}$ and $g:\tilde{D}\to D$. Let $(M, Fil^{a}M, \nabla, \Phi^{a}_{M})\in \MFa(Spf(R))$ and let $(M_{\tilde{D}}, Fil^{a}M_{\tilde{D}}, \nabla_{\tilde{D}}, \Phi^{a}_{\tilde{D}})$ be the induced object in $\MFa(Spf(R))^{\tilde{D}}$. Now applying the functor in the other direction, one gets $M\hat{\otimes}_{gf}D$ with filtration $F^{a}=\sum Fil^{r}_{M_{\tilde{D}}}\hat{\otimes} J_{D}^{[s]}$ (where again $Fil^{r}_{M_{\tilde{D}}}$ is defined by \cref{filtration}). The connection $\nabla$ then induces an isomorphism $M\xrightarrow{\cong} M\otimes_{gf}D$, which takes $Fil^{a}M$ into $F^{a}$ and one checks that the Frobenius on $F^{a}$ restricts to the Frobenius on $Fil^{a}M$. But by \cref{cris-maximality}, $Fil^{a}M$ is maximal, so we actually see that $Fil^{a}M=F^{a}$.
\end{proof}
The proof above in particular also shows that for all $n\geq 1$, for the construction of the subsheaves $Fil^{a}\mathcal{M}_{n}\subset \mathcal{M}_{n}$, one may use the full filtration of $\mathcal{M}_{n}(\tilde{D})$ over any $\tilde{D}$ as above.\\
\cref{indep-mf} now implies the following corollary, which ensures that we also get a category $\MFa(\mathcal{X})$ for arbitrary smooth $p$-adic formal schemes by glueing.
\begin{cor}\label{mf-descent}
The association $\mathcal{X}\mapsto \MFa^{\mr{tor-free}}$ satisfies effective \'etale descent.
\end{cor}

Finally, we need a description of the category $\MF$ within the realm of quasi-regular semiperfectoid rings. Here we again assume that $\mathcal{X}=Spf(R)$ is small affine. Recall from \cref{sec-prel-cris} the weakly initial objects $(\mathbb{A}_{\mr{cris}}(\tilde{R}/p), J_{\tilde{R}})$ and $(\mathbb{A}_{\mr{cris}}(R_{\infty}/p), J_{R_{\infty}})$. Let $A$ be either one of the two. We now define categories 
\begin{center}
$\MFa^{\mr{big}}(\mathbb{A}_{\mr{cris}}(A/p)^{(\bullet)})^{\mr{full}}$ and $\MFa^{\mr{tor-free}}(\mathbb{A}_{\mr{cris}}(A/p)^{(\bullet)})^{\mr{full}}$.
\end{center}
analogously to \cref{def-strat}.
\begin{defn}
We denote by $\MFa^{\mr{big}}(\mathbb{A}_{\mr{cris}}(A/p)^{(\bullet)})^{\mr{full}}$ the category consisting of the following data
\begin{itemize}
\item A triple $(M, Fil^{a}M, \phi_{M})\in \MFa^{\mr{big}}(\mathbb{A}_{\mr{cris}}(A/p))$. 
\item An isomorphism $\epsilon:p_{1}^{*}M\xrightarrow{\cong} p_{2}^{*}M$, satisfying the ususal cocycle condition over $\mathbb{A}_{\mr{cris}}(A/p)^{(2)}$. 
\item Consider the full filtration $Fil^{\bullet}_{M, Fil^{a}M}$, defined analogously to \cref{filtration}. Then the isomorphism satisfies $\epsilon(F^{k}(p_{1}^{*}M))=F^{k}(p_{2}^{*}M)$, for all $0\leq k \leq a$ and $\epsilon\circ \Phi_{p_{1}^{*}M}^{a}=\Phi_{p_{2}^{*}M}^{a}\circ \epsilon$. Here for $i=1,2$, we define $F^{k}(p_{i}^{*}M):=\sum\limits_{r+s=k}F^{r}_{M}\hat{\otimes} J_{\mathbb{A}_{\mr{cris}}(A/p)^{(1)}}^{[s]}$, and the Frobenius $\Phi_{p_{i}^{*}M}^{a}:F^{a}(p_{i}^{*}M)\to p_{i}^{*}M$ by $\Phi_{p_{i}^{*}M}^{a}=\sum\limits_{r+s=a}\Phi_{M}^{r}\otimes \phi^{s}_{\mathbb{A}_{\mr{cris}}(A/p)^{(1}}$. 
\end{itemize}
Again, we also have the full subcategory $\MFa^{\mr{tor-free}}(\mathbb{A}_{\mr{cris}}(A/p)^{(\bullet)})$.
\end{defn}
\begin{rmk}
In contrast to \cref{def-strat}, we had to add the condition that the stratification fixes the full filtration, which is automatic for objects over $D^{(\bullet)}$.\\
Note that we will see later (see \cref{sec-constr}) that this extra condition is also automatically satisfied for $A=R_{\infty}$.
\end{rmk}
The filtration of objects in $\MFa(\mathbb{A}_{\mr{cris}}(R_{\infty}/p)^{(\bullet)})$ is again maximal.
\begin{lem}\label{Acris-maximality}
Let $A$ be a perfectoid $\mathcal{O}_{C}$-algebra. Then for any $(M, Fil^{a}M, \epsilon, \Phi^{a}_{M})\in \MFa^{big}(\mathbb{A}_{\mr{cris}}(A/p))$, $Fil^{a}M$ is maximal in the sense of \cref{cris-maximality}.
\end{lem}
\begin{proof}
The proof of \cref{cris-maximality} works also in this case.
\end{proof}

\begin{prop}\label{left-inverse}
Let $A$ denote either $\tilde{R}$ or $R_{\infty}$. Then there is a canonical functor $\mathbb{D}_{A}:\MFa^{big}(\mathbb{A}_{\mr{cris}}(A/p)^{(\bullet)})^{\mr{full}}\to \MFa^{big}(Spf(R))$.
\end{prop}
\begin{proof}
We have to show that any $(M, Fil^{a}M, \epsilon, \Phi^{a}_{M})\in \MFa^{big}(\mathbb{A}_{\mr{cris}}(A/p)^{(\bullet)})^{\mr{full}}$ descends to an object over $D$. For this, we first show that $(M, Fil^{a}M, \epsilon)$ descends, which follows from the following.
\begin{claim}
For any $n\geq 1$, the triple $(M/p^{n}, Fil^{a}M/p^{n}, \epsilon/p^{n})$ defines a pair $(\mathcal{M}_{n}, Fil^{a}\mathcal{M}_{n})$ of a crystal in $\mathcal{O}_{\mathcal{X}_{cris, n}}$-modules, equipped with a sheaf of submodules $Fil^{a}\mathcal{M}_{n}\subset \mathcal{M}_{n}$, such that $\mathcal{I}_{mathcal{X}_{cris, n}}\mathcal{M}_{n}\subset Fil^{a}\mathcal{M}_{n}$.
\end{claim}
To see this, we first recall that $\mathbb{A}_{\mr{cris}}(A/p)/p^{n}$ is a weakly initial object of $(\mathcal{X}_{n}/\Z_p/p^{n})_{\mr{cris}}$. So for any object $(B, J)\in (\mathcal{X}_{n}/\Z_p/p^{n})_{\mr{cris}}$, there exists a faithfully flat covering $(B, J)\to (C, JC)$, such that there exists a morphism $(\mathbb{A}_{\mr{cris}}(A/p)/p^{n}, J_{\tilde{R}})\to (C, JC)$. Now consider the full filtration $Fil^{\bullet}_{M, Fil^{a}M}$ of $M/p^{n}$. Then the stratification $\epsilon$ defines a descent datum of filtered modules over the filtered ring $C$ (equipped with the PD-filtration $(JC)^{[\bullet]}$). The filtered module $(M_{C}, Fil_{M_{C}}^{\bullet})$ over $C$ corresponds to the graded module $\tilde{M}=\bigoplus\limits_{n\geq 0}Fil_{M_{C}}^{n}z^{n}$ over the Rees algebra $R_{C}:=\bigoplus\limits_{n\geq 0}J_{C}^{[n]}Cz^{n}\subset C[z]$. Now, since $J_{C}=JC$, the algebra $R_{C}$ is simply the pushout of the Rees algebra $R_{B}$ along the map $B\to C$, where $R_{B}= \bigoplus\limits_{n\geq 0}J_{B}^{[n]}$. Hence $R_{B}\to R_{C}$ is faithfully flat and the filtered module $(M_{C}, Fil_{M_{C}}^{\bullet})$ descends to a filtered module over $B$, by fpqc descent for graded modules over graded rings.
In particular, after passing to the $p$-adic limit, we obtain a filtered module $(M_{D}, Fil^{\bullet}_{D})$ over the filtered ring $D$ with a stratification $\epsilon_{D}$ over $D^{(1)}$, which fixes the filtration, i.e. $\epsilon_{D}$ takes the filtered pullback $p_{1}^{*}Fil^{k}$ to $p_{2}^{*}Fil^{k}$. This implies that the filtration satisfies Griffiths transversality with respect to the connection $\nabla$, induced by $\epsilon$. So in particular $J_{D}\nabla(Fil^{a})\subset Fil_{D}^{a}$..\\[0.5cm]
We are thus left with proving that the Frobenius descends as well. Note also that it is enough to descend $\Phi^{0}:M\to M$ as we then always get a divided Frobenius by dividing by $p^{r}$. Assume first that $B=\tilde{R}$ and consider the qrsp algebra $\tilde{R}\hat{\otimes}_{R}R_{\infty}$. The morphism of cosimplicial rings $\mathbb{A}_{\mr{cris}}(\tilde{R}/p)^{(\bullet)}\to \mathbb{A}_{\mr{cris}}(\tilde{R}\hat{\otimes}_{R}R_{\infty}/p)^{(\bullet)}$ is $p$-completely faithfully flat and compatible with Frobenius lifts. The same holds for $\mathbb{A}_{\mr{cris}}(R_{\infty}/p)^{(\bullet)}\to \mathbb{A}_{\mr{cris}}(\tilde{R}\hat{\otimes}_{R}R_{\infty}/p)^{(\bullet)}$. We thus first get a Frobenius with descent datum over $\mathbb{A}_{\mr{cris}}(\tilde{R}\hat{\otimes}_{R}R_{\infty}/p)$, which we then can descend to $\mathbb{A}_{\mr{cris}}(R_{\infty}/p)^{(\bullet)}$ by $p$-completely faithfully flat descent. The Frobenius may then be further descended along $D\to \mathbb{A}_{\mr{cris}}(R_{\infty}/p)$, induced by mapping a coordinate $t_{i}$ to $[t_{i}^{\flat}]$. This gives the desired functor $\mathbb{D}_{\tilde{R}}$.\\[0.5cm]
To construct the functor $\mathbb{D}_{R_{\infty}}$, now consider the $p$-completely faithfully flat map $D^{(\bullet)}\to \mathbb{A}_{\mr{cris}}(R_{\infty}/p)^{(\bullet)}$ of cosimiplicial rings, compatible with Frobenius lifts. By descent, we then again get a functor $\MFa^{big}(\mathbb{A}_{\mr{cris}}(R_{\infty}/p)^{(\bullet)})^{\mr{full}}\to \MFa^{big}(D^{(\bullet)})$.
\end{proof}
\begin{rmk}
The functor $\mathbb{D}_{\tilde{R}}$ is indeed canonical. Namely, assume that $\tilde{D}$ is defined like $D$, but using a different set of \'etale coordinates $T\subset R$. Then there is a commutative diagram
\begin{center}
$\xymatrix{
& \MFa^{big}(\mathbb{A}_{\mr{cris}}(\tilde{R}/p)^{(\bullet)})^{\mr{full}}\ar[dl]\ar[dr] & \\
 \MFa^{big}(Spf(R))\ar[rr]^{\cong} && \MFa^{big}(Spf(R))^{\tilde{D}}
}$
\end{center}
where the diagonal functors are obtained by carrying out the construction of the previous proof, using either $R\to R_{\infty}$ or $R\to R_{\infty, T}$, where $R_{\infty, T}$ is now the perfectoid ring, obtained by extracting roots of elements in $T$.\\
To see this, denote by $M=\mathcal{M}(D)$ the descended object over $D$ and by $\mathcal{M}(\tilde{D})$ the one over $\tilde{D}$. One can then show that the Frobenius on $M\hat{\otimes}_{D}\mathbb{A}_{cris}(\tilde{R}\hat{\otimes}R_{\infty}\hat{\otimes} R_{\infty, T})$ is compatible with the Frobenius on $\mathcal{M}(\tilde{D})\hat{\otimes}_{\tilde{D}}\mathbb{A}_{cris}(\tilde{R}\hat{\otimes}R_{\infty}\hat{\otimes} R_{\infty, T})$ under the canonical isomorphism $f:\mathcal{M}(D)\hat{\otimes}_{D}\mathbb{A}_{cris}(\tilde{R}\hat{\otimes}R_{\infty}\hat{\otimes} R_{\infty, T})\cong \mathcal{M}(\tilde{D})\hat{\otimes}_{\tilde{D}}\mathbb{A}_{cris}(\tilde{R}\hat{\otimes}R_{\infty}\hat{\otimes} R_{\infty, T})$.\\
Namely, write $\mathcal{M}(\tilde{D})=M\otimes_{D}\tilde{D}$ and denote the map $D\to \mathbb{A}_{cris}(\tilde{R}\hat{\otimes}R_{\infty}\hat{\otimes} R_{\infty, T})$ by $\alpha$ and the map $D\to \tilde{D}\to \mathbb{A}_{cris}(\tilde{R}\hat{\otimes}R_{\infty}\hat{\otimes} R_{\infty, T})$ by $\beta$.  To simplify the notation, we also write $\mathbb{A}:=\mathbb{A}_{cris}(\tilde{R}\hat{\otimes}R_{\infty}\hat{\otimes} R_{\infty, T})$. Then we claim that there is a commutative diagram
\begin{align}\label{diagram}
\xymatrix{
M\hat{\otimes}_{\alpha} \mathbb{A}\ar[d]_{f}^{\cong}\ar[r]^{\Phi_{M}\otimes \phi} & M\hat{\otimes}_{\alpha} \mathbb{A}\ar[d]_{f}^{\cong} \\
 M\hat{\otimes}_{\beta} \mathbb{A}\ar[r]^{\tilde{\Phi}} & M\hat{\otimes}_{\beta} \mathbb{A}.
}
\end{align}
Here $\tilde{\Phi}$ is defined as
\[M\otimes_{\beta} \mathbb{A}\to \tilde{M}\otimes_{\beta}\mathbb{A}\otimes_{\phi_{\mathbb{A}}}\mathbb{A}\xrightarrow[\cong]{\iota}\tilde{M}\otimes_{\phi_{D}}D\otimes_{\beta}\mathbb{A}\xrightarrow{\Phi_{M}} M\otimes_{\beta}\mathbb{A},\]
where $\iota$ is defined by \cref{p-taylor}. Note that this is the base extension of the Frobenius on $\mathcal{M}(\tilde{D})$ induced by $\Phi_{M}$ via the equivalence $\MFa^{big}(Spf(R))\longleftrightarrow  \MFa^{big}(Spf(R))^{\tilde{D}}$. We will indicate how one shows that \cref{diagram} is indeed commutative. We will stick to the case of a single coordinate, the higher dimensional case follows by induction as in the proof of \cref{prop-cris-strat}.\\
We first note that $\tilde{\Phi}\circ f$ is given by 
\begin{align}
m\otimes 1 & \mapsto  & \sum\limits_{i,j} \Phi_{M}(\nabla_{\partial_{t}}^{i+j}(m
))\otimes \gamma_{i}(\phi_{\mathbb{A}}(\alpha(t))-\phi_{\mathbb{A}}(\beta(t)))\gamma_{j}(\phi_{\mathbb{A}}(\beta(t))-\beta(\phi_{D}(t)))\\
& = & \sum\limits_{k}\Phi_{M}(\nabla_{\partial_{t}}^{k}(m
))\otimes \gamma_{k}(\phi_{\mathbb{A}}(\alpha(t))-\beta(\phi_{D}(t))),
\end{align}
where we again write $\gamma$ for the divided powers.\\
On the other hand $f\circ \Phi_{M}$ is given as
\begin{align}
m\otimes 1 & \mapsto  & \sum\limits_{k}\nabla_{\partial_{t}}^{k}(\Phi_{M}(m))\otimes \gamma_{k}(\alpha(t)-\beta(t))
\end{align}
But now, using the higher chain rules (similar to the proof of \cref{prop-cris-strat}) one can rewrite this as 
\[\sum\limits_{l=1}^{\infty}\Phi(\nabla^{l}(m))\otimes \sum\limits_{k=l}^{\infty}\beta(B_{k, l})\gamma_{k}(\alpha(t)-\beta(t))=\sum\limits_{l}\Phi_{M}(\nabla^{l}(m))\otimes \gamma_{l}(\alpha(\phi_{D}(t))-\beta(\phi_{D}(t))\]
where again $B_{k, l}:=B_{k, l}(\partial_{t}\phi_{D}(t),\ldots , \partial^{k-l+1}_{t}\phi_{D}(t))$ denote the partial Bell polynomials. Using that $\alpha\phi_{D}=\phi_{\mathbb{A}}\alpha$, we see that this agrees with $(\tilde{\Phi}(f(m\otimes 1))$.

\end{rmk}

We did not find a way to prove that the $p$-torsion freeness of $M/Fil^{a}M$ of an object over $\mathbb{A}_{\mr{cris}}(\tilde{R}/p)^{(\bullet)}$ induces the $p$-torsion freeness of the quotient of the corresponding object over $D$, as it seems difficult to descend this property. This is however easy to prove for the perfectoid cover.
\begin{prop}\label{left-inverse-perf}
The functor $\mathbb{D}_{R_{\infty}}$ restricts to a functor 
\[\mathbb{D}_{R_{\infty}}:\MFa^{\mr{tor-free}}(\mathbb{A}_{\mr{cris}}(R_{\infty}/p)^{(\bullet)})^{\mr{full}}\to \MFa^{\mr{tor-free}}(Spf(R)),\]
which is an equivalence of categories.
\end{prop}
\begin{proof}
In this case, for any $(M, Fil^{a}M, \epsilon, \Phi^{a}_{M})\in \MFa(D^{(\bullet)})$, the filtered base change to $\mathbb{A}_{\mr{cris}}(R_{\infty}/p)$ of the full filtration $Fil^{\bullet}_{M, Fil^{a}M}$ is simply the normal base change, i.e. the closure of $Fil_{M, Fil^{a}M}^{\bullet}\otimes \mathbb{A}_{\mr{cris}}(R_{\infty}/p)$ and this is seen to agree with the full filtration associated to $Fil^{a}M_{\mathbb{A}_{\mr{cris}}(R_{\infty}/p)}$. Using this, one sees that the filtered base-change functor is a quasi-inverse to $\mathbb{D}_{R_{\infty}}$.
\end{proof}
For the proof of our main theorem, we will have to make use of torsion objects. We therefore make the following definitions.
\begin{defn}\label{MF-tor}
Let $Spf(A)\in \mathcal{X}_{qsyn}$, where $A$ is a $p$-torsionfree qrsp $\mathcal{O}_{C}$-algebra. We then define $\MFa^{\mr{tor}}(\mathbb{A}_{\mr{cris}}(A/p)^{(\bullet)})$ to be the category of direct sums of quadruples $(M, Fil^{a}M, \Phi^{a}_{M}, \epsilon)$, where
\begin{itemize}
\item $M$ is an $\mathbb{A}_{\mr{cris}}(A/p)/p^{n}$-module, which is $A_{\mr{cris}}/p^{n}$-flat, for some $n\geq 1$.
\item $Fil^{a}M\subset M$ is a submodule, such that $J_{A}^{[a]}M\subset Fil^{a}M$ and $M/Fil^{a}M$ is $\Z_p/p^{n}$-flat.
\item $\Phi^{a}_{M}:Fil^{a}M\to M$ is a Frobenius-linear map, which induces an isomorphism
\begin{align}\label{Frob-acris-iso}
Fil^{a}M/J_{A}Fil^{a}M\otimes_{\phi_{\mathbb{A}_{\mr{cris}}(A/p)}}\mathbb{A}_{\mr{cris}}(A/p)/p\xrightarrow{\cong} M/pM.
\end{align}
Moreover, we ask that it satisfies the following condition
\begin{align}\label{cond-frob-2}
\Phi^{a}_{M}(xm)=c_{\xi}^{-a}\phi^{a}_{\mathbb{A}_{\mr{cris}}(\tilde{R}/p)}(x)\Phi^{a}_{M}(\xi^{a}m),
\end{align}
where $c_{\xi}:=\frac{\phi(\xi)}{p}$.
\item $\epsilon:p_{1}^{*}M\cong p_{2}^{*}M$ is an isomorphism, which satisfies $\epsilon(F^{a}(p_{1}^{*}M))=F^{a}(p_{2}^{*}M)$, for all $0\leq k \leq a$ and $\epsilon\circ \Phi_{p_{1}^{*}M}^{a}=\Phi_{p_{2}^{*}M}^{a}\circ \epsilon$, where again $F^{a}(p_{1}^{*}M)$ denotes the filtered-pullback with respect to the full filtration $Fil^{\bullet}_{M, Fil^{a}M}$ associated to $Fil^{a}M$.
\end{itemize}
Finally, we define the category $\MFa(\mathbb{A}_{\mr{cris}}(A/p)^{(\bullet)})$ to be the smallest category containing both $\MFa^{\mr{tor}}(\mathbb{A}_{\mr{cris}}(A/p)^{(\bullet)})$ and $\MFa^{\mr{tor-free}}(\mathbb{A}_{\mr{cris}}(A/p)^{(\bullet)})$ as full subcategories.
\end{defn}  
\begin{rmk}
Condition \ref{cond-frob-2} is always satisfied for objects contained in $\MFa^{\mr{tor-free}}(\mathcal{X})$. Indeed, if $M$ is $p$-torsion free, multiplying $\Phi^{a}(xm)$ with $\phi_{D}(\xi)^{a}$, we get $\phi(E)^{a}\Phi_{M}^{a}(xm)=\Phi^{a}(\xi^{a}xm)=\phi(x)\Phi^{a}(xm)$. Dividing by $p^{a}$ then gives $c_{\xi}^{a}\Phi^{a}(xm)=\phi^{a}_{D}(x)\Phi^{a}_{M}(\xi^{a}m)$.\\
Note also that \ref{cond-frob-2} does not depend on the distinguished element $\xi$ (i.e. if it holds for one, it also holds for any other).
\end{rmk}
\begin{rmk}
It is clear that mod $p^{n}$ reductions of objects in $\MFa^{\mr{tor-free}}$ lie in $\MFa^{\mr{tor}}$. In this paper, the torsion categories will only be used to give the proof of our main results. In general, it seems difficult to us to construct a category $\MFa^{\mr{tor}}(Spf(R))$, which is well-behaved and globalizes. The problem is to construct well-defined base-change functors in general.\\ One can define global categories of torsion objects if one restricts to so called filtered-free objects in the sense of Faltings (see \cref{rmk-examples} below).
\end{rmk}

We finish the section with some remarks on how the here described objects are related to earlier work by various authors.
\begin{rmk}\label{rmk-examples}
\begin{enumerate}
\item One can show that the category $\MFa^{\mr{tor-free}}(\mathcal{O}_{K})$ is equvivalent to the category of $p$-torsion free crystalline Breuil-modules. Recall that a Breuil module is given by a tuple $(M, Fil^{a}M, \phi_{M}, N_{M})$, where $M$ is a module over the Breuil ring $\mathcal{S}$, $Fil^{a}M\subset M$ is a submodule, such that $E(u)^{a}M\subset FilM$ and equipped with a divided Frobenius $\phi_{M}:FilM\to M$ and a monodromy operator $N_{M}:M\to M$, satisfying various conditions, similar to the ones in \cref{def-MF}. The Breuil module is called crystalline if $N_{M}(M)\subset (u\mathcal{S}+Fil^{1}\mathcal{S})M$.
\item In \cite{Fal-very} Faltings considered a more restrictive version of the category $\MFa^{\mr{tor-free}}(\mathcal{X})$ by only considering so called locally filtered-free objects: \\Consider the following category $\mathcal{M}\mathcal{F}^{a, \mr{fil-free}}(\mathcal{X})$, whose objects are given by
\begin{itemize}
\item A finite locally free crystal $\mathcal{E}$ equipped with a decreasing exhaustive filtration $\{Fil^{i}(\mathcal{E})\}_{i\in \mathbb{N}}$, such that $(\mathcal{E}, Fil^{\bullet}(\mathcal{E}))$ is a filtered module over $\mathcal{O}_{\mathcal{X}_{\mr{cris}}}$, and such that there exists locally an isomorphism of filtered modules $\mathcal{E}\cong \bigoplus\limits_{n=0}^{N}\mathcal{O}_{\mathcal{X}_{\mr{cris}}}(q_{n})$, where $0\leq q_{n} \leq a$ and $\mathcal{O}_{\mathcal{X}_{\mr{cris}}}(q_{n})$ denotes the trivial module $\mathcal{O}_{\mathcal{X}_{\mr{cris}}}$, with filtration shifted by $q_{n}$. \\
\item For any $(B, J, \phi_{B}, f)$ as in the previous definition, we have Frobenius maps 
\begin{center}
$\Phi^{i}:\varprojlim_{n}Fil^{i}\mathcal{E}(B/J^{n})\to \varprojlim_{n}\mathcal{E}(B/J^{n})$, 
\end{center}
whose image generates all of $\varprojlim_{n}\mathcal{E}(B/J^{n})$.
\end{itemize}
There is a functor 
\begin{align*}
\mathcal{M}\mathcal{F}^{a, \mr{fil-free}}(\mathcal{X}) & \longrightarrow &  \MFa^{\mr{tor-free}}(\mathcal{X})\\
(\mathcal{E}, Fil^{\bullet}(\mathcal{E}), \{\Phi^{i}\}_{i}) & \longmapsto & (\mathcal{E}, Fil^{a}\mathcal{E}, \Phi^{a}).
\end{align*}
One can show that this functor is fully faithful.
\item Assume that $K$ is totally unramified, i.e. $W(k)=\mathcal{O}_{K}$. There then exists the category of $p$-torsion free Faltings-Fontaine-Laffaille modules (see \cite{Font-Laff}, \cite{Fal-cris}, and in particular \cite[\S 4]{Tsuji-mf}) which we here denote by $\MFa(\mathcal{X})'$, which in the small affine case (i.e. $\mathcal{X}=Spf(R)$) consists of $(M, \nabla, \{F^{i}_{M}\}_{i}, \{\Phi^{i}_{M}\}_{i})$, with $0\leq i \leq a$. Here $M$ is a direct sum $M=R^{n}$ equipped with an integrable connection $\nabla$ and $F^{\bullet}_{M}$ is a decreasing exhaustive filtration of direct factors, which satisfies Griffiths transversality with respect to $\nabla$ and $\Phi^{i}:F^{i}_{M}\to M$ are horizontal divided Frobenius maps (for some chosen Frobenius lift on $R$), satisfying certain properties.\\
Let $D=R\hat{\otimes}_{W(k)} \mathcal{S}$ be the $p$-adically completed pd hull of $R[[u]]\twoheadrightarrow R$. Then by \cref{indep-mf}, $\MFa(Spf(R))$ is canonically equivalent to the category $\MFa(Spf(R))^{D}$. There is a functor $\MFa(Spf(R))'\to \MFa(Spf(R))^{D}$, which takes $(M, \nabla, \{F^{i}_{M}\}_{i}, \{\Phi^{i}_{M}\}_{i})$ to the quadruple $(M_{D}, \nabla_{M_{D}}, Fil^{a}M_{D}, \Phi^{a}_{M_{D}})$ defined by 
\begin{itemize}
\item $M_{D}:=M\otimes_{R} D$.
\item $\nabla_{M_{D}}=\nabla\otimes id + id\otimes d$, where $d:D\to  \hat{\Omega}^{1}_{D/\Z_p}$ is the universal PD-differential.
\item $Fil^{a}M_{D}:=\sum\limits_{r+s=a} F^{r}M\otimes J_{R}^{[s]}$.
\item $\Phi^{a}_{M_{D}}:=\sum\limits_{r+s=a} \Phi^{r}_{M}\otimes \phi^{s}_{D}$.
\end{itemize}
One can again prove that this functor is fully faithful. Moreover, it glues to a functor $\MFa(\mathcal{X})'\to \MFa(\mathcal{X})$ for global $\mathcal{X}$.
\end{enumerate}
\end{rmk}

\section{Divided prismatic F-crystals}
Throughout the section, we fix a smooth $p$-adic formal scheme $\mathcal{X}$ over $\mathcal{O}_{K}$.
\subsection{The Nygaard filtration of completed prismatic F-crystals}
We start out by recalling some concepts from \cite{DLMS}. 
\begin{defn}\cite[Def. 3.16]{DLMS}\label{completed-crystals-pris-site}
A $p$-torsion free completed prismatic $F$-crystal of height $r\geq 0$ on $\mathcal{X}$, is a completed crystal $\mathcal{M}$ in quasi-coherent $\mathcal{O}_{\prism}$-modules, together with a Frobenius-linear map $\Phi_{\mathcal{M}}:\mathcal{M}\to \mathcal{M}$, such that
\begin{itemize}
\item $\mathcal{M}_{(\mathfrak{S}(R), E)}$ is finitely generated, torsion free and saturated, meaning
\begin{center}
$\mathcal{M}(\mathfrak{S}(R))=\mathcal{M}(\mathfrak{S}(R))\left[\frac{1}{p}\right]\cap \mathcal{M}(\mathfrak{S}(R))\left[\frac{1}{E}\right]$.
\end{center}
\item The map $\phi_{\mathfrak{S}(R)}^{*}\mathcal{M}(\mathfrak{S}(R))\to \mathcal{M}(\mathfrak{S}(R))$ is injective and its cokernel is $E^{r}$-torsion.
\end{itemize}
For a general smooth formal $p$-adic schemes $\mathcal{X}$, a prismatic F-crystal is a pair $(\mathcal{M}, \Phi_{\mathcal{M}})$, which satisfies the above conditions Zariski locally.\\
The category of $p$-torsion free completed prismatic F-crystals is denoted by $CR^{[\wedge, \phi, \mr{tor-free}]}_{[0, r]}(\mathcal{X}_{\prism})$.
\end{defn}
\begin{rmk}
\begin{itemize}
\item In \cite{DLMS}, a completed Frobenius crystal is moreover assumed to be projective away from $(p, E)$. This is in fact automatic in our case (see \cref{proj-away}).
\item By \cite[Rmk. 3.16]{DLMS}, $(\mathcal{M}, \Phi_{\mathcal{M}})$ is saturated if and only if $\mathcal{M}(\mathfrak{S}(R))/p$ is $E$-torsion free. In particular this means that $\mathcal{M}(\mathfrak{S}(R))$ is saturated if and only if $\mathcal{M}(\mathfrak{S}(R))$ is flat over $\mathfrak{S}$ which in turn is equivalent to asking that $(p, E)$ (or again equivalently $(E, p)$) is a regular sequence for $\mathcal{M}(\mathfrak{S}(R))$.
\item Let $R\to R_{\infty}$ be the usual perfectoid cover. Then using the faithfully flat map $\mathfrak{S}(R)\to \prism_{R_{\infty}}$, one sees that the conditions in the definition above may also be replaced by asking that $\mathcal{M}(\prism_{R_{\infty}})$ is finitely generated, torsion free and saturated (i.e. $(p, I)$-completely flat over $A_{inf}$) over $\prism_{R_{\infty}}$ and that the linearization of Frobenius over $\prism_{R_{\infty}}$ is injective with $\tilde{\xi}^{a}$-torsion cokernel. 
\end{itemize}
\end{rmk}

\begin{defn}\label{completed-crys}
The category $CR^{[\wedge, \phi, \mr{tor-free}]}_{[0, r]}(\mathcal{X})$ consists of pairs $(\mathcal{M}, \Phi_{\mathcal{M}})$, where $\mathcal{M}$ is a completed crystal of $\mathcal{O}^{\mr{pris}}$-modules on $\mathcal{X}_{\mr{qsyn}}$ and $\Phi_{\mathcal{M}}$ is a Frobenius-linear map, such that for any $Spf(B)\in \mathcal{X}_{\mr{qsyn}}$ with $B$ qrsp,
\begin{itemize}
\item $\mathcal{M}(B)$ is a finitely generated torsion free module over $\prism_{B}$ which is saturated, meaning
\begin{center}
$\mathcal{M}(B)=\mathcal{M}(B)\left[\frac{1}{p}\right]\cap \mathcal{M}(B)\left[\frac{1}{I_{B}}\right]$,
\end{center}
\item the map $\phi_{\prism_{B}}^{*}\mathcal{M}(B)\to \mathcal{M}(B)$ is injective and $I_{B}^{r}$-torsion.
\end{itemize}
\end{defn}
We want to prove that the two categories introduced above are in fact equivalent.
\begin{lem}
Assume that $\mathcal{X}=Spf(R)$ is small and let $(\mathcal{M}, \Phi_{\mathcal{M}})\in CR^{[\wedge, \phi, \mr{tor-free}]}_{[0, r]}(\mathcal{X}_{\prism})$. Let $Spf(B)\in \mathcal{X}_{\mr{qsyn}}$ with $B$ qrsp, and let $\tilde{\xi}$ be a generator of $I\subset \prism_{B}$. Then $\mathcal{M}(\prism_{B})$ is saturated, i.e.
\begin{center}
$\mathcal{M}(B)=\mathcal{M}(B)[\tilde{\xi}^{-1}]\cap \mathcal{M}(B)[p^{-1}]$
\end{center}
 and $(\tilde{\xi}, p)$ is a regular sequence for $\mathcal{M}(\prism_{B})$.
\end{lem}
\begin{proof}
Let $R\to R_{\infty}$ be the perfectoid ring, obtained by extracting $p$-th power roots of \'etale coordinates and base extending to $\mathcal{O}_{C}$. Assume first that $B$ lies over $R_{\infty}$, i.e. there is $p$-completely faithfully flat map $R_{\infty}\to B$ compatible with the structure maps from $R$. Then $\prism_{B}$ is flat over $\mathfrak{S}(R)$, since $\mathfrak{S}(R)\to \prism_{R_{\infty}}\to \prism_{B}$ is a composition of $(p, E)$-completely flat maps and $\mathfrak{S}(R)$ is noetherian. But then $\mathcal{M}(\prism_{B})=\mathcal{M}(\mathfrak{S}(R))\hat{\otimes}\prism_{B}$ is saturated, by \cite[Lemma 3.19]{DLMS}, so $(\tilde{\xi}, p)$ is a regular sequence for $\mathcal{M}(\prism_{B})$.\\
For a general $B$, we consider $B_{\infty}=B\hat{\otimes}_{R}R_{\infty}$. Then $(\tilde{\xi}, p)$ is a regular sequence for $\mathcal{M}(B_{\infty})=\mathcal{M}(B)\hat{\otimes}_{\prism_{B}}\prism_{B_{\infty}}$, i.e. $\mathcal{M}(B_{\infty})/I=\mathcal{M}(B)/I\hat{\otimes}\prism_{B_{\infty}}/I$ is $p$-torsion free. But then, since $\prism_{B}/I\to \prism_{B_{\infty}}/I$ is $p$-completely faithfully flat, one sees that $\mathcal{M}(B)/I$ is $p$-torsion free as well (namely, by faithful flatness, one deduces that $\mathcal{M}(B)/(I, p^{n})$ is $\Z_p/p^{n}$ flat, for all $n\geq 1$).
\end{proof}

The lemma above implies the following.
\begin{cor}
The functor $\mu_{*}$ (see \cref{equiv-qsyn}) induces a natural equivalence of categories
\begin{center}
$CR^{[\wedge, \phi, \mr{tor-free}]}_{[0, r]}(\mathcal{X}_{\prism})\longleftrightarrow  CR^{[\wedge, \phi, \mr{tor-free}]}_{[0, r]}(\mathcal{X})$.
\end{center}
\end{cor}
We will also consider torsion objects.
\begin{defn}
We denote by $CR^{[\wedge, \phi, tor]}_{[0, r]}(\mathcal{X})$ the category (in the obvious big category of all prismatic crystals with Frobenius linear map) whose objects are direct sums of Frobenius-modules $(\mathcal{M}, \Phi)$, for which the following hold.
\begin{itemize}
\item $\mathcal{M}$ is $\Z_p/p^{n}$-flat for some $n\geq 1$.
\item For any $Spf(B)\in \mathcal{X}_{\mr{qsyn}}$, with $B$ qrsp, both $\mathcal{M}(B)$and $\mathcal{M}(B)/p$ are $I_{B}$-torsion free.  
\item The pair $(\mathcal{M}, \Phi)$ is admissible, i.e. $\phi^{*}\mathcal{N}_{M}^{r}\xrightarrow{\cong} \mathcal{I}^{r}_{\prism}\mathcal{M}$ is an isomorphism. Here $\mathcal{N}^{r}_{\mathcal{M}}$ denotes the Nygaard filtration (see below)
\item For any $Spf(B)\in \mathcal{X}_{\mr{qsyn}}$, with $B$ qrsp, the linearization of $\Phi_{\mathcal{M}(B)}$ is injective and its cokernel is $I_{B}^{r}$-torsion.
\end{itemize}
Finally, we denote by $CR^{[\wedge, \phi]}_{[0, r]}(\mathcal{X})$ the smallest category, containing both $CR^{[\wedge, \phi, tor]}_{[0, r]}(\mathcal{X})$ and $CR^{[\wedge, \phi, \mr{tor-free}]}_{[0, r]}(\mathcal{X})$ as full subcategories.
\end{defn}
\begin{rmk}
\begin{itemize}
\item The first two conditions above are equivalent to asking that $\mathcal{M}(R_{\infty})$ is flat over $A_{inf}/p^{n}$.
\item Since we are also considering torsion-objects, we deviate here a little bit from the notation of \cite{DLMS}. In loc. cit., the category $CR^{[\wedge, \phi, \mr{tor-free}]}_{[0, r]}(\mathcal{X}_{\prism})$ is denoted by $CR^{[\wedge, \phi]}_{[0, r]}(\mathcal{X}_{\prism})$.\\
Note that any object in $CR^{[\wedge, \phi, \mr{tor-free}]}_{[0, r]}(\mathcal{X})$ is automatically admissible, by \cref{prop-admissible}.
\end{itemize}
\end{rmk}
Objects in $CR^{[\wedge, \phi]}_{[0, r]}(\mathcal{X})$ are always projective away from $(I, p)$.
\begin{lem}\label{proj-away}
Let $(\mathcal{M}, \Phi_{\mathcal{M}})\in CR^{[\wedge, \phi, p-tor]}_{[0, r]}(\mathcal{X}_{\prism})$. Then $M=\mathcal{M}(\mathfrak{S}(R))$ is projective away from $E$.\\
If $(\mathcal{M}, \Phi_{\mathcal{M}})\in CR^{[\wedge, \phi, \mr{tor-free}]}_{[0, r]}(\mathcal{X}_{\prism})$, then $M=\mathcal{M}(\mathfrak{S}(R))$ is projective away from $(p, E)$, i.e. $M[p^{-1}]$ (resp. $M[E^{-1}]_{p}^{\wedge}$) is a projective $\mathfrak{S}(R)[p^{-1}]$-module (resp. a projective $\mathfrak{S}(R)[E^{-1}]_{p}^{\wedge}$-module).
\end{lem}
\begin{proof}
The second statement follows from \cite[Prop. 4.13]{DLMS} (which gives projectivity away from $p$) and from the first statement (which gives projectivity away from $E$).\\
The first statement follows from the arguments in \cite[\S2]{Fal-MF-st} (see also the proof of \cref{frob-mod-proj}), which in particular show that a module with invertible Frobenius over a regular local ring of characteristic $p$ is projective.
\end{proof}
\begin{cor}
Let $(\mathcal{M}, \Phi_{\mathcal{M}})\in CR^{[\wedge, \phi, tor]}_{[0, r]}(\mathcal{X})$ be $p$-torsion. Then for any qrsp $Spf(B)\in \mathcal{X}_{\mr{qsyn}}$, the $\prism_{B}[\frac{1}{I_{B}}]/p$-module $\mathcal{M}(B)[\frac{1}{I_{B}}]$ is projective.\\
If $(\mathcal{M}, \Phi_{\mathcal{M}})\in CR^{[\wedge, \phi, \mr{tor-free}]}_{[0, r]}(\mathcal{X}_{\prism})$, then $M=\mathcal{M}(B)$ is projective away from $(p, I_{B})$.
\end{cor}
\begin{proof}
This follows from the previous lemma together with \cite[Theorem 1.6]{Mat}.
\end{proof}
For any completed prismatic F-crystal $(\mathcal{M}, \Phi_{\mathcal{M}})$, we have a Nygaard filtration, defined as 
\begin{center}
$\mathcal{N}_{\mathcal{M}}^{i}:=\Phi_{\mathcal{M}}^{-1}(\mathcal{I}_{\prism}^{i}\mathcal{M})$.
\end{center}
\begin{prop}\label{prop-fil-crys}
Let $(\mathcal{M}, \Phi_{\mathcal{M}})\in CR^{[\wedge, \phi, \mr{tor-free}]}_{[0, r]}(\mathcal{X})$ and let $B\to C$ be a morphism of qrsp rings in $\mathcal{X}_{\mr{qsyn}}$, satisfying one of the following conditions:
\begin{enumerate}
\item The map $B\to C$ is $p$-completely faithfully flat. 
\item $B=R_{\infty}$, where $R\to R_{\infty}$ is the perfectoid covering defined in \cref{coverings}. 
\end{enumerate}
Then we have
\begin{center}
$\mathcal{N}_{\mathcal{M}}^{k}(C)=\sum\limits_{r+s=k}\mathcal{N}_{\mathcal{M}}^{r}(B)\hat{\otimes} N_{\prism_{C}}^{s}$
\end{center}
for any $k\geq 1$, under the canonical identification $\mathcal{M}(B)\hat{\otimes}_{\prism_{B}}\prism_{C}\xrightarrow{\cong} \mathcal{M}(C)$.\\
By $\sum\limits_{r+s=k}\mathcal{N}_{\mathcal{M}}^{r}(B)\hat{\otimes} N_{\prism_{C}}^{s}$ we here denote the closure of the image of $\sum\limits_{r+s=k}\mathcal{N}_{\mathcal{M}}^{r}(B)\otimes N_{\prism_{C}}^{s}\to \mathcal{M}(B)\hat{\otimes}_{\prism_{B}}\prism_{C}$
\end{prop}
\begin{proof}
First, assume that $B\to C$ is $p$-completely faithfully flat. Set $M:=\mathcal{M}(B)$. By definition $\sum\limits_{r+s=k}\mathcal{N}_{M}^{r}\hat{\otimes} N_{\prism_{C}}^{s}$ is the closure of $\sum\limits_{r+s=k}\mathcal{N}_{M}^{r}\otimes N_{\prism_{C}}^{s}$ in $\mathcal{M}(C)$. Also, $\mathcal{N}^{k}_{M_{\prism_{C}}}$ is closed in $\mathcal{M}(C)$, since it is the preimage of the closed submodule $\tilde{\xi}^{k}\mathcal{M}(C)$ under the Frobenius map, which is continuous. Therefore, we have 
\begin{center}
$\sum\limits_{r+s=k}\mathcal{N}_{M}^{r}\hat{\otimes} N_{\prism_{C}}^{s}\subset \mathcal{N}^{k}_{M_{\prism_{C}}}$
\end{center}
and want to prove that this is even an equality. To simplify notation, we write $Fil^{k}M_{\prism_{C}}:=\sum\limits_{r+s=k}\mathcal{N}_{M}^{r}\hat{\otimes} N_{\prism_{C}}^{s}$.
\begin{claim}
We have $Fil^{k}M_{\prism_{C}}[\tilde{\xi}^{-1}]_{p}^{\wedge}=\mathcal{N}^{k}_{M_{\prism_{C}}}[\tilde{\xi}^{-1}]_{p}^{\wedge}$ 
and $Fil^{k}M_{\prism_{C}}[p^{-1}]=\mathcal{N}^{k}_{M_{\prism_{C}}}[p^{-1}]$.
\end{claim}
The first part of the claim is easy, since $N_{M}^{k}[\tilde{\xi}^{-1}]=M$.\\
For the second part, note that the $\prism_{B}$-module $M[p^{-1}]$ is finite projective, so after localizing we may even assume that it is free on a basis $v_{1},\hdots , v_{n}$. Let $H$ be the matrix representing $\Phi_{M[p^{-1}]}$ (the linear extension of $\Phi_{M}$ to  $M[p^{-1}]$). Let $m=(m_{1}, \hdots, m_{n})\in M[p^{-1}]\hat{\otimes}\prism_{C}[p^{-1}]\cong \prism_{C}[p^{-1}]^{n}$ be an element such that \\
$\Phi(m)=f(H)(\phi_{\prism_{C}[p^{-1}]}(m_{1}), \hdots, \phi_{\prism_{C}[p^{-1}]}(m_{n}))^{t}=\sum\limits_{i}f(Hv_{i})\phi(m_{i})$ is divisible by $\tilde{\xi}^{k}$.\\
Since $H$ is invertible up to a power of $\tilde{\xi}$, the elements $f(Hv_{i})$ are linearly independent. Therefore $\tilde{\xi}^{k}$ must divide $f(Hv_{i})\phi(m_{i})$, for all $i$. Writing $Hv_{i}=\sum\limits_{j=1}^{n}h_{ij}v_{j}$, we see that $\tilde{\xi}^{k}$ must divide $f(h_{ij})\phi(m_{i})$, for all $i$ and $j$. We now claim that $f(h_{ij})$ is divisible by some power $\tilde{\xi}^{s}$, if and only if $\tilde{\xi}^{s}$ divides $h_{ij}$. Namely, $f$ mod $\tilde{\xi}$ is injective because $\prism_{B}/\tilde{\xi}\to \prism_{C}/\tilde{\xi}$ is $p$-completely faithfully flat (this means that the map is injective modulo $p^{n}$ for all $n\geq 1$, hence injective itself, since $\prism_{B}/p$ is classically $p$-complete because $\prism_{B}$ is bounded). Let then $r_{i}\in \N$ be the largest number for which $\tilde{\xi}^{r_{i}}$ divides $\phi(m_{i})$. If $r_{i} < k$, then $h_{ij}$ must be divisible by $\tilde{\xi}^{k-r_{i}}$, for all $j$. But this then implies that $m$ lies in $Fil^{k}M_{\prism_{C}}[p^{-1}]$. This proves the claim.\\
Now, $M_{\prism_{C}}$ is saturated, i.e. 
\begin{center}
$M_{\prism_{C}}=M_{\prism_{C}}\left[\frac{1}{\tilde{\xi}}\right]\cap M_{\prism_{C}}[p^{-1}]$.
\end{center}
Moreover, we then also have $\tilde{\xi}^{k}M_{\prism_{C}}=\tilde{\xi}^{k}M_{\prism_{C}}\left[\frac{1}{\tilde{\xi}}\right]\cap \tilde{\xi}^{k}M_{\prism_{C}}[p^{-1}]$, and hence see that $\mathcal{N}^{k}_{M_{\prism_{C}}}$ is saturated as well. Similarly, one easily checks that $Fil^{k}M_{\prism_{C}}$ is also saturated. \\
Hence, the above claim implies that
\begin{center}
$\mathcal{N}_{\mathcal{M}}^{k}(C)=\sum\limits_{r+s=k}\mathcal{N}_{\mathcal{M}}^{r}(B)\hat{\otimes} N_{\prism_{C}}^{s}$.
\end{center}
Now for (2), we first consider the morphism $g:\mathfrak{S}(R)\to \prism_{R_{\infty}}\to \prism_{C}$. Since $R$ is an integral domain and $R\to C$ is classically flat ($\mathfrak{S}(R)$ being noetherian), we again get that $g$ is injective modulo $I$. Thus, we may again prove as above that we have 
\begin{center}
$\mathcal{N}_{\mathcal{M}}^{k}(\prism_{A})=\sum\limits_{r+s=k}\mathcal{N}_{\mathcal{M}}^{r}(\mathfrak{S}(R))\hat{\otimes} N_{\prism_{A}}^{s}$,
\end{center}
for either $A=R_{\infty}$ or $A=C$ - where by abuse of notation, we denote by $\mathcal{M}(\mathfrak{S}(R))$ the value of the crystal on the prismatic site associated to $\mathcal{M}$. But then, it also follows easily that
\begin{center}
$\mathcal{N}_{\mathcal{M}}^{k}(C)=\sum\limits_{r+s=k}\mathcal{N}_{\mathcal{M}}^{r}(R_{\infty})\hat{\otimes} N_{\prism_{C}}^{s}$.
\end{center}
\end{proof}
For torsion objects, we record the following.
\begin{prop}\label{prop-fil-crys-tor}
Let $(\mathcal{M}, \Phi_{\mathcal{M}})\in CR^{[\wedge, \phi, tor]}_{[0, r]}(\mathcal{X})$ be locally free over $\mathcal{O}^{\mr{pris}}/p^{n}$, for some $n\geq 1$. Let $B\to C$ be a map of qrsp rings in $\mathcal{X}_{\mr{qsyn}}$, which is either $p$-completely faithfully flat or $B=R_{\infty}$. Then for any $k\geq 0$, we have
\begin{center}
$\mathcal{N}_{\mathcal{M}}^{k}(C)=\sum\limits_{r+s=k}\mathcal{N}_{\mathcal{M}}^{r}(B)\hat{\otimes} N_{\prism_{C}}^{s}$
\end{center}
under the canonical identification $\mathcal{M}(\prism_{B})\hat{\otimes}_{\prism_{B}}\prism_{C}\xrightarrow{\cong} \mathcal{M}(\prism_{C})$.
\end{prop}
\begin{proof}
The proof of the claim in the previous proposition for the crystal $\mathcal{M}[p^{-1}]$ works also in this case.
\end{proof}
We also have the following result, which we will need later.
\begin{lem}\label{quot-flatness}
Let $(\mathcal{M}, \Phi_{\mathcal{M}})\in CR^{[\wedge, \phi, \mr{tor-free}]}_{[0, r]}(\mathcal{X})$ and let $Spf(B)\in \mathcal{X}_{\mr{qsyn}}$, with $B$ qrsp. Then $\mathcal{M}(\prism_{B})/\mathcal{N}_{\mathcal{M}}^{k}(\prism_{B})$ is $p$-torsion free, for all $k\geq 0$.\\
Moreover, let $(\mathcal{M}, \Phi_{\mathcal{M}})\in CR^{[\wedge, \phi, tor]}_{[0, r]}(\mathcal{X})$, such that $\mathcal{M}$ is $\Z_p/p^{n}$-flat, for some $n\geq 1$ and let $B$ now be a perfectoid ring, quasi-syntomic over $R$. Then $\mathcal{M}(\prism_{B})/\mathcal{N}_{\mathcal{M}}^{k}(\prism_{B})$ is $\Z_p/p^{n}$-flat as well.
\end{lem}
\begin{proof}
Let $\tilde{\xi}\in I\subset \prism_{B}$ be a generator and let $M:=\mathcal{M}(\prism_{B})$. Now, if $M/\mathcal{N}_{M}^{r}$ weren't $p$-torsion free, there would exist an $x\in M\backslash \mathcal{N}_{M}^{r}$, such that $px\in \mathcal{N}_{M}^{r}$. Thus $p\Phi_{M}(x)=\tilde{\xi}^{k}y$, for some $y\in M$. But $M$ is saturated, so this implies that $\tilde{\xi}^{k}$ must divide $\Phi_{M}(x)$, a contradiction.\\
For the second part, let now $(\mathcal{M}, \Phi_{\mathcal{M}})\in CR^{[\wedge, \phi, tor]}_{[0, r]}(\mathcal{X})$ and $B$ be perfectoid, such that $\mathcal{M}$ is $\Z_p/p^{n}$-flat and assume that $M/\mathcal{N}_{M}^{k}$ is not $\Z_p/p^{n}$-flat. This means the map $M/(p, \mathcal{N}^{k}_{M})\xrightarrow{\cdot p^{n-1}}p^{n-1}M/\mathcal{N}^{k}_{M}$ is not injective. Thus there exists some non-zero $x\in M/p$ which is not in the image of $\mathcal{N}_{M}^{k}$, such that $p^{n-1}\tilde{x}\in \mathcal{N}_{M}^{k}$, for some lift $\tilde{x}\in M$. But then, as above, we get $p^{n-1}\Phi(\tilde{x})=\tilde{\xi}^{k}y$, for some $y\in M$. 

Since $M$ and $M[\tilde{\xi}^{-1}]$ are $\Z_p/p^{n}$-flat, taking $Tor^{\Z_p}(- ,\F_p)$ yields the following commutative diagram with exact rows
\begin{center}
$\xymatrix{
M\ar@{^{(}->}[r] & M[\tilde{\xi}^{-1}]\ar[r] & M[\tilde{\xi}^{-1}]/M\ar[r] & 0\\
(p^{n-1})M\ar@{^{(}->}[r]\ar@{^{(}->}[u] & (p^{n-1})M[\tilde{\xi}^{-1}]\ar[r]\ar@{^{(}->}[u] & M[\tilde{\xi}^{-1}]/M[p]\ar@{^{(}->}[u]\ar[r] & 0
}$
\end{center}
This shows that we have $(p^{n-1})M=(p^{n-1})M[\tilde{\xi}^{-1}]\cap M$. It then follows that $(p^{n-1})\frac{\Phi(\tilde{x})}{\tilde{\xi}^{k}}\in (p^{n-1})M$, so there is some $z\in M$, with $p^{n-1}z=(p^{n-1})\frac{\Phi(\tilde{x})}{\tilde{\xi}^{k}}$. But we thus have $z= \frac{\Phi(\tilde{x})}{\tilde{\xi}^{k}}$ mod $p$, which gives that $\frac{\Phi(x)}{\tilde{\xi}^{k}}\in M/p$. Now since $B$ is perfectoid, it is easy to see that the Nygaard filtration of $M/p$ is the mod $p$ reduction of the Nygaard filtration $\mathcal{N}_{M}^{k}$ of $M$. Hence, the fact that $\frac{\Phi(x)}{\tilde{\xi}^{k}}\in M/p$ contradicts the assumption $x\notin Image(\mathcal{N}_{M}^{k}\to M/p)$.
\end{proof}

\begin{lem}\label{qrsp-adm}
Let $R$ be a $p$-torsion free quasi-regular semiperfectoid $\Z_p$-algebra. Let $i\geq 0$. Then the morphism $\phi_{\prism_{R}}^{*}N_{\prism}^{i}\to I^{a}\prism_{R}$, induced by the linearization of $\phi_{\prism_{R}}$, is an isomorphism. 
\end{lem}
\begin{proof}
As the surjectivity is obvious, we only need to show that $\phi_{\prism_{R}}^{*}N_{\prism}^{i}\to \prism_{R}$ is injective. Let $K$ be the kernel. Then $K$ is (derived) $(p, I)$-complete, so by Nakayama's lemma, it is enough to prove that $K/p=0$. Moreover, since $\prism_{R}$ is $p$-torsion free, $K/p$ is the kernel of $\phi_{\prism_{R}}^{*}N_{\prism}^{i}/p\to \prism_{R}/p$. We now prove that $Tor_{1}^{\prism_{R}/p}(\prism_{R}^{\phi}/p, \prism_{R}/(N_{\prism}^{i}, p))=0$, where $\prism_{R}^{\phi}$ denotes $\prism_{R}$ equipped with the module structure induced by $\phi_{\prism_{R}}$. We proceed by induction on $i$. First assume $i=1$. Then $\prism_{R}/N_{\prism_{R}}^{1}=R$. Moreover, base extending along the faithfully flat map (see \cite[Lemma 7.10]{BS-pris}) $R/p\to \prism_{R}/(p, I)$ yields \begin{center}
$Tor^{\prism_{R}/p}_{1}(\prism_{R}^{\phi}/p, R/p)\otimes_{R/p}\prism_{R}/(I, p)=Tor^{\prism_{R}/p}_{1}(\prism_{R}^{\phi}/p, \prism_{R}/(I, p))$,
\end{center}
and so it is enough to prove that the latter is zero. Now $\phi_{\prism}^{*}(I)=I\otimes_{\phi}\prism$ is a trivial rank $1$ module by \cite[Lemma 3.6]{BS-pris}, and a generator is given by an element $f$ which maps to a generator $\phi(d)$ of $\phi(I)\prism_{R}$. But then, since $\prism_{R}/p$ is $\phi(d)$-torsion free, the map $\phi_{\prism}^{*}(I)/p\to \phi(I/p)\prism_{R}$ is an isomorphism, which yields $Tor^{\prism_{R}/p}_{1}(\prism_{R}^{\phi}/p, \prism_{R}/(I, p))=0$. This in turn gives $Tor^{\prism_{R}/p}_{1}(\prism_{R}^{\phi}/p, R/p)=0$ - in fact we even have $Tor^{\prism_{R}/p}_{j}(\prism_{R}^{\phi}/p, R/p)=0$, for all $j > 0$, since $\prism_{R}/(I, p)$ is of projective dimension $1$.\\
Now, for the case $i > 0$, we first remark that the sequence
\begin{center}
$0\to gr^{i}_{N}/p\to \prism_{R}/(N^{i+1}, p)\to \prism_{R}/(N^{i}, p)\to 0$
\end{center}
is exact. This is because $\prism_{R}/N^{i}$ is $p$-torsion free, which in turn follows by d\'evissage from the $p$-torsion freeness of both $\prism_{R}/N^{1}=R$ and $gr^{i}_{N}\cong Fil^{i}\bar{\prism}_{R}\{i\}$ (see \cite[Theorem 12.2]{BS-pris}).\\
Now, consider the Tor-spectral sequence associated to the projection map $\prism_{R}/p\to R/p$
\begin{center}
$Tor_{n}^{R/p}(Tor^{\prism_{R}/p}_{m}(\prism_{R}^{\phi}/p, R/p), gr_{N}^{i}/p)\Longrightarrow Tor^{\prism_{R}/p}_{n+m}(\prism_{R}^{\phi}/p, gr^{i}_{N}/p)$.
\end{center}
Since $Tor^{\prism_{R}/p}_{j}(\prism_{R}^{\phi}/p, R/p)=0$, for all $j > 0$, this collapses to 
\[Tor_{m}^{R/p}(\prism_{R}/(I, p), gr^{i}_{N}/p)\cong Tor^{\prism_{R}/p}_{m}(\prism^{\phi}/p, gr^{i}_{N}/p).\]
But since $R/p\to \prism_{R}/(I, p)$ is faithfully flat, this means that $Tor^{\prism_{R}/p}_{m}(\prism^{\phi}/p, gr^{i}_{N}/p)=0$, for all $m > 0$. By d\'evissage, one thus obtains $Tor^{\prism_{R}/p}_{1}(\prism^{\phi}/p, \prism_{R}/N^{i}_{\prism}/p)=0$, for all $i > 0$.
\end{proof}

\begin{prop}\label{prop-admissible}
Let $(\mathcal{M}, \Phi_{\mathcal{M}})\in CR^{[\wedge, \phi, \mr{tor-free}]}_{[0, a]}(\mathcal{X})$. Then the morphism $\phi^{*}\mathcal{N}_{\mathcal{M}}^{a}\to \mathcal{I}_{\prism}^{a}\mathcal{M}$ is an isomorphism.
\end{prop}
\begin{proof}
We may assume that $\mathcal{X}=Spf(R)$ is small. We first prove that $\phi^{*}\mathcal{N}_{\mathcal{M}}^{a}\to \mathcal{I}_{\prism}^{a}\mathcal{M}$ is an epimorphism. Let $R\to R_{\infty}$ be a quasi-syntomic covering by a perfectoid ring. We will show that $\phi_{A_{\mr{inf}}(R_{\infty})}^{*}N_{M}^{r}\to \xi^{r}M$ is surjective. Since $coker(\phi^{*}M\to M)$ is $\xi^{r}$-torsion, we know that any element $m\in \xi^{r}M$ may be written as $m=\sum\limits_{i}\alpha_{i}\Phi_{M}(m_{i})$, for some $\alpha_{i}\in A_{\mr{inf}}(R_{\infty})$. But since $\phi_{A_{\mr{inf}}(R_{\infty})}$ is an isomorphism, we have 
\begin{center}
$m=\sum\phi_{A_{\mr{inf}}(R_{\infty})}^{-1}(\alpha_{i})\Phi_{M}(m_{i})=\Phi_{M}(\sum\phi_{A_{\mr{inf}}(R_{\infty})}^{-1}(\alpha_{i})m_{i})$.
\end{center}
Now let $Spf(B)\in \mathcal{X}_{\mr{qsyn}}$ be some object with $B$ a qrsp $\Z_p$-algebra. Assume that there is a map $R_{\infty}\to B$ in $\mathcal{X}_{\mr{qsyn}}$. We have $I_{R_{\infty}}M_{\prism_{B}}=(I_{A_{\mr{inf}}(R_{\infty})}M)\hat{\otimes}\prism_{B}$ and by \cref{prop-fil-crys} (2), $N^{a}_{M_{\prism_{B}}}=\sum\limits_{r+s=a}N^{r}_{M}\otimes N^{s}_{\prism_{B}}$. This gives surjectivity of $\phi^{*}N^{a}_{M_{\prism_{B}}}\to I_{\prism_{B}^{a}}M(\prism_{B})$. We thus see that the map $\phi^{*}\mathcal{N}_{\mathcal{M}}^{a}\to \mathcal{I}_{\prism}^{a}\mathcal{M}$ becomes an epimorphism after restriction to $R_{\infty}$.\\

Let now $Spf(B)\in \mathcal{X}_{\mr{qsyn}}$ be an arbitrary qrsp object. We will show that $\phi^{*}\mathcal{N}_{\mathcal{M}}^{a}(B)\to \mathcal{I}_{\prism}^{a}\mathcal{M}(B)$ is injective. Set $M:=\mathcal{M}(B)$. Let $K$ be the kernel of $\phi^{*}\mathcal{N}_{M}^{a}\to \mathcal{I}_{\prism}^{a}M$. It is a complete module, so it is enough to show that $K/p=0$. By \cref{quot-flatness}, $M/\mathcal{N}_{M}^{a}$ is $p$-torsion free, which means that $K/p$ identifies with the kernel of $\phi^{*}\mathcal{N}_{M}^{a}/p\to M/p$. 
Now $\Phi_{M}:\phi^{*}M\to M$ is injective, so that $K$ is the kernel of the map $\phi^{*}\mathcal{N}_{M}^{a}\to \phi^{*}M$. It is therefore enough to show $Tor_{1}^{\prism_{B}/p}(\prism_{B}^{\phi}/p, M/(N_{M}^{a},p ))=0$, where $\prism_{B}^{\phi}$ denotes $\prism_{B}$, equipped with the module structure given by $\phi$. The proof is then similar to the proof of \cref{qrsp-adm}. There is a spectral sequence (to ease notation, we write $\prism_{B}=\prism$).
\begin{center}
$Tor^{\prism/(N_{\prism}^{a}, p)}_{m}(Tor^{\prism/p}_{n}(\prism^{\phi}/p, \prism/(N^{a}_{\prism}, p)), M/(N_{M}^{a}, p))\Longrightarrow Tor^{\prism/p}_{n+m}(\prism^{\phi}/p, M/(N_{M}^{a}, p))$.
\end{center}
Now $\prism/p\otimes_{\phi}\prism/(N_{\prism}^{a}, p)=\prism/(I^{a}, p)$ and by the proof of \cref{qrsp-adm} we have $Tor^{\prism/p}_{n}(\prism^{\phi}/p, \prism/(N^{a}_{\prism}, p))=0$, for $n > 0$. Moreover, the map $\prism/(N_{\prism}^{a}, p)\to \prism/(I^{a}, p)$ is faithfully flat by the fiberwise flatness criterion. Therefore, the spectral sequence above yields $Tor^{\prism/p}_{m}(\prism^{\phi}/p, M/(N_{M}^{a}, p))=0$, for all $m > 0$.

\end{proof}
\subsection{Divided crystals}
In this section, we will see that the category $CR^{[\wedge, \phi ]}_{[0, r]}(\mathcal{X})$  is equivalent to a category of divided prismatic F-crystals (whose objects correspond to crystalline local systems of positive Hodge-Tate weights. We first define the following big category.
\begin{defn}\label{defn-divcrys}
A divided prismatic $F$-crystal of height $a\geq 0$ on $\mathcal{X}$ is a triple $(\mathcal{M}, Fil^{a}\mathcal{M}, \Phi_{\mathcal{M}}^{a})$, such that
\begin{itemize}
\item $\mathcal{M}$ is a finitely generated completed crystal in $\mathcal{O}^{\mr{pris}}$-modules.
\item  $Fil^{a}\mathcal{M}\subset \mathcal{M}$ is a submodule, such that $\mathcal{N}^{a}_{\prism}\cdot \mathcal{M}\subset Fil^{a}\mathcal{M}$.
\item  $\Phi_{\mathcal{M}}^{a}:Fil^{a}\mathcal{M}\to \mathcal{M}$ is a $\phi_{\mathcal{O}^{\mr{pris}}}$-linear map whose linearization $\phi^{*}Fil^{a}\mathcal{M}\xrightarrow{\cong} \mathcal{M}$ is an isomorphism and such that $Fil^{a}\mathcal{M}$ is maximal, i.e. there is no bigger subsheaf $Fil^{a}\mathcal{M}\subset \mathcal{F}\subset \mathcal{M}$, equipped with a map $\Psi:\mathcal{F}\to \mathcal{M}$, extending $\Phi_{\mathcal{M}}^{a}$.
\end{itemize}
We denote the category of divided prismatic $F$-crystals by $DivCR^{\wedge, \phi, big}_{[0, r ]}$.\\
As before, we then define the subcategories $DivCR^{\wedge, \phi, \mr{tor-free}}_{[0, r ]}(\mathcal{X})$ and $DivCR^{\wedge, \phi, tor}_{[0, r ]}(\mathcal{X})$ of $DivCR^{\wedge, \phi, big}_{[0, r ]}(\mathcal{X})$, where the first consists of divided prismatic $F$-crystals $(\mathcal{M}, Fil^{a}\mathcal{M}, \Phi^{a}_{\mathcal{M}})$, for which $\mathcal{M}(B)$ is $p$-torsion free and saturated, for any qrsp object $Spf(R)\subset \mathcal{X}_{\mr{qsyn}}$.\\
The category $DivCR^{\wedge, \phi, tor}_{[0, r ]}$ contains direct sums of objects $(\mathcal{M}, Fil^{a}\mathcal{M}, \Phi^{a}_{\mathcal{M}})$ for which $\mathcal{M}$ is $\Z_p/p^{n}$-flat and saturated, for some $n\geq 1$.\\
Finally, $DivCR^{\wedge, \phi}_{[0, r ]}(\mathcal{X})$ is the smallest subcategory of $DivCR^{\wedge, \phi, big}_{[0, r ]}(\mathcal{X})$, containing both $DivCR^{\wedge, \phi, tor}_{[0, r ]}(\mathcal{X})$ and $DivCR^{\wedge, \phi, \mr{tor-free}}_{[0, r ]}(\mathcal{X})$.
\end{defn}
\begin{ex}
\begin{itemize}
\item First we recall the Breuil-Kisin twist $\mathcal{O}_{\prism}\{1\}$ (\cite[Ex. 4.5]{BS-pris-cris}), which may informally be written as $\mathcal{O}_{\prism}\{1\}=\bigotimes\limits_{n\geq 0}(\phi^{n})^{*}\mathcal{I}$. If one defines $Fil^{1}\mathcal{O}_{\prism}\{1\}=\mathcal{N}^{1}_{\prism}\mathcal{O}_{\prism}\{1\}$, one sees that $\phi^{*}\mathcal{N}^{1}_{\prism}\mathcal{O}_{\prism}\{1\}=\mathcal{O}_{\prism}\{1\}$. Hence, we obtain the structure of a divided prismatic $F$-crystal.
\item More generally, for any $s\geq 1$, we have Frobenius maps $\mathcal{N}_{\prism}^{s}\{s\}\to \mathcal{O}_{\prism}\{s\}$.
\end{itemize}
\end{ex}

Let now $(\mathcal{M}, \Phi_{\mathcal{M}})\in CR^{[\wedge, \phi ]}_{[0, r]}(\mathcal{X})$ be a completed prismatic $F$-crystal of height $r$. Tensoring $\mathcal{N}^{a}_{\mathcal{M}}\subset \mathcal{M}$ with the invertible module $\mathcal{O}_{\prism}\{a\}$ we obtain $Fil^{a}(\mathcal{M}\{a\}):=\mathcal{N}^{a}_{\mathcal{M}}\{a\}\subset \mathcal{M}\{a\}$ and see that
\begin{align*}
\phi^{*}Fil^{a}(\mathcal{M}\{a\}) & \xrightarrow{\cong}  \mathcal{I}_{\prism}^{a}\mathcal{M}\otimes \bigotimes\limits_{n\geq 1}(\phi^{n})^{*}\mathcal{I}_{\prism}^{a} \\ & = \mathcal{M}\otimes \mathcal{I}_{\prism}^{a}\bigotimes\limits_{n\geq 1}(\phi^{n})^{*}\mathcal{I}_{\prism}^{a} \\ & = \mathcal{M}\{a\},
\end{align*}
where the first map is induced by $\Phi_{\mathcal{M}}$ (which is an isomorphism by \cref{prop-admissible}. We also used that $\mathcal{I}_{\prism}\mathcal{M}=\mathcal{I}_{\prism}\otimes \mathcal{M}$, which follows from the fact that $\mathcal{M}$ is $\mathcal{I}_{\prism}$-torsion free.\\
In this way, we obtain a divided prismatic $F$-crystal $(\mathcal{M}\{a\}, Fil^{a}(\mathcal{M}\{a\}), \Phi_{\mathcal{M}\{a\}})$.\\
To make the construction of $\Phi_{M\{a\}}$ more precise, recall from \cite[Ex. 4.5]{BS-pris-cris} that for any qrsp $Spf(B)\in \mathcal{X}_{\mr{qsyn}}$, the module $(\phi^{n})^{*}(I_{\prism_{B}})$ becomes trivial after base changing to $B/I_{n}$, where $I_{n}=\prod\limits_{i=1}^{n-1}\phi^{i}(I_{\prism_{B}})$. We have $B=\lim_{n}B/I_{n}$, so we can first define $\Phi_{M\{a\}}$ modulo $I_{n}$ in the same way as above and get the map $\Phi_{M\{a\}}$ by passing to the limit.\\[0.5cm]
In the other direction, let $(\mathcal{M}, Fil^{a}\mathcal{M}, \Phi_{\mathcal{M}}^{a})\in DivCR^{\wedge, \phi}_{[0, r ]}$ be a divided completed prismatic $F$-crystal. Then on $\mathcal{M}\{-a\}=\mathcal{M}\otimes \bigotimes\limits_{n\geq 0}(\phi^{n})^{*}\mathcal{I}_{\prism}^{-a}$, we obtain a Frobenius structure $\Phi_{\mathcal{M}\{-a\}}$ via
\begin{align*}
\phi^{*}\mathcal{M}\{-a\} & =  \phi^{*}\mathcal{M}\otimes \bigotimes\limits_{n\geq 1}(\phi^{n})^{*}\mathcal{I}_{\prism}^{-a} \\ & =  \mathcal{I}_{\prism}^{a}\phi^{*}\mathcal{M}\otimes \bigotimes\limits_{n\geq 0}(\phi^{n})^{*}\mathcal{I}_{\prism}^{-a}\rightarrow  (\phi^{*}Fil^{a}\mathcal{M})\{-a\} \xrightarrow[\Phi_{\mathcal{M}}]{\cong} & \mathcal{M}\{-a\},
\end{align*}
where the second last arrow is induced by $\mathcal{I}_{\prism}^{a}\phi^{*}\mathcal{M}\to \phi^{*}Fil^{a}\mathcal{M}$, which comes from base extending $\mathcal{N}_{\prism}^{a}\mathcal{M}\subset Fil^{a}\mathcal{M}$ along the Frobenius map. It is easy to see that $(\mathcal{M}\{-a\}, \Phi_{\mathcal{M}\{-a\}})$ is any object of $CR^{[\wedge, \phi ]}_{[0, a]}(\mathcal{X})$.
\begin{prop}\label{equ-div-crys}
The above constructions are quasi-inverse to each other and thus yield an equivalence of categories
\begin{center}
$CR^{\wedge, \phi}_{[0, a]}(\mathcal{X})\longleftrightarrow DivCR^{\wedge, \phi}_{[0, a ]}(\mathcal{X})$.
\end{center}
\end{prop}
\begin{proof}
Let us denote the functors by 
$\xymatrix{
F: CR^{\wedge, \phi}_{[0, a]}(\mathcal{X})\ar@<.5ex>[r]  & DivCR^{\wedge, \phi}_{[0, a ]}(\mathcal{X}) :G\ar@<.5ex>[l]}$.\\
First we check that $F\circ G=id$. Let $(\mathcal{M}, Fil^{a}\mathcal{M}, \Phi^{a}_{\mathcal{M}})\in DivCR^{\wedge, \phi}_{[0, a]}(\mathcal{X})$. Denote by $\Phi_{\mathcal{M}\{-a\}}:\mathcal{M}\{-a\}\to \mathcal{M}\{-a\}$, the Frobenius crystal constructed above. It is then easy to see that $Fil^{a}\mathcal{M}\{-a\}\subset \mathcal{M}\{-a\}$ is contained in $\mathcal{N}_{\mathcal{M}\{-a\}}^{a}=\Phi_{\mathcal{M}\{-a\}}^{-1}(\mathcal{I}^{a})$. Using the maximality of $Fil^{a}\mathcal{M}\subset \mathcal{M}$, one then deduces $Fil^{a}\mathcal{M}\{-a\}=\mathcal{N}_{\mathcal{M}\{-a\}}^{a}$ and $F(G((\mathcal{M}, Fil^{a}\mathcal{M}, \Phi^{a}_{\mathcal{M}})))=(\mathcal{M}, Fil^{a}\mathcal{M}, \Phi^{a}_{\mathcal{M}})$.\\
The proof that that $G\circ F=id$ is left to the reader.
\end{proof}
Using \cref{proj-away}, we obtain the following.
\begin{lem}
Let $(\mathcal{M}, Fil^{a}\mathcal{M}, \Phi_{\mathcal{M}}^{a})\in DivCR^{\wedge, \phi}_{[0, r ]}$. Then $\mathcal{M}(B)$ is projective away from $(I, p)$, for any qrsp $Spf(B)\in \mathcal{X}_{\mr{qsyn}}$.
\end{lem}
Let $(\mathcal{M}, Fil^{a}\mathcal{M}, \Phi_{\mathcal{M}}^{a})\in DivCR^{\wedge, \phi}_{[0, r ]}$. Using \cref{equ-div-crys} we see that there is a full filtration $Fil^{\bullet}\mathcal{M}$ on $\mathcal{M}$, coming from the full Nygaard filtration of $\mathcal{M}\{-a\}$. For any qrsp $B\in \mathcal{X}_{\mr{qsyn}}$, the filtration is seen to satisfy
\begin{center}
\begin{equation*}
Fil^{i}\mathcal{M}(B):= \begin{cases}
\{m\in \mathcal{M}(B) : \mathcal{N}_{\prism}^{a-i}(B)m\subset Fil^{a}\mathcal{M}(B)\} &\text{if $i\leq a$}\\
\mathcal{N}_{\prism}^{i-a}(B)Fil^{a}\mathcal{M}(B) &\text{if $i > a$}.
\end{cases}
\end{equation*}
\end{center}
Indeed, for any completed Frobenius crystal $(\mathcal{F}, \Phi_{F})$, and any $0\leq i\leq a$, one sees that $\mathcal{N}_{\mathcal{F}}^{i}(B)=\{m\in \mathcal{F}(B): \mathcal{N}^{a-i}_{\mathcal{F}}(B)m\subset \mathcal{N}^{a}_{\mathcal{F}}(B)\}$ (using $I$-torsion freeness of $\mathcal{F}(B)$).\\
This turns $\mathcal{M}$ into a sheaf of filtered modules over the Nygaard-filtered sheaf of rings $\mathcal{O}^{\mr{pris}}$. There are also Frobenius morphisms $\Phi^{i}_{\mathcal{M}}:Fil_{\mathcal{M}, Fil^{a}\mathcal{M}}^{i}\{i-a\}\to \mathcal{M}\{i-a\}$, which for $i\leq a$ are given by
\begin{align*}
\phi^{*}Fil_{\mathcal{M}, Fil^{a}\mathcal{M}}^{i}\{i-a\} & = \mathcal{I}^{a-i}_{\prism} (\phi^{*}Fil_{\mathcal{M}, Fil^{a}\mathcal{M}}^{i})\{i-a\} \to (\phi^{*}Fil^{a}\mathcal{M})\{i-a\}\xrightarrow[\Phi_{\mathcal{M}}]{\cong} \mathcal{M}\{i-1\}.
\end{align*}
Here the second last arrow is obtained from the Frobenius pushout of $\mathcal{N}^{a-i}_{\prism}Fil_{\mathcal{M}, Fil^{a}\mathcal{M}}^{i}\to Fil^{a}\mathcal{M}$.\\
We will need that locally divided prismatic $F$-crystals may be described by prismatic stratifications.
\begin{defn}
Let $A$ be a qrsp $p$-torsion free $\Z_p$-algebra. We define $DivCR^{\wedge, \phi, \mr{tor-free}}_{[0, a]}(\prism_{A})$ as the category of triples $(M, Fil^{a}M, \Phi_{M})$, where
\begin{itemize}
\item $M$ is a finitely generated saturated $\prism_{A}$-module, which is classically $(I, p)$-complete and projective away from $(I, p)$.
\item $Fil^{a}M\subset M$ is a submodule, such that $\mathcal{N}_{\prism_{A}}^{a}M\subset Fil^{a}M$.
\item $\Phi_{M}:Fil^{a}M\to M$ is a $\phi_{A}$-linear map, whose linearization $\phi^{*}Fil^{a}M\xrightarrow{\cong}M$ is an isomorphism.
\end{itemize}
We also define the categories $DivCR^{\wedge, \phi, tor}_{[0, a]}(\prism_{A})$ and $DivCR^{\wedge, \phi}_{[0, a]}(\prism_{A})$ analogously to \cref{completed-crys}, except that we drop the maximality condition. Instead, we denote by $DivCR^{\wedge, \phi}_{[0, a]}(\prism_{A})^{max}$ the full subcategory of objects for which $Fil^{a}M$ is maximal as before.
\end{defn}
\begin{rmk}\label{perf-max}
When $A$ is perfectoid, it is easy to see that every filtration of an object in $DivCR^{\wedge, \phi, tor}_{[0, a]}(\prism_{A})$ is automatically maximal, since now $\phi_{\prism_{A}}$ is an isomorphism, so whenever $Fil^{a}M\subset \tilde{Fil}^{a}M$ is a bigger filtration, the isomorphism $\phi^{*}Fil^{a}M\cong \phi^{*}\tilde{Fil}^{a}M$ forces $Fil^{a}M=\tilde{Fil}^{a}M$.
\end{rmk}
As above, for any $(M, Fil^{a}M, \Phi_{M})\in DivCR^{\wedge, \phi}_{[0, r ]}(\prism_{A})$, we define $Fil^{\bullet}_{M, Fil^{a}M}$ to be the decreasing filtration given by
\begin{center}
\begin{equation*}
Fil_{M, Fil^{a}M}^{i}:= \begin{cases}
\{m\in M : \mathcal{N}_{\prism_{A}}^{a-i}m\subset Fil^{a}M\} &\text{if $i\leq a$}\\
\mathcal{N}_{\prism_{A}}^{i-a}Fil^{a}M &\text{if $i > a$}.
\end{cases}
\end{equation*}
\end{center}
There are Frobenius maps $\Phi^{i}_{\xi}:Fil_{M, Fil^{a}M}^{i}\to M$, defined by $\Phi_{\xi}^{i}(m)=\Phi^{a}(\xi^{a-i}m)$, where $\xi$ is the generator of the Nygaard filtration of $A_{\mr{inf}}(S)$, where $S\to A$ is some perfectoid ring mapping to $A$. Let now $A\to B$ be a morphism of $p$-torsion free qrsp $\Z_p$-algebras. Then for any $0\leq k \leq a$, the Frobenius map $\Phi_{\xi}^{a-k}\otimes \frac{\phi_{\prism_{B}}}{\tilde{\xi}^{k}}:Fil_{M, Fil^{a}M}^{a-k}\hat{\otimes} \mathcal{N}_{\prism_{B}}^{k}\to M_{\prism_{B}}$, is seen to be independent of the choice of $\xi$. Here we again write $\tilde{\xi}=\phi(\xi)$.\\
We have a base change functor
\begin{center}
$DivCR^{\wedge, \phi}_{[0, a]}(\prism_{A})\to DivCR^{\wedge, \phi}_{[0, a]}(\prism_{B})$,
\end{center}
which takes $(M, Fil^{a}M, \Phi^{a}_{M})$ to the triple $(M_{\prism_{B}}, Fil^{a}M_{\prism_{B}}, \Phi^{a}_{M_{\prism_{B}}})$, where $Fil^{a}M_{\prism_{B}}$ is defined as the image of $\bigoplus\limits_{r+s=a}Fil_{M, Fil^{r}M}^{i}\hat{\otimes} \mathcal{N}_{\prism_{B}}^{s}\to M_{\prism_{B}}$ and $\Phi^{a}_{M}$ is induced from \\ $\oplus_{r+s=a}\Phi^{r}_{\xi}\otimes \frac{\phi}{\tilde{\xi^{s}}}$.\\

Assume that $\mathcal{X}=Spf(R)$ is small affine and let $R\to \tilde{R}$ be a $p$-completely faithfully flat cover by a qrsp $\Z_p$-algebra $\tilde{R}$. We again denote by $\prism_{\tilde{R}}^{(\bullet)}$ the cosimplicial prism in $\mathcal{X}_{\prism}$, with $\prism_{\tilde{R}}^{(i)}:=\prism_{\tilde{R}^{\otimes i+1}}$. We further denote by $\xymatrix@C-=0.5cm{
\prism_{\tilde{R}}\ar@<.5ex>[r]^{p_{1}}\ar@<-.5ex>[r]_{p_{2}}  & \prism_{\tilde{R}\hat{\otimes}_{R}\tilde{R}}}$, the maps induced by the canonical inclusions $\tilde{R}\to \tilde{R}\hat{\otimes}_{R}\tilde{R}$.
\begin{defn}
The category $DivCR^{\wedge, \phi}_{[0, a]}(\prism_{\tilde{R}}^{(\bullet)})$ is the category, where objects are given by the following data
\begin{itemize}
\item A divided $F$-module $(M, Fil^{a}M, \Phi_{M}^{a})\in DivCR^{\wedge, \phi}_{[0, a]}(\prism_{\tilde{R}})$.
\item An isomorphism $\epsilon:p_{1}^{*}(M, Fil^{a}M, \Phi_{M}^{a})\xrightarrow{\cong} p_{2}^{*}(M, Fil^{a}M, \Phi_{M}^{a})$ of objects in $DivCR^{\wedge, \phi}_{[0, a]}(\prism_{R_{\infty}}^{(1)})$, satisfying the obvious cocycle condition over $\prism_{\tilde{R}}^{(2)}$.
\end{itemize}
\end{defn}
We have the following lemma.
\begin{lem}
Let $(\mathcal{M}, Fil^{a}\mathcal{M}, \Phi_{\mathcal{M}}^{a})\in DivCR^{[\wedge, \phi, \mr{tor-free}]}(\mathcal{X})$. Then for any $p$-completely faithfully flat morphism $B\to C$ of qrsp $\Z_p$-algebras in $\mathcal{X}_{\mr{qsyn}}$, we have 
\begin{align}\label{mult-div}
Fil^{a}\mathcal{M}(C)=\sum\limits_{r+s=a}Fil^{r}\mathcal{M}\{-s\}(B)\hat{\otimes} \mathcal{N}_{\prism}^{s}\{s\}(C),
\end{align}
under the canonical identification $\mathcal{M}(B)\hat{\otimes}_{B}C\xrightarrow{\cong} \mathcal{M}(C)$, where again we write 
\[\sum\limits_{r+s=a}Fil^{r}\mathcal{M}\{-s\}(B)\hat{\otimes} \mathcal{N}_{\prism}^{s}\{s\}(C)\] for the closure of the image of $\sum\limits_{r+s=a}Fil^{r}\mathcal{M}\{-s\}(B)\otimes \mathcal{N}_{\prism}^{s}\{s\}(C)$. Moreover, the Frobenius $\Phi^{a}_{\mathcal{M}(C)}$ is induced from $\sum\limits_{r+s=a}\Phi^{r}\otimes \phi^{s}$.\\
\end{lem}
\begin{proof}
This follows from \cref{prop-fil-crys}, using the equivalence of categories in \cref{equ-div-crys}.
\end{proof}
Using the lemma above, we see that there is an evaluation functor $DivCR^{\wedge, \phi}_{[0, a]}(\mathcal{X})\to DivCR^{\wedge, \phi}_{[0, a]}(\prism_{\tilde{R}}^{(\bullet)})^{\mr{max}}$. This functor is in fact an equivalence of categories, i.e. we have effective descent for divided Frobenius modules.
\begin{prop}
Let $\mathcal{X}=Spf(R)$ be small affine. Then evaluation yields an equivalence of categories
\begin{center}
$DivCR^{\wedge, \phi}_{[0, a]}(\mathcal{X})\longleftrightarrow DivCR^{\wedge, \phi}_{[0, a]}(\prism_{\tilde{R}}^{(\bullet)})^{max}$.
\end{center}
\end{prop}
\begin{proof}
The construction of the equivalence of categories between $DivCR^{\wedge, \phi}_{[0, a]}$ and $CR^{\wedge, \phi}_{[0, a]}$ above may also be carried out on the level of descent data, which yields an equivalence of categories
\begin{center}
$DivCR^{\wedge, \phi}_{[0, a]}(\prism_{\tilde{R}}^{(\bullet)})^{max}\longleftrightarrow CR^{\wedge, \phi}_{[0, a]}(\prism_{\tilde{R}}^{(\bullet)})$.
\end{center}
On the other hand, for completed F-crystals we have effective descent, meaning that evaluating yields an equivalence of categories
\begin{center}
$CR^{\wedge, \phi}_{[0, a]}(\mathcal{X})\longleftrightarrow CR^{\wedge, \phi}_{[0, a]}(\prism_{\tilde{R}}^{(\bullet)})$.
\end{center}
Putting the equivalences above together yields the result.
\end{proof}
\begin{rmk}\label{full-filt}
Let $(M, Fil^{a}M, \phi^{a}_{M}, \epsilon)\in DivCR^{\wedge, \phi}_{[0, a]}(\prism_{\tilde{R}}^{(\bullet)})^{max}$. Then the full filtration $Fil^{\bullet}_{M, Fil^{a}M}$ is induced by the Nygaard-filtraiton of the associated object in $CR^{\wedge, \phi}_{[0, a]}$. It thus follows that $\epsilon$ is in fact automatically compatible with the full product filtration, i.e. it takes $p_{1}^{*}Fil^{k}M$ to $p_{2}^{*}Fil^{k}M$, for all $k\geq 0$.
\end{rmk}
We finish this section by discussing the \'etale realization functor for divided prismatic crystals. 
\begin{defn}
We denote by $\Z_p-Loc(\mathcal{X}_{K})$ the category of finite locally free $\hat{\Z}_{p}$-sheaves on the pro-\'etale site of the generic fiber $\mathcal{X}_{K}$.\\
We also denote by $\Z_p-Loc_{\mr{[0, a]}}^{\mr{cris}}(\mathcal{X}_{K})$ the category of crystalline local systems with Hodge-Tate weights in $[0, a]$ on $\mathcal{X}_{K}$ (see \cite[Def. A.11.]{DLMS} or \cite[Def. 2.25]{GR} for the definition of crystalline local systems).
\end{defn}
We remark here that we use the convention that $\Z_p(1)$ has Hodge-Tate weight equal to $1$.\\
We denote by $LF(\mathcal{X}, \mathcal{O}^{\mr{pris}}[\frac{1}{\mathcal{I}_{\prism}}]^{\wedge}_{p})^{\phi=1}$ the category of finite locally free $\mathcal{O}^{\mr{pris}}[\frac{1}{\mathcal{I}_{\prism}}]^{\wedge}_{p}$-crystals $\mathcal{E}$, equipped with a Frobenius, whose linearization $\phi^{*}\mathcal{E}\xrightarrow{\cong} \mathcal{E}$ is an isomorphism.
\begin{prop}
The category $LF(\mathcal{X}, \mathcal{O}^{\mr{pris}}[\frac{1}{\mathcal{I}_{\prism}}]^{\wedge}_{p})^{\phi=1}$ is canonically identified with $\Z_p-Loc(\mathcal{X}_{K})$.
\end{prop}
\begin{proof}
See \cite[\S 3]{BS-pris-cris}.
\end{proof}
For any $k\geq 0$, we have $\mathcal{N}^{k}_{\prism}[\mathcal{I}_{\prism}^{-1}]=\mathcal{O}_{\prism}[\mathcal{I}_{\prism}^{-1}]$. Therefore, for any $(\mathcal{M}, Fil^{a}\mathcal{M}, \Phi^{a}_{\mathcal{M}})$, we also have $Fil^{a}\mathcal{M}[\mathcal{I}_{\prism}^{-1}]=\mathcal{M}[\mathcal{I}_{\prism}^{-1}]$. Hence, there is an \'etale realization functor 
\begin{center}
$T_{\acute{e}t}:DivCR^{\wedge, \phi}_{[0, a]}(\mathcal{X})\to \Z_p-Loc(\mathcal{X}_{K})$.
\end{center} 
\begin{prop}
Let $F_{a}:CR^{\wedge, \phi}_{[0, a]}(\mathcal{X})\to DivCR^{\wedge, \phi}_{[0, a]}(\mathcal{X})$ be the equivalence from \cref{equ-div-crys}. For any $\mathcal{M}\in CR^{\wedge, \phi}_{[0, a]}(\mathcal{X})$,  we then have $T_{\acute{e}t}(F_{a}(\mathcal{M}))=T_{\acute{e}t}(\mathcal{M})(a)$.\\
In particular, the \'etale realization functor restricts to an equivalence of categories between $DivCR^{\wedge, \phi, \mr{tor-free}}_{[0, a]}(\mathcal{X})$ and $\Z_p-Loc_{\mr{[0, a]}}^{\mr{cris}}(\mathcal{X})$ the category of crystalline $\Z_p$ local systems on $\mathcal{X}_{K}$ with Hodge-Tate weights in $[0, a]$.
\end{prop} 
\begin{proof}
The first statement is easily varified. The statement about crystalline local systems then follows from \cite[Theorem 1.3]{DLMS}. We only remark here that we use the covariant \'etale realization functor. I.e. by $T_{\acute{e}t}:CR^{\wedge, \phi}_{[0, a]}(\mathcal{X})\to \Z_p-Loc(\mathcal{X})$ we mean the functor $(\mathcal{M}, \Phi_{\mathcal{M}})\mapsto (\mathcal{M}[\frac{1}{\mathcal{I}_{\prism}}]_{p}^{\wedge})^{\Phi=id}$, whereas in \cite{DLMS} the authors take the dual of this local system. In particular the \'etale realization functor used here takes $CR^{\wedge, \phi}_{[0, a]}(\mathcal{X})$ to crystalline local systems with negative Hodge-Tate weights.
\end{proof}
\subsubsection{Duality}
The aim of this section is to provide an alternative way to pass from completed F-crystals to divided F-crystals by means of duality. We will only sketch this for locally free objects - possibly this could be extended to cover all completed F-crystals.\\ Consider the following construction: First, let $(\mathcal{M}, \phi_{\mathcal{M}})\in CR^{\wedge, \phi}_{\mr{[0, a]}}(\mathcal{X})$ be finite locally free of height $a$. On the dual crystal, we then have a filtration $Fil^{a}\mathcal{M}^{\vee}\subset \mathcal{M}^{\vee}:=\mathcal{H}om(\mathcal{M}, \mathcal{O}^{\mr{pris}})$, defined by $Fil^{a}\mathcal{M}^{\vee}=\{f\in \mathcal{M}^{\vee}:f(\mathcal{N}^{a}_{\mathcal{M}})\subset \mathcal{N}_{\prism}^{a}\}$. For any $f\in Fil^{a}\mathcal{M}^{\vee}$, we then define $F:\mathcal{I}^{a}\mathcal{M}\to \mathcal{I}^{a}\mathcal{O}_{\prism}$, making the following diagram commutative.
\begin{center}
$\xymatrix{
\phi_{\mathcal{M}}:\phi^{*}\mathcal{N}^{a}_{\mathcal{M}}\ar[rr]^{\cong}\ar[dr]_{f} && \mathcal{I}^{a}\mathcal{M}\ar[dl]^{F} \\
& \mathcal{I}^{a}\mathcal{O}^{pris} &
}$
\end{center}
We can thus define a $\phi$-linear map $\phi_{\mathcal{M}^{\vee}}:Fil^{a}\mathcal{M}^{\vee}\to \mathcal{M}^{\vee}$, by setting $\phi_{\mathcal{M}^{\vee}}(f)=\mathcal{I}^{-a}F$. It is easy to see that this defines an object $(\mathcal{M}^{\vee}, Fil^{a}\mathcal{M}^{\vee}, \phi_{\mathcal{M}^{\vee}})\in DivCR^{\wedge, \phi}_{[0, a]}(\mathcal{X})$.\\
In the other direction, let $(\mathcal{M}, \overbar{\mathcal{M}}, \phi_{\mathcal{M}})\in DivCR^{\wedge, \phi}_{[0, a]}(\mathcal{X})$ be a locally free divided crystal. For a section $f\in \mathcal{M}^{\vee}$, we define $\phi_{\mathcal{M}^{\vee}}(f)$ to be the unique morphism, making the following diagram commutative.
\begin{center}
$\xymatrix{
\overbar{\mathcal{M}}\ar@{^{(}->}[r]\ar[d]^{\phi_{\mathcal{M}}} & \mathcal{M} \ar[r]^{f} & \mathcal{O}^{pris}\ar[d]^{\phi} \\
\mathcal{M}\ar[rr]^{\phi_{\mathcal{M}^{\vee}}(f)} && \mathcal{O}^{pris}
}$
\end{center}
Concretely, $\phi_{\mathcal{M}^{\vee}}(f)$ is the morphism
\begin{center}
$\mathcal{M}\xrightarrow{\cong}\phi^{*}\overbar{\mathcal{M}}\hookrightarrow \phi^{*}\mathcal{M}\xhookrightarrow{f\otimes \phi}\mathcal{O}$.
\end{center}
We see that $(\mathcal{M}^{\vee}, \phi_{\mathcal{M}^{\vee}})$ is a completed prismatic F-crystal of height $r$. One then has the following result, whose proof is left to the reader.
\begin{prop}
The duality functor described above yields an equivalence of categories
\begin{center}
$
\mathbb{D}:CR^{\wedge, \phi}_{[0, a]}(\mathcal{X})\xrightarrow{\cong}DivCR^{\wedge, \phi}_{[0, a]}(\mathcal{X})^{opp}.
$
\end{center}
Moreover, the duality functor is compatible with the \'etale realization functor, i.e. the canonical map 
\begin{center}
$ T_{\acute{e}t}(\mathbb{D}(\mathcal{M}, \Phi_{\mathcal{M}}))\to T_{\acute{e}t}(\mathcal{M}, \Phi_{\mathcal{M}})^{\vee}$
\end{center}
is an isomorphism of $\Z_p$-local systems on $\mathcal{X}_{K}$.
\end{prop}

\section{The category $\MF$ and prismatic F-crystals}

\subsection{The functor from divided prismatic F-crystals to $\MF$}\label{sec-constr}
We now explain how one associates an object in the category $\MFa$ to an object in $\DPa$. Recall the notation $C=\hat{\overbar{K}}$ and $(A_{\mr{inf}}=\prism_{\mathcal{O}_{C}}, \tilde{\xi})$ for the universal prism over $\mathcal{O}_{C}$, where we use the convention $\tilde{\xi}=\phi(\xi)$, where $\xi$ is a generator of the kernel of Fontaine's map $\theta:A_{\mr{inf}}\to \mathcal{O}_{C}$. We further denote by $A_{\mr{cris}}=A_{\mr{inf}}\{\frac{\tilde{\xi}}{p}\}^{\wedge}$ the $p$-adically completed PD-envelope of $\xi$ in $A_{\mr{inf}}$, and by $J_{A_{\mr{cris}}}$ the ideal generated by divided powers of $\xi$. Let $B$ be a qrsp algebra over $\mathcal{O}_{C}$, which is $p$-torsion free. By the universal property of $\prism_{B}$, the canonical projection $B\to B/p$ induces a canonical map $c:\prism_{B}\to \prism_{B/p}=\mathbb{A}_{\mr{cris}}(B/p)$. Since prismatic cohomology is compatible with base change of the base prism, this map is simply the pushout of the natural map $A_{\mr{inf}}=\prism_{\mathcal{O}_{C}}\to \prism_{\mathcal{O}_{C}/p}=A_{\mr{cris}}$, along $\prism_{\mathcal{O}_{C}}\to \prism_{B}$.\\
There is a surjective map $\prism_{B/p}=\prism_{B}\hat{\otimes}_{A_{\mr{inf}}}A_{\mr{cris}}\twoheadrightarrow (\prism_{B}/N^{1}_{\prism_{B}})\hat{\otimes}(A_{\mr{cris}}/J_{A_{\mr{cris}}})=B$, whose kernel is a PD-ideal, which we denote by $J_{B}\subset \prism_{B/p}$ (this is the universal map already considered in \cref{sec-MF}).\\
The following lemma follows then easily from the fact that $\mathbb{A}_{\mr{cris}}(B/p)=\prism_{B}\hat{\otimes}_{A_{\mr{inf}}}A_{\mr{cris}}$.
\begin{lem}\label{ideals}
Let $Spf(B)\in \mathcal{X}_{\mr{qsyn}}$, where $B$ is a qrsp $\mathcal{O}_{C}$-algebra. Then for $i\leq p-1$, we have $J_{B}^{[i]}=c(N_{\prism_{B}}^{i})+Fil^{p}A_{\mr{cris}}\cdot B$.
\end{lem}
We now come to our main result.
\begin{thm}\label{main-thm}
Let $\mathcal{X}$ be a smooth $p$-adic formal scheme over $\mathcal{O}_{K}$. Let $a\leq p-2$. Then there is a natural equivalence of categories 
\begin{align}\label{functor}
DivCR^{\wedge, \phi, \mr{tor-free}}_{[0, a]}(\mathcal{X})\longleftrightarrow \MFa^{\mr{tor-free}}(\mathcal{X}).
\end{align}
\end{thm}
\begin{proof}
To describe the functor, it is enough, by Zariski descent, to define the functor for affine $\mathcal{X}=Spf(R)$ and check that it is natural in $R$.\\
So assume that $\mathcal{X}=Spf(R)$ and let $R\to \tilde{R}$ be the quasi-syntomic covering discussed in \cref{coverings}. Recall that $DivCR^{\wedge, \phi, \mr{tor-free}}_{[0, a]}(\mathcal{X})$ is naturally equivalent to the category $DivCR^{\wedge, \phi, \mr{tor-free}}_{[0, a]}(\prism_{\tilde{R}}^{(\bullet)})^\mr{{max}}$. On the other hand, recall that by \cref{left-inverse}, we have a functor $\mathbb{D}_{\tilde{R}}:\MFa^{\mr{big}}(\mathbb{A}_{\mr{cris}}(\tilde{R}/p)^{(\bullet)})^{\mr{full}}\to \MFa^{\mr{big}}(\mathcal{X})$ which is natural in $R$. Here again we write $\mathbb{A}_{\mr{cris}}(\tilde{R}/p)^{(n)}:=\mathbb{A}_{\mr{cris}}(\tilde{R}^{\otimes n+1}/p)$. Recall that $\prism_{\tilde{R}^{\otimes n}/p}\cong \mathbb{A}_{\mr{cris}}(\tilde{R}^{\otimes n}/p)$ canonically and that we have a natural map 
\begin{center}
$c:\prism_{\tilde{R}}^{(n)}\to \prism_{\tilde{R}/p}^{(n)}=\mathbb{A}_{\mr{cris}}(\tilde{R}/p)^{(n)}$,
\end{center}
such that $c(\mathcal{N}_{\prism_{\tilde{R}}}^{a})\subset J_{\tilde{R}}^{[a]}$, for $a\leq p-1$. \\

Now, if $(M, Fil^{a}M, \Phi^{a}_{M})\in \DPa(\prism_{\tilde{R}}^{(n)})$, we can construct a triple \\
$(M_{\mathbb{A}_{\mr{cris}}(\tilde{R}/p)^{(n)}}, Fil^{a}M_{\mathbb{A}_{\mr{cris}}(\tilde{R}/p)^{(n)}}, \Phi^{a}_{M_{\mathbb{A}_{\mr{cris}}(\tilde{R}/p)^{(n)}}})$, where $M_{\mathbb{A}_{\mr{cris}}(\tilde{R}/p)^{(n)}}=M\hat{\otimes}_{\prism_{\tilde{R}}^{(n)}} \mathbb{A}_{\mr{cris}}(\tilde{R}/p)^{(n)}=M\hat{\otimes}_{A_{\mr{inf}}}A_{\mr{cris}}$ and $Fil^{a}M_{\mathbb{A}_{\mr{cris}}(\tilde{R}/p)^{(n)}}\subset M_{\mathbb{A}_{\mr{cris}}(\tilde{R}/p)^{(n)}}$ is the submodule generated by 
\begin{align*}
& Im(Fil^{a}M\otimes \mathbb{A}_{\mr{cris}}(\tilde{R}/p)^{(n)})+ J_{\tilde{R}^{\otimes n}}^{[p]}\cdot M_{\mathbb{A}_{\mr{cris}}(\tilde{R}/p)^{(n)}}\\
= & Im(Fil^{a}M\otimes \mathbb{A}_{\mr{cris}}(\tilde{R}/p)^{(n)})+  J_{\tilde{R}^{\otimes n+1}}^{[a]}\cdot M_{\mathbb{A}_{\mr{cris}}(\tilde{R}/p)^{(n)}}\\
= & Im(Fil^{a}M\otimes \mathbb{A}_{\mr{cris}}(\tilde{R}/p)^{(n)})+  Fil^{a}A_{\mr{cris}}\cdot M_{\mathbb{A}_{\mr{cris}}(\tilde{R}/p)^{(n)}}. 
\end{align*}
The Frobenius on $Fil^{a}M_{\mathbb{A}_{\mr{cris}}(\tilde{R}/p)^{(n)}}$ is defined as follows:
First, note that the Frobenius on $Fil^{a}M$ may be written as 
\begin{center}
$Fil^{a}M\hookrightarrow M\xrightarrow{\Phi^{0}}M\xrightarrow{\frac{1}{\tilde{\xi}^{a}}} M$,
\end{center}
where $\Phi^{0}:M\to M$ is defined as $M\xrightarrow{\cdot \xi^{a}}Fil^{a}M\xrightarrow{\Phi^{a}}M$. This shows that $Fil^{a}M\otimes A_{\mr{cris}}\xrightarrow{\Phi^{a}\otimes \phi}M_{\mathbb{A}_{\mr{cris}}(\tilde{R}/p)^{(n)}}$ factors over $Im(Fil^{a}M\otimes \mathbb{A}_{\mr{cris}}(\tilde{R}/p)^{(n)})$.\\
Then, we remark that $Fil^{a}A_{\mr{cris}}\cdot M_{\mathbb{A}_{\mr{cris}}(\tilde{R}/p)^{(n)}}=M\otimes Fil^{a}A_{\mr{cris}}\subset M\hat{\otimes}_{A_{\mr{inf}}}A_{\mr{cris}}$, using $\xi$-torsion freeness of $M$. We then define the Frobenius on $M\otimes Fil^{a}A_{\mr{cris}}$ as $\Phi^{0}\otimes \frac{\phi}{\phi(\xi)^{a}}$. It is then easy to see that the so defined Frobenius maps agree on the intersection 
\begin{center}
$Im(Fil^{a}M\otimes \mathbb{A}_{\mr{cris}}(\tilde{R}/p)^{(n)})\cap Fil^{a}A_{\mr{cris}}\cdot M_{\mathbb{A}_{\mr{cris}}(\tilde{R}/p)^{(n)}}$,
\end{center}
so that we get a well-defined Frobenius map $Fil^{a}M_{\mathbb{A}_{\mr{cris}}(\tilde{R}/p)^{(n)}}\to M_{\mathbb{A}_{\mr{cris}}(\tilde{R}/p)^{(n)}}$. This Frobenius is seen to satisfy the conditions \ref{Frob-acris-iso} and \ref{cond-frob-2}. Also, the Frobenius induces an isomorphism
\begin{center}
$Fil^{a}M_{\mathbb{A}_{\mr{cris}}(\tilde{R}/p)^{(n)}}/J_{\tilde{R}^{\otimes n+1}}\otimes_{\phi}\mathbb{A}_{\mr{cris}}(\tilde{R}/p)^{(n)}/p=Fil^{a}M/\mathcal{N}^{1}_{\prism}\otimes_{\phi}\prism_{\tilde{R}}^{(n)}/p\otimes_{\prism_{\tilde{R}}/p}\mathbb{A}_{\mr{cris}}(\tilde{R}/p)^{(n)}/p\cong M_{\mathbb{A}_{\mr{cris}}(\tilde{R}/p)^{(n)}}/pM_{\mathbb{A}_{\mr{cris}}(\tilde{R}/p)^{(n)}}$.
\end{center}
Moreover, the definition of Frobenius is canonical (i.e. independent of $\xi$), as it may be checked to be induced from 
\begin{center}
$\sum\limits_{r+s=a}\Phi_{M}^{r}\otimes \phi^{s}:\bigoplus\limits_{r+s=a}Fil^{r}M\{-s\}\otimes J_{\tilde{R}^{\otimes n+1}}^{[s]}(\mathbb{A}_{\mr{cris}}(\tilde{R}/p)^{(n)})\{s\}\to M_{\mathbb{A}_{\mr{cris}}(\tilde{R}/p)^{(n)}}$.
\end{center}
We remark that $\mathcal{O}\{1\}$ is globally trivial on $(\mathcal{X}/p)_{\prism}$ and $\phi^{s}$ is canonically identified with the restriction to $\mathcal{I}_{\mathcal{X}_{\mr{cris}}}$ of the the divided Frobenius $\frac{\phi}{p^{s}}:\mathcal{N}^{s}_{(\mathcal{X}/p)_{\prism}}\to \mathcal{O}_{(\mathcal{X}/p)_{\prism}}$, evaluated at $\mathbb{A}_{\mr{cris}}(\tilde{R}/p)^{(n)}$.\\[0.5cm]
Now if $(M, Fil^{a}M, \Phi^{a}_{M})\in DivCR^{\wedge, \phi, \mr{tor-free}}_{\mr{[0, a]}}(\prism_{\tilde{R}})$, we may also base change the full filtration $Fil^{i}M$, by setting $Fil^{i}M_{\mathbb{A}}:=Im(Fil^{i}M\hat{\otimes} \mathbb{A})+Fil^{p}A_{cris}\cdot M_{\mathbb{A}}$. \\
We claim that the full filtrations are compatible.
\begin{lem}\label{comp-filt}
We have $Fil^{i}M_{\mathbb{A}}=\{m\in M_{\mathbb{A}}: J_{\tilde{R}^{\otimes n+1}}^{[i-a]}m\subset Fil^{a}M_{\mathbb{A}}\}=Fil^{i}_{M_{\mathbb{A},Fil^{a}M_{\mathbb{A}}}}$ for $i\leq a$. 
\end{lem}
\begin{proof}
First note that the map $M\to M_{\mathbb{A}}$ is injective and for any $m\in M_{\mathbb{A}}$, we may write $m=x+y$, with $x\in M$ and $y\in Fil^{p}A_{cris}\cdot M_{\mathbb{A}}$. Assume now that $um\in Fil^{a}M_{\mathbb{A}}$, for all $u\in J_{\tilde{R}^{\otimes n+1}}^{[i-a]}$. Let $u\in N^{i-a}_{\prism}$. Modding out by $Fil^{p}A_{cris}$, we then get that $ux$ is congruent to an element in $Fil^{a}M$ modulo $N^{p}_{\prism}\cdot M$. But since $N^{p}_{\prism}\cdot M\subset Fil^{a}M$, this means that $ux\in Fil^{a}M$, hence $x\in Fil^{i}M$.
\end{proof}
Let $p_{1}, p_{2}:\prism_{\tilde{R}}\to \prism_{\tilde{R}}^{(1)}$ be the maps induced by the inclusions $\tilde{R}\to \tilde{R}\hat{\otimes}\tilde{R}$. Using the previous lemma, one then sees that the $p_{j}^{*}(M, Fil^{a}M, \Phi^{a}_{M})\in \DPa(\prism_{\tilde{R}}^{(1)})$ maps to $(p^{pd}_{j})^{*}(M_{\mathbb{A}}, Fil^{a}M_{\mathbb{A}}, \Phi^{a}_{M_{\mathbb{A}}})\in \MFa(\mathbb{A}^{(1)})$, where we write $p_{1}^{pd}$ (resp. $p_{2}^{pd}$) for the map $\mathbb{A}\to \mathbb{A}^{(1)}$, induced by mapping $\tilde{R}$ to the first (resp. second) factor of $\tilde{R}\hat{\otimes} \tilde{R}$.\\
Note also that by \cref{full-filt}, for any object $(M, Fil^{a}M, \phi^{a}_{M}, \epsilon)\in DivCR_{\mr{[0, a]}}^{\wedge, \phi, \mr{tor-free}}(\prism_{R_{\infty}}^{(\bullet)})$, the descent datum $\epsilon$ always fixes the full product filtration. Hence, \cref{comp-filt} then yields that this also holds for the induced object over $A_{cris}$.\\

In this way, we obtain a canonical functor 
\begin{center}
$DivCR^{\wedge, \phi}_{\mr{[0, a]}}(\prism_{\tilde{R}}^{(\bullet)})\to \MF(\mathbb{A}_{\mr{cris}}(\tilde{R}/p)^{(\bullet)})$
\end{center}
which restricts to a canonical functor.
\begin{center}
$DivCR^{\wedge, \phi, \mr{tor-free}}_{\mr{[0, a]}}(\prism_{\tilde{R}}^{(\bullet)})\to \MF^{\mr{tor-free}}(\mathbb{A}_{\mr{cris}}(\tilde{R}/p)^{(\bullet)})^{\mr{full}}$.
\end{center}
Restricting this functor to objects with maximal filtrations and composing with the functor $\mathbb{D}_{\tilde{R}}:\MFa^{big}(\mathbb{A}_{\mr{cris}}(\tilde{R}/p)^{(\bullet)})\to \MFa^{big}(\mathcal{X})$ (see \cref{left-inverse}) yields the desired functor
\begin{center}
$DivCR^{\wedge, \phi, \mr{tor-free}}_{\mr{[0, a]}}(\mathcal{X})\to \MFa^{\mr{tor-free}}(\mathcal{X})$.
\end{center}
This construction is natural in $R$. We remark here that a priori the composed functor only maps into $\MFa^{big}(\mathcal{X})$. However, using that for a torsion free prismatic divided F-crystal $(\mathcal{M}, Fil^{a}\mathcal{M}, \Phi^{a})$ the quotients $\mathcal{M}(B)/Fil^{a}\mathcal{M}(B)$ are $p$-torsion free for all qrsp $Spf(B)\in \mathcal{X}_{qsyn}$, one sees that the object it maps to in $\MFa^{\mr{big}}(\mathcal{X})$ has $p$-torsion free quotient as well.\\
We will now give the proof of \cref{main-thm} modulo the results of the following section. Assume that $\mathcal{X}=Spf(R)$ is small affine. To show that the functor \ref{functor} is an equivalence of categories, we may in fact work with any cover of $R$. So instead of the canonical covering $R\to \tilde{R}$ considered above, we instead use the perfectoid covering $R\to R_{\infty}$. Then consider the cosimplicial object $\prism_{R_{\infty}}^{(\bullet)}$ in $\mathcal{X}_{\prism}$, obtained by taking the \v{C}ech-nerve of the covering $R\to R_{\infty}$. Note that in the perfectoid case filtrations are automatically maximal by \cref{perf-max}. \\
Now the base change functor $\DPa(\prism_{R_{\infty}})\to \MFa(\mathbb{A}_{\mr{cris}}(R_{\infty}/p))$, defined in the same way as for the covering $\tilde{R}$, considered above, is an equivalence of categories by \cref{perf-equiv-free}. Also, for $i>0$, one has that the base change functor $\DPa(\prism_{R_{\infty}}^{(i)})\to \MFa(\mathbb{A}_{\mr{cris}}(R_{\infty}/p)^{(i)})$ is fully faithful, by \cref{prop-full-faith}. Combining these two statements then gives that $\DPa(\prism_{R_{\infty}}^{(\bullet)})\to \MFa(\mathbb{A}_{\mr{cris}}(R_{\infty}/p)^{(\bullet)})$ is an equivalence of categories. We note here again that by \cref{full-filt}, for any object $(M, Fil^{a}M, \phi^{a}_{M}, \epsilon)\in DivCR_{\mr{[0, a]}}^{\wedge, \phi, \mr{tor-free}}(\prism_{R_{\infty}}^{(\bullet)})$, the descent datum $\epsilon$ always fixes the full product filtration. Thus, we see here that this is also automatically true for any object in $\MFa(\mathbb{A}_{\mr{cris}}(R_{\infty}/p)^{(\bullet)})$, i.e. we have $\MFa(\mathbb{A}_{\mr{cris}}(R_{\infty}/p)^{(\bullet)})=\MFa(\mathbb{A}_{\mr{cris}}(R_{\infty}/p)^{(\bullet)})^{\mr{full}}$. \\

We thus have a commutative diagram of functors
\begin{center}
$\xymatrix{
DivCR_{\mr{[0, a]}}^{\wedge, \phi, \mr{tor-free}}(\mathcal{X})\ar[rr]\ar[d]^{\cong} && \MFa^{\mr{tor-free}}(\mathcal{X})\ar[d]^{\cong}\\
DivCR_{\mr{[0, a]}}^{\wedge, \phi, \mr{tor-free}}(\prism_{R_{\infty}}^{(\bullet)})\ar[rr]^{\cong} && \MFa^{\mr{tor-free}}(\mathbb{A}_{\mr{cris}}(R_{\infty}/p)^{(\bullet)})
}$
\end{center}
where the upper functor is the functor \ref{functor} and the lower functor is the base change functor along $\prism_{R^{\otimes n}_{\infty}}\to \prism_{R^{\otimes n}_{\infty}/p}$. The vertical functor on the right is the equivalence of categories from \cref{left-inverse-perf}.\\
This finishes the proof.
\end{proof}
We now define the \'etale realization functor
\begin{center}
$T_{\acute{e}t}^{\MF}:\MFa(\mathcal{X})\to Loc_{\Z_p}(\mathcal{X}_{K})$
\end{center}
as the composition of the functor $\MFa(\mathcal{X})\to DivCR^{\wedge, \phi}_{[0, a]}(\mathcal{X})$ with the \'etale realization functor $T_{\acute{e}t}:DivCR^{\wedge, \phi}_{[0, a]}(\mathcal{X})\to Loc_{\Z_p}(\mathcal{X}_{K})$ for divided prismatic Frobenius crystals.\\ 
Using the results from \cite{GR} and \cite{DLMS} (see in particular \cite[Thm. 1.3]{DLMS}), we then obtain the following result for $p$-torsion free objects.
\begin{thm}
Let $\mathcal{X}$ be a smooth $p$-adic formal scheme over $\mathcal{O}_{K}$. Let $a\leq p-2$. Then the \'etale realization functor 
\begin{center}
$T_{\acute{e}t}^{\MF}:\MFa^{\mr{tor-free}}(\mathcal{X})\to Loc^{crys, \mr{tor-free}}_{\Z_p, [0, a]}(\mathcal{X}_{K})$
\end{center}
is an equivalence of categories.
\end{thm}

\subsection{Comparing triples over $A_{inf}$ with triples over $A_{\mr{cris}}$}
This section contains the main technical results one needs to show that the functor \ref{functor} is an equivalence of categories. For this we need to compare objects over $A_{\mr{inf}}$ with the objects over $A_{\mr{cris}}$. Similar results have been obtained in the literature before (see for example \cite{Breuil-constr}, \cite{Kisin-moduli}, \cite{CarLiu}, \cite{Cais-Lau}, \cite{Fal-MF-st}, \cite{Tsuji-mf}), though not in the full generality necessary for us. The proofs are however often similar and inspired by the above mentioned works. Throughout the section we will denote by $B$ a qrsp $p$-torsion free $\mathcal{O}_{C}$-algebra. Recall from the previous section that we have a canonical map $c:\prism_{B}\to \prism_{B/p}=\prism_{B}\hat{\otimes}_{A_{\mr{inf}}}A_{\mr{cris}}=\mathbb{A}_{\mr{cris}}(B/p)$ and we consider the base change functor of divided Frobenius modules, which was constructed in the previous section. For a triple $(M, Fil^{a}M, \Phi^{a}_{M})$ over $\prism_{B}$ we will then denote by $(M, Fil^{a}M, \Phi^{a}_{M})_{\mr{cris}}=(M_{\mr{cris}}, Fil^{a}M_{\mr{cris}}, \Phi^{a}_{M_{\mr{cris}}})$ its essential image under the base change functor. We start by treating the case of $p$-torsion objects.

\begin{lem}\label{lem-mod-p}
Base changing along $\prism_{B}\to \prism_{B/p}$ induces a fully faithful functor
\begin{align*}
DivCR_{[0, a]}^{\wedge, \phi, p-tor}(\prism_{B}) & \longrightarrow   \MFa^{ p-tor}(\prism_{B/p})
\end{align*}
of objects annihilated by $p$.
\end{lem}
\begin{proof}
Modulo $p$, the natural map $\prism_{B}\to \prism_{B/p}$ factors as
\begin{center}
$(\prism_{B})/p\xtwoheadrightarrow{f} (\prism_{B})/(p, \xi^{p})\xlongrightarrow{g} A_{\mr{cris}}(B/p)/p$.
\end{center}
We will show that base changing along $f$ is fully faithful and base changing along $g$ induces an equivalence of categories.\\
To ease notation, we write $P:=(\prism_{B})/p$ and denote the absolute Frobenius of $P$ by $\sigma$. We then equip the ring $P/\xi^{p}$ with the filtration $F^{i}_{P}$, given by the image of the Nygaard filtration. We define the category $DivCR_{[0, a]}^{\wedge, \phi}(P/\xi^{p})$ of triples $(M, Fil^{a}M, \Phi^{a}_{M})$, where 
\begin{itemize}
\item $M$ is a finitely generated $P/\xi^{p}$-module, which is flat over $\mathcal{O}_{C}^{\flat}/\xi^{p}$. 
\item $Fil^{a}M\subset M$ a submodule, such that $F_{P}^{a}\cdot M\subset Fil^{a}M$. \item $\Phi^{a}_{M}:Fil^{a}M\to M$ is a Frobenius linear map, such that $\phi_{P/\xi^{p}}^{*}Fil^{a}M\xrightarrow{\cong} M$ is an isomorphism.
\end{itemize}
We also have a base change functor along $P\to P/\xi^{p}$. Namely, if $(M, Fil^{a}M, \Phi^{a}_{M})$ is a triple over $P$, we define $Fil^{a}(M/\xi^{p})=Im(Fil^{a}M/\xi^{p})\to M/\xi^{p}$. Also, since $\xi^{a}M\subset Fil^{a}M$ and $a<p$, we see that $\Phi^{a}_{M}(\xi^{p}M)\subset \xi^{p}M$. From this, it follows that $\Phi^{a}_{M}\otimes \phi:Fil^{a}M\otimes P/\xi^{p}\to M_{P/\xi^{p}}$ factors over $Fil^{a}(M_{P/\xi^{p}})$, which gives us the Frobenius $\Phi^{a}_{P/\xi^{p}}:Fil^{a}(M_{P/\xi^{p}})\to M$, and one immediately sees that its linearization is again an isomorphism. 
\begin{claim}
The base change functor $DivCR_{[0, a]}^{\wedge, \phi}(P)\to DivCR_{[0, a]}^{\wedge, \phi}(P/\xi^{p})$ is an equivalence of categories onto its essential image.
\end{claim}
We will construct a quasi-inverse. The construction is similar to \cite[\S 2]{Fal-MF-st}. Let $(M_{0}, Fil^{a}M_{0}, \Phi_{0})\in DivCR_{[0, a]}^{\wedge, \phi}(P/\xi^{p})$. We then construct a projective system $(M_{n}, Fil^{a}M_{n})$ by inductively defining 
\begin{align*}
M_{n+1}= & Fil^{a}M_{n}\otimes_{P, \sigma} P\\
Fil^{a}M_{n+1}= & ker(M_{n+1}\to M_{n}\to M_{n}/Fil^{a}M_{n}).
\end{align*}
We get inductively that 
\begin{align}\label{star}
M_{n+1}/Fil^{a}M_{n+1}=M_{n}/Fil^{a}M_{n}=M/Fil^{a}M\\
 Fil^{a}M_{n+1}/\xi=Fil^{a}M_{n}/\xi \\ M/\xi^{p}=M_{n}/\xi^{p}.
\end{align}
From this one then gets (using that $\xi\in Rad(P)$) that the transition maps are surjective.
We then define $(M, Fil^{a}M):=\varprojlim_{n}(M_{n}, Fil^{a}M_{n})$. One may then easily check that $\mathcal{N}_{P}^{a}M\subset Fil^{a}M$. We also have a morphism 
\begin{equation}\label{frob}
P\otimes_{\sigma} Fil^{a}M= P\otimes_{\sigma}(\varprojlim\limits_{n}Fil^{a}M_{n})\to \varprojlim\limits_{n}(P\otimes_{\sigma}Fil^{a}M_{n})=M. 
\end{equation}
Now assume that $(M, Fil^{a}M, \Phi^{a}_{M})\in DivCR_{[0, a]}^{\wedge, \phi}(P)$. We then claim that applying the construction above to $(M/\xi^{p}, Fil^{a}(M_{P/\xi^{p}}), \Phi^{a}_{M/\xi^{p}})$ gives back the original triple. This follows from condition (5.4).
\\[0.5cm]
Now the morphism $g$ is obtained by base changing along the morphism $\mathcal{O}_{C}^{\flat}/(\xi^{p})\to A_{\mr{cris}}/p$, which is faithfully flat - as $A_{\mr{cris}}/p$ is free as an $\mathcal{O}_{C}^{\flat}/\xi^{p}$-module. Base changing along the canonical projection $A_{\mr{cris}}/p\to \mathcal{O}_{C}^{\flat}/\xi^{p}$ shows that $g$ is a section of the canonical surjective map $A_{\mr{cris}}(R/p)/p\xtwoheadrightarrow{h} (\prism_{B})/(p, \xi^{p})$.\\
The maps $h$ and $g$  induce base change functors, where for a triple $(M, Fil^{a}M, \Phi^{a})\in DivCR_{[0, a]}^{\wedge, \phi}(P/\xi^{p})$, we define $(M_{\mr{cris}}, Fil^{a}M_{\mr{cris}}, \Phi^{a})$, where $M_{\mr{cris}}$ is just the base change along $g$, $Fil^{a}M_{\mr{cris}}=(Fil^{a}M)\otimes A_{\mr{cris}} + Fil^{p}A_{\mr{cris}}\cdot M_{\mr{cris}}$ and the Frobenius is just induced from the base extension of $\Phi^{a}$ to $A_{\mr{cris}}$ (acting by zero on $Fil^{p}A_{\mr{cris}}\cdot M_{\mr{cris}}$).\\
On the other hand, $h^{*}: \MFa(A_{\mr{cris}}(R/p)/p)\to DivCR_{[0, a]}^{\wedge, \phi}(P/\xi^{p})$ is defined just like the base change functor along $P\to P/\xi^{p}$, considered before.
\begin{claim}
The functors $g^{*}$ and $h^{*}$ are quasi-inverse to each other.
\end{claim}
Since $h\circ g=id$, it is easy to see that $h^{*}g^{*}$ is the identity functor. So let $(M, Fil^{a}M, \Phi^{a}_{M})\in \MFa(A_{\mr{cris}}(R/p)/p)$. Note that $\Phi^{a}$ factors over $Fil^{a}M\otimes P/\xi^{p}$ and the Frobenius on $A_{\mr{cris}}(R/p)/p$ factors as $A_{\mr{cris}}(R/p)/p\xrightarrow{h} P/\xi^{p}\xrightarrow{\phi}P/\xi^{p} \xrightarrow{g} A_{\mr{cris}}(R/p)/p$. This shows that $M\cong \phi_{A_{\mr{cris}}(R/p)/p}^{*}Fil^{a}M \cong \phi^{*}_{P/\xi^{p}}Fil^{a}(M_{P/\xi^{p}})\otimes_{g} A_{\mr{cris}}(R/p)/p \cong g^{*}h^{*}M$. One checks that this also fixes $Fil^{a}M$. 
\end{proof}
In the case when $B$ is perfectoid, we can improve the above result.
\begin{prop}\label{perf-mod-p}
Assume that $B$ is perfectoid. Then the functor
\begin{align*}
DivCR_{[0, a]}^{\wedge, \phi, p-tor}(\prism_{B}) & \longleftrightarrow  \MFa^{ p-tor}(\prism_{B/p})
\end{align*}
is an equivalence of categories. Moreover, it restricts to an equivalence
\begin{align*}
DivVect_{[0, a]}^{\wedge, \phi, p-tor}(\prism_{B}) & \longleftrightarrow  \MFa^{p-tor, lf}(\prism_{B/p})
\end{align*}
of the full subcategories, consisting of locally free objects.
\end{prop}
\begin{proof}
To prove the first statement, we need to show that the base-change functor $DivCR_{[0, a]}^{\wedge, \phi}(P)\to DivCR_{[0, a]}^{\wedge, \phi}(P/\xi^{p})$, considerd in the proof of the previous proposition, is now an equivalence of categories. Recall that for any triple $(M_{0}, Fil^{a}M_{0}, \Phi_{0})\in DivCR_{[0, a]}^{\wedge, \phi}(P/\xi^{p})$, we could construct a triple $(M, Fil^{a}M, \Phi^{a}_{M})$ over $P$, where $\Phi_{M}^{a}$ is given by \ref{frob}. We now want to show that this triple always lies in $DivCR_{[0, a]}^{\wedge, \phi}(P)$. First note that $\Phi^{a}_{M}$ is now an isomorphism, since Frobenius is an isomorphism on $P$, as $B$ is perfectoid. We thus only need to show that $M$ is saturated, which for $p$-torsion objects simply means $\xi$-torsion freeness.\\
To show that $M$ is $\xi$-torsion free, we prove the following: For any $n\geq 1$, the map $M_{n+1}[\xi]\to M_{n}[\xi]$, which is induced by the transition map $M_{n+1}\to M_{n}$, is the zero map.\\
Namely, first we apply  $Tor_{1}^{P}( - , P/\xi)$ to the exact sequence 
\begin{center}
$0\to Fil^{a}M_{0}\to M_{0}\to Q_{0}\to 0$,
\end{center}
to obtain the exact sequence
\begin{center}
$0\to Tor_{1}^{P}(Fil^{a}M_{0}, P/\xi)\hookrightarrow (\xi^{p-1})M_{0}\to \ldots  $,
\end{center}
using that $M_{0}$ is flat over $\mathcal{O}_{C}^{\flat}/\xi^{p}$. Using that $\xi^{a}M_{0}\subset Fil^{a}M_{0}$ and $a\leq p-2$, we then get $Tor_{1}^{P}(Fil^{a}M_{0}, P/\xi)\subset \xi Fil^{a}M_{0}$. Base changing with the faithfully flat Frobenius $\phi_{P}$ then gives $Tor_{1}^{P}(M_{1}, P/\xi^{p})\subset \xi^{p}M_{1}$. In particular, this also means $M_{1}[\xi]\subset \xi^{p}M_{1}$. This means that $Tor_{1}^{P}(\xi^{p}M_{1}, P/\xi)\to Tor_{1}^{P}(M_{1}, P/\xi)$ is surjective, which in turn means that the natural map $Tor_{1}^{P}(M_{1}, P/\xi)\to Tor_{1}^{P}(M_{1}/\xi^{p}, P/\xi)$ is the zero map. But the latter map is just the map $Tor_{1}^{P}(M_{1}, P/\xi)\to Tor_{1}^{P}(M_{0}, P/\xi)$, which is induced by the transition map (as $M_{1}/\xi^{p}=\phi^{*}_{P}Fil^{a}M_{0}\otimes P/\xi^{p}=\phi^{*}_{P/\xi^{p}}Fil^{a}M_{0}=M_{0}$). For $n>1$, we now have commutative diagrams 
\begin{center}
$\xymatrix{
Tor_{1}^{P}(Fil^{a}M_{n}, P/\xi)\ar@{^{(}->}[r]\ar[d] & Tor_{1}^{P}(M_{n}, P/\xi)\ar[d]\\
Tor_{1}^{P}(Fil^{a}M_{n-1}, P/\xi) \ar@{^{(}->}[r] & Tor_{1}^{P}(M_{n-1}, P/\xi)
}$
\end{center}
where by induction we may assume that the right vertical arrow is the zero map. Base changing the left vertical map along $\phi_{P}$ yields that $Tor_{1}^{P}(M_{n+1}, P/\xi^{p})\to Tor_{1}^{P}(M_{n}, P/\xi^{p})$ is the zero map. But then, as the $\xi$-torsion is contained in the $\xi^{p}$-torsion, we see that also $M_{n+1}[\xi]\to M_{n}[\xi]$ is the zero map.\\
For the second part of the proposition, assume that $M_{0}$ is finite projective over $P/\xi^{p}$. From the $\xi$-torsion freeness of $M$, it follows that $Tor_{1}^{P/\xi^{p^{n}}}(M/\xi^{p^{n}}, P/\xi^{p})=0$, for all $n\geq 1$. Thus, since $M/\xi^{p}=M_{0}$ is flat over $P/\xi^{p}$, the local flatness criterion yields that $M/\xi^{p^{n}}$ is finite flat over $P/\xi^{p^{n}}$, for all $n\geq 1$. Since $M_{0}$ is even finite projective, (\cite[Tag 0D4B]{stacks}) then  yields that $M=\varprojlim_{n}M/\xi^{p^{n}}$ is finite projective over $P$. Since $\sigma$ is faithfully flat, we also get that $Fil^{a}M$ is finite projective.\\
\end{proof}
\begin{prop}\label{prop-full-faith}
Let $B$ be a quasi-regular semiperfectoid $\mathcal{O}_{C}$-algebra. The functors 
\begin{align*}
DivCR_{[0, a]}^{\wedge, \phi, tor}(\prism_{B}) & \longrightarrow  \MFa^{ \mr{tor}}(\prism_{B/p})\\
DivCR_{[0, a]}^{\wedge, \phi}(\prism_{B}) & \longrightarrow  \MFa(\prism_{B/p})
\end{align*}
induced by base extending along $\prism_{B}\to \prism_{B/p}$ are exact and fully faithful.
\end{prop}
\begin{proof}
Let
\begin{center}
$0\to (M', Fil^{a}M', \Phi^{a}_{M'})\to (M, Fil^{a}M, \Phi^{a}_{M})\to (M'', Fil^{a}M'', \Phi^{a}_{M''})\to 0$
\end{center}
be an exact sequence in $DivCR_{[0, a]}^{\wedge, \phi, tor}(\prism_{B})$. Since $M''$ is a successive extension of $p$-torsion objects, we get $Tor_{1}^{A_{\mr{inf}}}(A_{\mr{cris}}, M'')=0$. Namely, any object $N$, killed by $p$, is flat over $\mathcal{O}_{C}^{\flat}$, from which we get $Tor_{1}^{A_{\mr{inf}}}(A_{\mr{cris}}, N)=0$, using that $A_{\mr{cris}}$ is $p$-torsion free. Therefore, the sequence 
\begin{center}
$0\to M'_{\mr{cris}}\to M_{\mr{cris}}\to M_{\mr{cris}}''\to 0$
\end{center}
is exact. Now, since $\mathcal{N}^{\leq a}\prism_{B}M\subset Fil^{a}M$, we have that $M/Fil^{a}M$ is $\xi^{a}$-torsion. From this, it follows that $M_{\mr{cris}}/Fil^{a}M_{\mr{cris}}=M/Fil^{a}M$, and likewise that $M'_{\mr{cris}}/Fil^{a}M'_{\mr{cris}}=M'/Fil^{a}M'$ and $M''_{\mr{cris}}/Fil^{a}M''_{\mr{cris}}=M''/Fil^{a}M''$. Hence, the sequence
\begin{center}
$0\to M'_{\mr{cris}}/Fil^{a}M'_{\mr{cris}}\to M_{\mr{cris}}/Fil^{a}M_{\mr{cris}}\to M''_{\mr{cris}}/Fil^{a}M''_{\mr{cris}}\to 0$
\end{center}
is exact, which then also yields the exactness of the sequence 
\begin{center}
$0\to Fil^{a}M'_{\mr{cris}}\to Fil^{a}M_{\mr{cris}}\to Fil^{a}M_{\mr{cris}}''\to 0$.
\end{center}

Full faithfulness now follows by d\'evissage from \cref{lem-mod-p}.
\end{proof}
Let $B$ be perfectoid and let $n>0$. An object $(H, Fil^{a}H, \Phi^{a})\in DivCR_{[0, a]}^{\wedge, \phi, \mr{tor}}(\prism_{B})$ is called $A_{\mr{inf}}/p^{n}$-flat, if $H$ is $A_{\mr{inf}}/p^{n}$-flat. It then follows from \cref{quot-flatness} that $H/Fil^{a}H$ is $\Z_p/p^{n}$-flat. An object $(H, Fil^{a}H, \Phi^{a})\in \MFa^{\mr{tor}}(\prism_{B/p})$ is called $A_{\mr{cris}}/p^{n}$-flat, if $H$ is $A_{\mr{cris}}/p^{n}$-flat and $H/Fil^{a}H$ is $\Z_p/p^{n}$-flat. \\

\begin{constr}
The following construction generalizes the one for mod $p$ objects from the proof of \cref{lem-mod-p}. Assume that $B$ is perfectoid and let $(H, Fil^{a}H, \Phi^{a})\in \MFa^{\mr{tor}}(\prism_{B/p})$ be $A_{\mr{cris}}/p^{k}$-flat. We want to show that it comes from an object over $\prism_{B}$. We construct a $\prism_{B}$-linear inverse system as follows:\\
First set $M_{-1}=M_{0}:=H$ and $F^{a}M_{-1}=F^{a}M_{0}=Fil^{a}H$. For $n\geq 1$, we then inductively define $M_{n}:=\phi_{\prism_{B}}^{*}F^{a}M_{n-1}$ and define $F^{a}M_{n}$ to be the kernel of the following map
\[M_{n}=\phi_{\prism_{B}}^{*}F^{a}M_{n-1}\to \phi_{\prism_{B}}^{*}F^{a}M_{0}\xrightarrow[Frob]{rel} \phi_{\prism_{B/p}}^{*}F^{a}M_{0}\xrightarrow{\Phi^{a}}M_{0}\to M_{0}/F^{a}M_{0}, \]
where the second map is just induced from the relative Frobenius of $\prism_{B/p}$ over $\prism_{B}$ and the first map is the Frobenius pullback of the transition map $F^{a}M_{n-1}\to F^{a}M_{0}$. The transition maps $M_{n}\to M_{n-1}$ are defined by
\[M_{1}=\phi_{\prism_{B}}^{*}F^{a}M_{0}\xrightarrow[Frob]{rel} \phi_{\prism_{B/p}}^{*}F^{a}M_{0}\xrightarrow{\Phi^{a}}M_{0}\]
and as the pullback of $F^{a}M_{n-1}\to F^{a}M_{n-2}$ along $\phi_{\prism_{B}}$, for $n>1$. We then set $M:=\varprojlim\limits_{n}M_{n}$ and $F^{a}M:=\varprojlim\limits_{n}F^{a}M_{n}$. We also obtain a map $\Phi^{a}_{M}:\phi_{\prism_{B}}^{*}F^{a}M\to M$.\\
For any $A_{\mr{cris}}/p^{k}$-flat $(H, Fil^{a}H, \Phi^{a})\in \MFa^{\mr{tor}}(\prism_{B/p})$ we also denote the associated triple over $\prism_{B}$ by $\mathcal{F}((H, Fil^{a}H, \Phi^{a}))$.
\begin{claim}
\begin{enumerate}
\item The triple $\mathcal{F}((H, Fil^{a}H, \Phi^{a}))$ is a $A_{\mr{inf}}/p^{k}$-flat object in $DivCR_{[0, a]}^{\wedge, \phi, \mr{tor}}(\prism_{B})$.
\item There is a canonical isomorphism $\mathcal{F}((H, Fil^{a}H, \Phi^{a}))_{\mr{cris}}\xrightarrow{\cong} (H, Fil^{a}H, \Phi^{a})$.
\item The natural map 
\[\mathcal{F}((H, Fil^{a}H, \Phi^{a}))/p^{k-1}\xrightarrow{\cong} \mathcal{F}((H/p^{k-1}, Fil^{a}H/p^{k-1}, \Phi^{a}/p^{k-1}))\]
is an isomorphism.

\end{enumerate}
\end{claim}
The claim will be proved by induction on $k$. We first assume $k=1$ and show that the construction in fact simply gives the equivalence of categories 
\begin{center}
$G:\MFa^{ p-tor}(\prism_{B/p})\longrightarrow DivCR_{[0, a]}^{\wedge, \phi, p-tor}(\prism_{B})$,
\end{center}
which was constructed before. Namely, recall that we have $\prism_{B/p}/p=\prism_{B}/(p, \xi^{p})[\delta_{i}, i\in \N]/(\delta_{i}^{p})$. The functor $G$ was then given by the construction from the proof of \cref{lem-mod-p} applied to the triple $(H, Fil^{a}H, \Phi^{a})/Fil^{p}A_{\mr{cris}}$ over $P/\xi^{p}$. Let $(M'_{n}, F^{a}M_{n}')$ be the inverse system associated to $(H, Fil^{a}H, \Phi^{a})/Fil^{p}A_{\mr{cris}}$. It is then easy to see that there is a morphism $(M_{n}')_{n}\to (M_{n})_{n}$ of inverse systems, such that all the maps $M'_{n}\to M_{n}$ are injective. Namely, recall that $(M'_{n})$ was defined by setting $M'_{0}=H/Fil^{p}A_{\mr{cris}}\cdot H$ and $F^{a}M'_{0}=Fil^{a}H/Fil^{p}A_{\mr{cris}}\cdot H$ and then inductively defining $M'_{n}:=\phi_{P}^{*}F^{a}M'_{n-1}$ and $F^{a}M'_{n}:=ker(M'_{n}\to H/Fil^{a}H)$. Moreover, we have $M_{0}=M'_{0}\otimes_{P/\xi^{p}}P_{\mr{cris}}$ and $F^{a}M_{0}=F^{a}M'_{0}\otimes P_{\mr{cris}}+Fil^{p}A_{\mr{cris}}\cdot M_{0}$. We then inductively obtain commutative diagrams
\begin{center}
$\xymatrix{
M'_{n}=\phi_{P}^{*}F^{a}M'_{n-1}\ar[d]\ar[rr] && \phi_{P}^{*}F^{a}M'_{0}\ar[d]\ar[r]^-{\cong} & M'_{0}\ar[r]\ar[d] & M'_{0}/F^{a}M'_{0} \ar[d]^{=}\\
M_{n}=\phi_{P}^{*}F^{a}M_{n-1}\ar[r] & \phi^{*}_{P}F^{a}M_{0}\ar[r]_-{Frob}^-{rel} & \phi_{P_{\mr{cris}}}^{*}F^{a}M_{0}\ar[r]^-{\Phi^{a}} & M_{0}\ar[r] & M_{0}/F^{a}M_{0}
}$
\end{center}
which then show that the map $M'_{n}\to M_{n}$ takes $F^{a}M'_{n}$ into $F^{a}M_{n}$. Passing to inverse limits, we thus obtain a map $\iota:(M', F^{a}M')\to (M, F^{a}M)$, which is seen to be compatible with the Frobenius maps. Moreover, one sees inductively that $M'_{n}\to M_{n}$ is injective, for all $n\geq 1$. Thus, the map $M'\to M$ is injective as well, by the left exactness of the inverse limit functor. To see that it is also surjective, we note that the relative Frobenius map $\phi_{P}^{*}P_{\mr{cris}}\to P_{\mr{cris}}$ factors as 
\[\phi_{P}^{*}P_{\mr{cris}}=P/\xi^{p^{2}}[\delta_{i}; i\in \N]/(\delta_{i}^{p})\to P/\xi^{p}\to P_{\mr{cris}},\]
where the first map sends all $\delta_{i}$ to zero and is the canonical projection on $P/\xi^{p^{2}}$. Using this, one inductively sees that the image of $M_{n}$ in $M_{n-1}$ is equal to $\iota(M_{n-1}')$. This shows that the map $\iota:(M', F^{a}M')\to (M, F^{a}M)$ is an isomorphism, which proves part (1) and (2) of the claim in the $p$-torsion case. Moreover, we remark that this also shows (since the transition maps of $(M'_{n})_{n}$ are surjective) that the inverse system $(M_{n})_{n}$ satisfies the Mittag-Leffler condition, so that $R^{1}\varprojlim\limits_{n}M_{n}=0$.\\
Assume now that $k>1$. As before, we denote the triple $\mathcal{F}((H, Fil^{a}H, \Phi^{a}))$ by $(M, F^{a}M, \Phi^{a}_{M})$. We first show that $M$ is flat over $\Z_p/p^{k}$. Since $H$ and $H/Fil^{a}H$ (and hence also $Fil^{a}H$) are flat over $\Z_p/p^{k}$, one sees inductively that $M_{n}$ is flat, for all $n\geq 0$. Thus, for all $n$, we have compatible exact sequences 
\[0\to M_{n}^{(1)}\to M_{n}\to M_{n}^{(k-1)}\to 0,\]
where we write $M^{(r)}_{n}:=M_{n}/p^{r}$. We remark that the construction of the inverse system is clearly compatible with reduction modulo $p^{r}$, so that $M_{n}^{(r)}$ is just the inverse system associated to $(H/p^{r}, Fil^{a}H/p^{r}, \Phi^{a}/p^{r})$. The maps on the left $M_{n}^{(1)}\to M_{n}$ are given by multiplication with $p^{k-1}$. In particular, the image is identified with $p^{k-1}M_{n}$. We clearly have $\varprojlim\limits_{n}(p^{k-1}M_{n})=p^{k-1}(\varprojlim\limits_{n}M_{n})$. Applying the inverse limit functor to the compatible exact sequences above and using $R^{1}\varprojlim M_{n}^{(1)}=0$, we get an exact sequence
\[0\to p^{k-1}M\to M\to \varprojlim\limits_{n} M^{(k-1)}_{n}\to 0.\]
This shows that $M/p^{k-1}=\varprojlim\limits_{n} M^{(k)}_{n}=\mathcal{F}((H/p^{k-1}, Fil^{a}H/p^{k-1}, \Phi^{a}/p^{k-1})$. Inductively, we see that $M/p^{k-r}=\mathcal{F}((H/p^{k-r}, Fil^{a}H/p^{k-r}, \Phi^{a}/p^{k-r}))$, for all $r<k$, which proves part (3) of the claim. This in turn also shows that the kernel $p^{k-1}M$ is identified with the map $M/p\xrightarrow{\cdot p^{k-1}} M$, which shows that $M$ is flat over $\Z_p/p^{k}$. The fiberwise flatness criterion then gives that $M$ is flat over $A_{inf}/p^{k}$. Since $(M, F^{a}M, \Phi^{a}M)/p=\mathcal{F}((H/p, Fil^{a}H/p, \Phi^{a}/p))$ is an object in $DivCR_{[0, a]}^{\wedge, \phi, \mr{tor}}(\prism_{B})$, this is enough to see that $(M, F^{a}M, \Phi^{a}M)$ is indeed an object in $DivCR_{[0, a]}^{\wedge, \phi, \mr{tor}}(\prism_{B})$ as well.\\
We still need to prove part (2) of the claim. The transition map $f_{0}:M\to H$ induces a map $M_{\mr{cris}}\to H$, which takes $F^{a}M_{\mr{cris}}=Im(F^{a}M\otimes A_{\mr{cris}})+Fil^{a}A_{\mr{cris}}\cdot M_{\mr{cris}}$ into $Fil^{a}H$, since $f_{0}$ fits into the commutative diagram 
\begin{center}
$\xymatrix{
M\ar[rr]^{f_{0}}\ar[dr] && H\ar[dl] \\
 & M/F^{a}M=H/Fil^{a}H &.
}$
\end{center}
Recall from \cref{sec-constr} that the Frobenius on $F^{a}M_{\mr{cris}}$ is defined as follows: On $Im(F^{a}M\otimes A_{\mr{cris}})$ it is induced from $\Phi^{a}_{M}\otimes \phi$, which is easily seen to be compatible with the Frobenius on $Fil^{a}H$. On $Fil^{a}A_{\mr{cris}}\cdot M_{\mr{cris}}=Fil^{a}A_{\mr{cris}}\otimes_{A_{inf}}M$, it is defined as $\frac{\phi_{A_{\mr{cris}}}}{\phi(\xi)^{a}}\otimes \Phi^{0}_{M}$. But this is also compatible with the Frobenius on $Fil^{a}H$ because $\Phi^{a}_{H}$ satisfies \cref{cond-frob-2}.\\
All in all we get a map $g:(M, F^{a}M, \Phi^{a}_{M})_{\mr{cris}}\to (H, Fil^{a}H, \Phi_{H}^{a})$ in $\MFa(\prism_{B/p})$. Using the compatibility of the construction $\mathcal{F}$ with reduction mod $p$ and also \cref{lem-mod-p}, we get that $g$ is an isomorphism mod $p$. Hence, $g$ is an isomorphism by Nakayama's lemma.
\end{constr}

\begin{prop}\label{perf-equiv-tor}
Let $B$ be a perfectoid ring. Then the functor
\begin{align*}
DivCR_{[0, a]}^{\wedge, \phi, tor}(\prism_{B}) & \longleftrightarrow  \MFa^{ tor}(\prism_{B/p})
\end{align*}
is an equivalence of categories.
\end{prop}
\begin{proof}
The construction above gives a quasi-inverse.
\end{proof}

By d\'evissage, we then also obtain the equivalence for $p$-torsion free objects.
\begin{cor}\label{perf-equiv-free}
Let $B$ be a perfectoid ring. Then the base-change functor
\begin{align*}
DivCR_{[0, a]}^{\wedge, \phi}(\prism_{B}) & \longleftrightarrow  \MFa(\prism_{B/p})
\end{align*}
is an equivalence of categories.
\end{cor}

\subsection{The case of small ramification}
In this section we will briefly consider the case where the ramification degree $e$ of $K$ is small, i.e. we assume that $ea<p-1$. The results here follow immediately from various previous works (in particular using arguments from \cite{CarLiu} and \cite{Fal-MF-st}), but we still record them for the reader's convenience.\\
We first show that in this case every completed prismatic $F$-crystal is projective. This is also remarked in \cite[Rmk. 3.37]{DLMS}. We will give a simple direct proof of this, using an argument from \cite[\S 2]{Fal-MF-st}.
\begin{prop}\label{frob-mod-proj}
Let $\mathcal{X}=Spf(R)$ be small affine. Assume that $ea<p-1$ and let $(\mathcal{M}, \Phi_{\mathcal{M}})\in CR^{\wedge, \phi}_{[0, a]}(\mathcal{X}_{\prism})$, such that $M=\mathcal{M}(\mathfrak{S}(R))$ is $\Z_p/p^{n}$-flat for some $n\geq 0$. Then $M$ is a projective module over $\mathfrak{S}(R)/p^{n}$. In particular $\mathcal{M}$ is a crystal in finite locally free $\mathcal{O}_{\prism}/p^{n}$-modules.
\end{prop}
\begin{proof}
We recall the argument of Faltings'. First note that we may assume that $M$ is $p$-torsion. So we have a module $M$ over $\mathfrak{S}(R)/p$, with an injective map $\Phi_{M}:\phi^{*}M\to M$, such that $Q:=coker(\Phi_{M})$ is $u^{ea}$-torsion. Localizing at a maximal ideal $\mathfrak{m}$ of $A:=\mathfrak{S}(R)/p$ we then further assume that $A$ is a regular local ring. Now, to show that $M$ is projective over $A$, it is then enough to check that $N=Ext^{i}_{A}(M, A)=0$, for all $i>0$. Localizing further at a minimal prime $\mathfrak{p}$ in $supp(N)$, we may assume that $N$ has finite length. Furthermore, since the Frobenius $\phi$ on $A$ is faithfully flat, we have $Ext^{i}(\phi^{*}M, A)\cong \phi^{*}Ext^{i}(M, A)$. Taking the long exact sequence for $Ext$ and using that $Q$ is $u^{ea}$-torsion, we get $u^{ea}N=u^{ea}\phi^{*}N$. On the other hand, using that $u$ is a regular element in $A$ and the argument given in \cite[\S 2]{Fal-MF-st}, we get that for the length we have $l(u^{ae}\phi^{*}N)\geq l(N)(p-ae)^{ht(\mathfrak{p})}$. But then we get $l(N)\geq l(u^{ea}N)=l(u^{ea}\phi^{*}N)\geq l(N)(p-ae)^{ht(\mathfrak{p})}$. This forces $N=0$.
\end{proof}
\begin{cor}
Let $\mathcal{M}\in \MFa^{\mr{tor-free}}(\mathcal{X})$. Then $\mathcal{M}$ is a crystal in finite locally free $\mathcal{O}_{\mathcal{X}_{\mr{cris}}}$-modules.
\end{cor}
We will also briefly discuss the \'etale realization functor for torsion objects. Denote by $\Z_p-Mod^{tor}(\mathcal{X})$, the category of $p$-torsion $\hat{\Z}_p$-modules on $(\mathcal{X}_{K})_{pro\acute{e}t}$.
\begin{prop}
Assume that $ea<p-1$. Let $\mathcal{X}$ be a smooth $p$-adic formal scheme over $\mathcal{O}_{K}$. Then the \'etale realization functor 
\begin{center}
$T_{\acute{e}t}:CR^{\wedge, \phi, tor}_{[0, a]}(\mathcal{X}_{\prism})\to \Z_p-Mod^{tor}(\mathcal{X})$
\end{center}
is fully faithful.
\end{prop}
\begin{proof}
We may assume that $\mathcal{X}=Spf(R)$ is small affine. By d\'evissage, it is moreover enough to prove the statement for $p$-torsion objects. So let $(\mathcal{M}, \Phi_{\mathcal{M}}), (\mathcal{N}, \Phi_{\mathcal{N}})\in CR^{\wedge, \phi, p-tor}_{[0, a]}(\mathcal{X}_{\prism})$ and assume that there is a map $f:\mathcal{M}[\frac{1}{\mathcal{I}_{\prism}}]\to \mathcal{N}[\frac{1}{\mathcal{I}_{\prism}}]$, compatible with the Frobenius maps. We need to show that $f(\mathcal{M})\subset \mathcal{N}$. For this it is enough to consider the evaluations at $\mathfrak{S}(R)$. Write $M=\mathcal{M}(\mathfrak{S}(R))$ and $N=\mathcal{N}(\mathfrak{S}(R))$. Now there exists an $r\geq 0$, such that $u^{r}f(M)\subset N$. Assume that $f(M)\not\subset N$ and let $r$ be minimal. Then $r>0$ and there exists an $x\in f(M)$, such that $u^{r-1}x\not\in N$. In particular $u^{r}x\not\in  uN$, from which one gets (using that $u^{ea}N\subset Im(\Phi_{N})$) 
\begin{center}
$\Phi_{N}(u^{r}x)\not\in u^{ea+1}N$
\end{center}
(see the proof of \cite[Lem. 3.2.4]{CarLiu}).\\
On the other hand, writing $x=f(y)$ for some $y\in M$, we have 
\begin{center}
$
\Phi_{N[\frac{1}{u}]}(u^{r}x)  =  u^{pr}f(\Phi_{M[\frac{1}{u}]}(y))= u^{(p-1)r}u^{r}f(\Phi_{M[\frac{1}{u}]}(y))\in u^{ea+1}N$
 \end{center}
 since $ea+1\leq p-1$. This gives a contradiction. Hence, we must have $f(M)\subset N$.
\end{proof}

\addcontentsline{toc}{section}{References}
\bibliographystyle{plain}
\bibliography{bib}
Institut f\"ur Mathematik, 
    Goethe-Universit\"at Frankfurt, 
    60325 Frankfurt am Main, 
    Germany\\
    wuerthen@math.uni-frankfurt.de
\end{document}